\documentclass[twoside, a4paper,10pt]{amsart}

\usepackage{amsmath}  
\usepackage{amsfonts}
\usepackage{amsthm}  
\usepackage{amssymb}
\usepackage{amscd}
\usepackage{enumerate}
\usepackage[all]{xy}
\SelectTips{cm}{}
\SilentMatrices

\usepackage[pdfauthor={Steffen Sagave and Christian Schlichtkrull}, pdftitle={Diagram spaces and symmetric spectra}]{hyperref} 

\numberwithin{equation}{section}
\theoremstyle{plain}
\newtheorem{lemma}[subsection]{Lemma}
\newtheorem{theorem}[subsection]{Theorem}

\newtheorem{corollary}[subsection]{Corollary}
\newtheorem{proposition}[subsection]{Proposition}
\theoremstyle{definition}

\newtheorem{definition}[subsection]{Definition}
\newtheorem{example}[subsection]{Example}
\newtheorem{remark}[subsection]{Remark}
\newtheorem{coarse-assumptions}[subsection]{Coarse Assumptions}
\newtheorem{fine-assumptions}[subsection]{Fine Assumptions}

\newcommand{\mC}{{\mathbb C}}

\newcommand{\mN}{{\mathbb N}}
\newcommand{\mS}{{\mathbb S}}
\newcommand{\mZ}{{\mathbb Z}}

\newcommand{\cA}{{\mathcal A}}
\newcommand{\cB}{{\mathcal B}}
\newcommand{\cC}{{\mathcal C}}
\newcommand{\cD}{{\mathcal D}}
\newcommand{\cE}{{\mathcal E}}
\newcommand{\cF}{{\mathcal F}}
\newcommand{\cI}{{\mathcal I}}
\newcommand{\cJ}{{\mathcal J}}

\newcommand{\cK}{{\mathcal K}}
\newcommand{\cM}{{\mathcal M}}
\newcommand{\cN}{{\mathcal N}}

\newcommand{\cP}{{\mathcal P}}
\newcommand{\cS}{{\mathcal S}}

\DeclareMathOperator{\id}{id}
\DeclareMathOperator{\Hom}{Hom}

\DeclareMathOperator{\Map}{Map}

\DeclareMathOperator{\Ev}{Ev}

\DeclareMathOperator{\image}{im}
\DeclareMathOperator{\const}{const}
\DeclareMathOperator{\colim}{colim}
\DeclareMathOperator{\hocolim}{hocolim}

\newcommand{\GL}{\mathrm{GL}}

\DeclareMathOperator{\Sing}{Sing}
\DeclareMathOperator{\pt}{\ast}
\DeclareMathOperator{\cof}{-cof}

\DeclareMathOperator{\cell}{-cell}
\DeclareMathOperator{\sgn}{\mathrm{sgn}}
\DeclareMathOperator{\Sgn}{\mathrm{Sgn}}

\DeclareMathOperator{\concat}{\sqcup}

\newcommand{\eins}{\mathbf{1}}
\newcommand{\tensor}{\otimes}

\newcommand{\ot}{\leftarrow}
\newcommand{\iso}{\cong}

\newcommand{\op}{{\mathrm{op}}}
\newcommand{\sm}{\wedge}

\newcommand{\Spsym}{{\mathrm{\textrm{Sp}^{\Sigma}}}}
\newcommand{\bld}[1]{{\mathbf{#1}}}

\newcommand{\GLoneI}{\GL^{\cI}_1}
\newcommand{\OmegaJ}{\Omega^{\cJ}}
\newcommand{\GLoneJ}{\GL^{\cJ}_1}
\newcommand{\OmegaK}{\Omega^{\cK}}

\newcommand{\Top}{\mathrm{Top}}

\newcommand{\arxivlink}[1]{\href{http://arxiv.org/abs/#1}{\texttt{arXiv:#1}}}

\newcommand{\hocolimwithlimits}{\operatornamewithlimits{hocolim}}
\newcommand{\colimwithlimits}{\operatornamewithlimits{colim}}
\newcommand{\I}{\mathcal I}

\newcommand{\xl}{\xleftarrow}
\newcommand{\xr}{\xrightarrow}

\begin{document} \title{Diagram spaces and symmetric spectra} 

\author{Steffen Sagave} \address{Steffen Sagave, Mathematical
Institute, University of Bonn, Endenicher Allee 60, 53115 Bonn,
Germany} \email{sagave@math.uni-bonn.de}

\author{Christian Schlichtkrull} \address{Christian Schlichtkrull,
    Department of Mathematics, University of Bergen, Johannes
    Brunsgate 12, 5008 Bergen, Norway} \email{christian.schlichtkrull@math.uib.no}

\date{\today} 

\begin{abstract}
  We present a general homotopical analysis of structured diagram
  spaces and discuss the relation to symmetric spectra. The main
  motivating examples are the \emph{$\cI$-spaces}, which are diagrams
  indexed by finite sets and injections, and \emph{$\cJ$-spaces},
  which are diagrams indexed by Quillen's localization construction
  $\Sigma^{-1}\Sigma$ on the category $\Sigma$ of finite sets and
  bijections.

  We show that the category of $\cI$-spaces provides a convenient
  model for the homotopy category of spaces in which every
  $E_{\infty}$ space can be rectified to a strictly commutative
  monoid. Similarly, the commutative monoids in the category of
  $\cJ$-spaces model graded $E_{\infty}$ spaces.

  Using the theory of $\cJ$-spaces we introduce the \emph{graded
    units} of a symmetric ring spectrum. The graded units detect
  periodicity phenomena in stable homotopy and we show how this can be
  applied to the theory of topological logarithmic structures.
\end{abstract}

\maketitle

\section{Introduction}
In this paper we present a general approach to structured diagram spaces and the relation to symmetric spectra. We begin by discussing the motivating examples of $\cI$- and $\cJ$-spaces. Here the underlying category of spaces $\cS$ may be interpreted either as the category of (compactly generated weak Hausdorff) topological spaces or the category of simplicial sets.

\subsection{$\cI$-spaces and $E_{\infty}$ spaces}
Let $\cI$ be the category whose objects are the finite sets 
$\bld n=\{1,\dots,n\}$ (including the empty set $\bld 0$) and whose morphisms are the injective maps. The usual ordered concatenation of ordered sets makes $\cI$ a symmetric monoidal category. An $\cI$-space is a functor $X\colon \cI\to\cS$ and we write $\cS^{\cI}$ for the category of $\cI$-spaces. As it is generally the case for a diagram category indexed by a small symmetric monoidal category, $\cS^{\cI}$ inherits a symmetric monoidal structure from 
$\cI$. A map of $\cI$-spaces $X\to Y$ is said to be an 
\emph{$\cI$-equivalence} if the induced map of homotopy colimits 
$X_{h\cI}\to Y_{h\cI}$ is a weak homotopy equivalence. The 
$\cI$-equivalences are the weak equivalences in a model structure on 
$\cS^{\cI}$, the \emph{projective $\cI$-model structure}, with the property that the usual adjunction  
\[
\xymatrix{
\colim_{\cI} \colon \cS^{\cI} \ar@<0.5ex>[r] &
\cS \thinspace \colon   \! \! \const_{\cI} \ar@<0.5ex>[l] 
}
\] 
defines a Quillen equivalence with respect to the standard model structure on $\cS$. Thus, the homotopy category of $\cS^{\cI}$ is equivalent to the usual homotopy category of spaces. We think of $X_{h\cI}$ as the underlying space of the $\cI$-space $X$. 

The main advantage of $\cS^{\cI}$ compared to $\cS$ is that it provides a more flexible setting for working with structured objects, in particular $E_{\infty}$ structures. Let $\cC\cS^{\cI}$ be the category of commutative 
$\cI$-space monoids, that is, commutative monoids in $\cS^{\cI}$. Utilizing an idea of Jeff Smith first implemented in the category of symmetric spectra \cite{MMSS} we show that there is a \emph{positive projective $\cI$-model structure} on $\cS^{\cI}$ that lifts to a model structure on $\cC\cS^{\cI}$. In the following theorem we consider an $E_{\infty}$ operad $\cD$ with associated monad $\mathbb D$ and $\cS[\mathbb D]$ denotes the category of $\mathbb D$-algebras in $\cS$. 

\begin{theorem}\label{thm:intro-I-theorem}
Let $\cD$ be an $E_{\infty}$ operad. Then there is a chain of Quillen equivalences $\cC\cS^{\cI}\simeq\cS[\mathbb D]$ relating the positive projective $\cI$-model structure on $\cC\cS^{\cI}$ to the standard model structure on 
$\cS[\mathbb D]$ lifted from $\cS$. 
\end{theorem}

In the topological setting this result applies for instance to the little $\infty$-cubes operad and the linear isometries operad. Specializing to the Barratt-Eccles operad $\cE$ we can give an explicit description of the induced equivalence of homotopy categories: A commutative $\cI$-space monoid $A$ is mapped to the homotopy colimit $A_{h\I}$ with its canonical action of the monad $\mathbb E$ associated to $\cE$, see  
\cite[Proposition~6.5]{Schlichtkrull-Thom_symm}.    

As a consequence of the theorem any $E_{\infty}$ homotopy type can be represented by a commutative $\cI$-space monoid. It is well-known that in general such a rectification cannot be carried out in $\cS$ since a grouplike commutative monoid in $\cS$ is equivalent to a product of 
Eilenberg-Mac\ Lane spaces.

\begin{example}\label{ex:X-bullet-example}
A based space $X$ gives rise to a commutative $\cI$-space monoid $X^{\bullet}\colon \bld n\mapsto X^n$ and it is proved in \cite{Schlichtkrull-infinite} that the  underlying space $X^{\bullet}_{h\cI}$ is equivalent to the Barratt-Eccles construction $\Gamma^+(X)$  (in the topological setting we need the extra assumption that $X$ be well-based). Thus, for connected $X$ it follows from the Barratt-Priddy-Quillen Theorem that $X^{\bullet}$ represents the infinite loop space $Q(X)$.
\end{example}

\begin{example}\label{ex:BO-example}
Writing $O(n)$ for the orthogonal groups we have the commutative $\cI$-space monoid $BO\colon\bld n\mapsto BO(n)$ in the topological setting. In this case $BO_{h\cI}$ is equivalent to the classifying space $BO(\infty)$ for stable vector bundles.
Thus, we have represented the $E_{\infty}$ space $BO(\infty)$ as a commutative $\cI$-space monoid.
\end{example}

\begin{example}\label{ex:BGL(R)-example}
Let $R$ be an ordinary ring and consider the commutative $\cI$-space monoid 
$B\GL(R)\colon \bld n\mapsto B\GL_n(R)$. In this case the underlying space $B\GL(R)_{h\cI}$ is equivalent to Quillen's plus construction 
$B\GL_{\infty}(R)^+$. Thus, we have represented the higher algebraic 
$K$-theory of $R$ by a commutative $\I$-space monoid. 
\end{example}

For a commutative $\cI$-space monoid $A$, we know from \cite[Theorem 4.1]{Schwede_S-algebras} that the projective model structure on $\cS^{\cI}$ lifts to a monoidal model structure on the category
of $A$-modules with the symmetric monoidal product inherited from 
$\cS^{\cI}$. This symmetric monoidal structure on $A$-modules is one of the benefits of working with a strictly commutative monoid and is hard to come by in the operadic context.

The category of $\cI$-spaces is closely related to the category of symmetric spectra $\Spsym$ and in particular we have an adjunction  
$\mS^{\cI}[-] \colon\cS^{\cI} \rightleftarrows \Spsym \colon\Omega^{\cI}$,
where $\mathbb S^{\cI}[-]$ is an $\cI$-space version of the (unbased) suspension spectrum functor and $\Omega^{\cI}$ takes a symmetric spectrum $E$ to the $\I$-space $\Omega^{\cI}(E)\colon\bld n\mapsto \Omega^n(E_n)$. 
As demonstrated in \cite{Schlichtkrull-units} this can be used to give an $\cI$-space model $\GLoneI(R)$ for the units of a symmetric ring spectrum $R$. In this paper we analyze the homotopical properties of these constructions and we lay the foundation for the applications in \cite{Rognes-TLS} to the theory of topological logarithmic structures. Commutative $\cI$-space monoids as models for $E_{\infty}$ spaces have also proved useful for the study of algebraic K-theory of
structured ring spectra~\cite{Schlichtkrull-units}, Thom
spectra~\cite{Schlichtkrull-Thom_symm}, and the topological Hochschild
homology of Thom spectra~\cite{Blumberg-C-S_THH-Thom,Schlichtkrull-higher-THH}. In particular, the present paper provides results which were referred to (and, in some cases, also used) in
\cite{Blumberg-C-S_THH-Thom,Rognes-TLS,
Schlichtkrull-Thom_symm,Schlichtkrull-higher-THH}.  
In~\cite{Sagave-S_group-completion}, the authors examine group completion of commutative $\cI$-space monoids, express it in terms of model structures, relate it to the notion of units, and explain the
connection to $\Gamma$-spaces. 

It is in order to relate this work to the framework for structured spectra developed by Elmendorf, Kriz, Mandell, and May \cite{EKMM}. In this framework the analogues of symmetric spectra are the so-called \emph{$S$-modules}. Just as one may view $\cI$-spaces as space level analogues of symmetric spectra, there is a space level analogue of $S$-modules known as \emph{$*$-modules}, see 
\cite[Section 4]{Blumberg-C-S_THH-Thom}. It is proved by Lind \cite{Lind-diagram} that the $(\mathbb S^{\cI}[-],\Omega^{\cI})$-adjunction discussed above has an $S$-module analogue which on the level of homotopy categories agrees with the latter up to natural isomorphism. Furthermore, Lind goes on to establish a Quillen equivalence between the category $\cC\cS^{\cI}$ and the category of commutative monoids in $*$-modules which in turn can be identified with algebras for the linear isometries operad. This gives a way to relate commutative $\cI$-space monoids with $E_{\infty}$ spaces which is quite different from the approach taken in this paper.   

\subsection{$\cJ$-spaces and graded units}
Whereas the $\cI$-space monoid $\GLoneI(R)$ is a useful model for the units of a connective symmetric ring spectrum, this construction is of limited value for symmetric ring spectra that are not connective: If $R\to R'$ is a map of positive fibrant symmetric ring spectra which induces an isomorphism of homotopy groups in non-negative degrees, then the induced map $\GLoneI(R)\to \GLoneI(R')$ is an 
$\cI$-equivalence. Consequently, the $\cI$-space units do not distinguish between a symmetric ring spectrum and its connective cover and 
cannot detect periodicity phenomena in stable homotopy. This defect of the 
$\cI$-space units is shared by any other previous definitions of the units of a structured ring spectrum. 

What we seek instead is a notion of the units which takes into account all the possible stable $R$-module equivalences $\Sigma^{n_2}R\to \Sigma^{n_1}R$ between suspended copies of $R$. Motivated by the definition of 
$\Omega^{\cI}(R)$ it is natural to try organizing the collection of spaces 
$\Omega^{n_2}(R_{n_1})$ into a $\cJ$-diagram for a suitable small symmetric monoidal category $\cJ$. One of the main features of the paper is to show that Quillen's localization construction $\Sigma^{-1}\Sigma$ on the category of finite sets and bijections $\Sigma$ is a natural choice for such a category $\cJ$. 
The objects of $\Sigma^{-1}\Sigma$ are pairs $(\bld n_1,\bld n_2)$ of objects in $\cI$ and the morphisms $(\bld m_1,\bld m_2)\to(\bld n_1,\bld n_2)$ can be described explicitly as triples $(\alpha_1,\alpha_2,\rho)$ given by a pair of morphisms $\alpha_1\colon\bld m_1\to\bld n_1$ and $\alpha_2\colon\bld m_2\to\bld n_2$ in $\cI$, together with a bijection $\rho$ identifying the complements of the images of these morphisms. The point is that the extra connecting tissue provided by $\rho$ is exactly the data needed to make the correspondence $(\bld n_1,\bld n_2)\mapsto \Omega^{n_2}(R_{n_1})$ functorial. 

With this choice of $\cJ$ we define a \emph{$\cJ$-space} to be a functor $X\colon\cJ\to\cS$ and write $\cS^{\cJ}$ for the category of $\cJ$-spaces. A map of $\cJ$-spaces $X\to Y$ is said to be a \emph{$\cJ$-equivalence} if the induced map of homotopy colimits $X_{h\cJ}\to Y_{h\cJ}$ is a weak homotopy equivalence. We show that the $\cJ$-equivalences are the weak equivalences in a model structure on $\cS^{\cJ}$, the \emph{projective $\cJ$-model structure}, with the property that there is a chain of Quillen equivalences 
$\cS^{\cJ}\simeq \cS/B\cJ$ relating it to the standard model structure on the category $\cS/B\cJ$ of spaces over $B\cJ$. There also is a \emph{positive projective $\cJ$-model structure} on $\cS^{\cJ}$ which lifts to a model structure on the category $\cC\cS^{\cJ}$ of commutative $\cJ$-space monoids (that is, commutative monoids in $\cS^{\cJ}$). In the next theorem 
$\mathbb E$ again denotes the monad associated to the Barratt-Eccles operad and $\cS[\mathbb E]/B\cJ$ is the category of $\mathbb E$-algebras over the $\mathbb E$-algebra $B\cJ$. 

\begin{theorem}\label{thm:intro-J-theorem}
There is a chain of Quillen equivalences $\cC\cS^{\cJ}\simeq \cS[\mathbb E]/B\cJ$ relating the positive projective model structure on $\cC\cS^{\cJ}$ to the standard model structure on $\cS[\mathbb E]/B\cJ$ lifted from $\cS$. 
\end{theorem}

By work of Barratt, Priddy, and Quillen, it is known that $B\cJ$ is equivalent to $Q(S^0)$, so the above theorem allows us to interpret $\cC\cS^{\cJ}$ as a model for the category of $E_{\infty}$ spaces over $Q(S^0)$. This fits well with the general point of view that in a spectral context the sphere spectrum $\mathbb S$ takes the role played by the ring of integers $\mathbb Z$ in the traditional algebraic context. Indeed, in algebra a graded monoid is logically the same as a monoid $A$ together with a monoid homomorphism 
$A\to \mathbb Z$ to the underlying additive group $(\mathbb Z,+,0)$. In topology it is customary to think of $Q(S^0)$ as the ``additive group'' of 
$\mathbb S$ and hence we can think of commutative $\cJ$-space monoids as representing graded commutative spaces.

The category of $\cJ$-spaces is related to the category of symmetric spectra by a pair of  monoidal adjoint functors 
$\mS^{\cJ}[-] \colon\cS^{\cJ} \rightleftarrows \Spsym \colon\Omega^{\cJ}$.
Given a symmetric ring spectrum $R$ we define the graded units $\GLoneJ(R)$ as a suitable sub $\cJ$-space monoid of 
$\Omega^{\cJ}(R)$. In general, we define for any positive fibrant 
$\cJ$\nobreakdash-space monoid $A$ a graded signed monoid $\pi_0(A)$ of ``components'' (see Definition \ref{def:graded-signed-monoid}). The definition of $\pi_0(A)$ is motivated by the next theorem where $\pi_*(R)^{\times}$ denotes the graded group of multiplicative units in the graded ring of homotopy groups $\pi_*(R)$. 

\begin{theorem}
Let $R$ be a positive fibrant symmetric ring spectrum. Then there is a natural isomorphism of graded signed monoids $\pi_0(\Omega^{\cJ}(R))\simeq \pi_*(R)$ which restricts to an isomorphism of graded signed groups 
$\pi_0(\GLoneJ(R))\simeq \pi_*(R)^{\times}$.
\end{theorem}

Hence all units of the graded ring $\pi_*(R)$ are incorporated in the graded units $\GLoneJ(R)$ while the corresponding notions of units using $\cI$-space monoids or $E_{\infty}$ spaces only detect $\pi_0(R)^{\times}$. We illustrate the use of these concepts by applying them to the theory of topological logarithmic structures developed by Rognes \cite{Rognes-TLS}. The question of how to associate spectra with grouplike commutative $\cJ$-space monoids like $\GLoneJ(R)$ and how to form group completions in this setting is studied by the first author in~\cite{Sagave_spectra-of-units}.

\subsection{Well-structured index categories}
In order to treat the theories of $\cI$- and $\cJ$-spaces in a common framework and to express our results in the natural level of generality, we introduce the notion of a \emph{well-structured index category}. Such a category $\cK$ is a small symmetric monoidal category equipped with a degree functor to the ordered set of natural numbers and satisfying a short list of axioms, see Definition~\ref{def:well-structured-index}. The axioms guarantee that the associated category of $\cK$-spaces $\cS^{\cK}$ inherits a well-behaved \emph{projective $\cK$-model structure} whose weak equivalences are the 
$\cK$-equivalences, that is, the maps that induce weak homotopy equivalences of the associated homotopy colimits. Here ``well-behaved'' means that the projective $\cK$-model structure is cofibrantly generated, proper, monoidal, and satisfies the monoid axiom.   
Assuming that the full subcategory $\cK_+$ of $\cK$ whose  objects have positive degree is homotopy cofinal, there also is a \emph{positive projective 
$\cK$-model structure} on $\cS^{\cK}$. The latter model structure lifts to the category of commutative $\cK$-space monoids $\cC\cS^{\cK}$ provided that 
for each pair of objects $\bld k$ and $\bld l$ in $\cK_+$, the action of the symmetric group $\Sigma_n$ on the $n$-fold iterated monoidal product $\bld k^{\sqcup n}$ induces a free right action on the set of connected components of the comma category $(\bld k^{\sqcup n}\sqcup-\downarrow \bld l)$. We express this by saying that the discrete subcategory $O\cK_+$ of identity morphisms with positive degree defines a \emph{very well-structured relative index category $(\cK,O\cK_+)$}. 
The categories $\cI$ and $\cJ$ are well-structured index categories and the next theorem generalizes 
Theorem~\ref{thm:intro-I-theorem} and Theorem~\ref{thm:intro-J-theorem}. Here $\mathbb E$ again denotes the monad associated to the Barratt-Eccles operad. If the symmetric monoidal category $\cK$ is permutative (that is, symmetric strict monoidal), then $B\cK$ is an $\mathbb E$-algebra and we write $\cS[\mathbb E]/B\cK$ for the category of $\mathbb E$-algebras over $B\cK$.
\begin{theorem}\label{thm:intro-K-theorem}
Let $\cK$ be a well-structured index category. 
\begin{enumerate}[(i)]
\item
There is a chain of Quillen equivalences $\cS^{\cK}\simeq \cS/B\cK$ relating the projective $\cK$-model structure on $\cS^{\cK}$ to the standard model structure on $\cS/B\cK$.
\item
Suppose that the underlying symmetric monoidal category of $\cK$ is permutative, that $\cK_+$ is homotopy cofinal in $\cK$, and that 
$(\cK,O\cK_+)$ is very well-structured. Then there is a chain of Quillen equivalences 
$\cC\cS^{\cK}\simeq \cS[\mathbb E]/B\cK$ relating the positive projective 
$\cK$-model structure on  $\cC\cS^{\cK}$ to the standard model structure on 
$\cS[\mathbb E]/B\cK$. 
\end{enumerate}
\end{theorem}

There are many examples of well-structured index categories and in each case one may view the above theorem as a kind of rectification principle. We give a further example related to algebraic $K$-theory. Let $R$ be an ordinary algebraic ring with invariant basis number (for instance any commutative ring with more than one element), let $\cF_R$ be the isomorphism category of $R$-modules of the standard form $R^n$, and let $\cK_R$ be Quillen's localization construction $\cF_R^{-1}\cF_R$ on this category. The classifying space $B\cK_R$ represents the ``free'' algebraic $K$-theory functor which is equivalent to the ordinary algebraic $K$-theory functor in positive degrees. We refer the reader to  Example~\ref{ex:K-R-example} for a full discussion and the verification that $\cK_R$ defines a well-structured index category such that the conditions in 
Theorem~\ref{thm:intro-K-theorem} are satisfied. Applied in this case the latter theorem provides a rectification of $K$-theoretical data:
Each homotopy class of maps $X\to B\cK_R$ can be represented by a unique $\cK_R$-space homotopy type, and each homotopy class of $E_{\infty}$ maps can be represented by a unique commutative $\cK_R$-space monoid homotopy type.

For any serious work with diagram spaces it is important to understand
the homotopical properties of the monoidal structure. One may ensure
that a $\cK$-space is homotopically well-behaved with respect to the
monoidal product by imposing suitable cofibrancy conditions and in
practice it often happens that there are several useful model
structures on the same diagram category $\cS^{\cK}$ providing such
notions of cofibrancy. This is analogous to the situation for
symmetric spectra where the stable projective model structure
\cite{HSS,MMSS} is accompanied by the stable flat (or $S$-) model
structure \cite{HSS, Shipley-convenient} and the corresponding
positive variants \cite{MMSS, Shipley-convenient}; see also
\cite{Schwede-SymSp}. In this paper we set up a general framework for
analyzing model structures on diagram categories by introducing the
notion of a \emph{well-structured relative index category} $(\cK,\cA)$
given by a small symmetric monoidal category $\cK$ together with a
subcategory of automorphisms $\cA$ satisfying a suitable list of
axioms, see Definition~\ref{def:well-structured-relative-index}.
Letting $\cA$ be the discrete category of identity morphisms in $\cK$
we recover the notion of a well-structured index category and the
corresponding projective $\cK$-model structure. Similarly, if we let
$\cA$ be the category of identity morphisms of positive degree we get
the positive projective $\cK$-model structure. If in the case of $\cI$
and $\cJ$ we let $\cA$ be the subcategory of all (positive)
automorphisms, we get the analogues of the stable (positive) flat
model structure on symmetric spectra. This diversity might be
confusing at first sight, but there are useful features of each of
these model structures and no single one has all the desirable
properties simultaneously.

\subsection{Organization}
The paper is roughly divided into two parts: In Sections 2--4 we present the theory of $\cI$- and $\cJ$-spaces in detail and we show how the graded 
units $\GLoneJ(R)$ can be used in connection with the theory of topological logarithmic structures. Many of the proofs in this first part of the paper are deferred to the remaining Sections 5--14 where we develop the homotopical properties in the general framework of diagram spaces indexed by a well-structured relative index category.  The technical results on operad algebras needed for the paper are established in Appendix \ref{app:analysis-operad-algebras}. It is hoped that by first presenting the applications to $\cI$- and $\cJ$-spaces, the reader will be motivated to go through the more general material in the second part of the paper. The specific organization of the material should be clear from the table of contents.

\tableofcontents

\subsection{Acknowledgments}
This project benefited from support through the YFF grant \emph{Brave
  new rings} held by John Rognes and the \emph{Topology in Norway}
project, both funded by the Research Council of Norway, the SFB 478
at M\"unster and the HCM at Bonn. The authors like to thank Ruth
Joachimi, Thomas Kragh, Wolfgang L\"uck, John Rognes, Stefan Schwede,
Mirjam Solberg, and Rainer Vogt for helpful conversations related to
this project. Moreover, the suggestions made by an anonymous referee also helped to improve the manuscript.

\section{Preliminaries on diagram spaces}\label{sec:preliminaries}
We work simultaneously in a topological and a simplicial setting and write 
$\cS$ for our category of spaces. Thus, unless stated otherwise, $\cS$ denotes both the category of compactly generated weak Hausdorff topological spaces and the category of simplicial sets. The corresponding based categories are denoted by $\cS_*$.

\begin{definition} 
Given a small category $\cK$, a \emph{$\cK$-space} is a functor 
$X\colon\cK\to\cS$. We write $\cS^{\cK}$ for the category of $\cK$-spaces
with morphisms the natural transformations.
\end{definition}

The next lemma recalls the basic formal properties of the category of 
$\cK$-spaces. 
\begin{lemma}\label{lem:K-spaces-S-category}
  The category $\cS^{\cK}$ is bicomplete with limits and colimits constructed levelwise. Furthermore, $\cS^{\cK}$ is
  enriched, tensored, and cotensored over
  $\cS$. For a $\cK$-space $X$ and a space $T$ (in $\cS$), the
  tensor $X \times T$ and cotensor $X^T$ are the $\cK$-spaces defined by
\[
(X\times T)(\bld{k}) =X(\bld{k})\times T\quad\textrm{ and }\quad 
X^T(\bld k)=\Map(T,X(\bld k)).
\]  
  The space of maps from $X$
  to $Y$ is the end
  \[ \Map(X,Y) = \int_{\bld{k}\in\cK} \Map(X(\bld{k}),
  Y(\bld{k})).
  \eqno\qed
  \]
\end{lemma}

Suppose now that $(\cK,\sqcup,\bld 0)$ is a symmetric monoidal category with monoidal structure $\sqcup$ and unit $\bld 0$. 
The left Kan extension along $\concat \colon \cK \times \cK \to \cK$
defines a symmetric monoidal product $\boxtimes$ on $\cK$-spaces in the usual way: Given a pair of $\cK$-spaces $X$ and $Y$,
\[ (X \boxtimes Y)(\bld{n}) = \colim_{\bld{k}\concat \bld{l} \to
  \bld{n}} X(\bld{k}) \times Y(\bld{l})\] with the colimit taken over
the comma category $(\concat \downarrow \cK)$. The monoidal unit is the levelwise discrete $\cK$-space 
$\bld 1_{\cK}=\cK(\bld{0},-)$.

\begin{definition} 
  A \emph{(commutative) $\cK$-space monoid} is a (commutative) monoid in
  the symmetric monoidal category of $\cK$-spaces $(\cS^{\cK},\boxtimes,
  \bld 1_{\cK})$. We write 
  $\cC\cS^{\cK}$ for the category of commutative $\cK$-space monoids.
\end{definition}

By the universal property of the left Kan extension, the data defining a monoid structure on a $\cK$-space $A$ amounts to a map $*\to A(\bld 0)$ and a map of $(\cK\times\cK)$-spaces $A(\bld k)\times A(\bld l)\to A(\bld k\sqcup\bld l)$, subject to the usual associativity and unitality conditions. The commutativity condition amounts to the commutativity of the diagram of $(\cK\times\cK)$-spaces
\[
\xymatrix@-1pc{
A(\bld k)\times A(\bld l) \ar[r] \ar[d]& A(\bld k\sqcup \bld l)\ar[d]\\
A(\bld l)\times A(\bld k) \ar[r] & A(\bld l\sqcup \bld k)
}
\] 
where the left hand side flips the factors and the right hand map is induced by the symmetry isomorphism of $\cK$. An equivalent way of expressing this is to say that $A$ defines a (lax) symmetric monoidal functor from $(\cK,\sqcup,\bld 0)$ to $(\cS,\times,*)$.
 
The symmetric monoidal structure on $\cS^{\cK}$ is closed in the sense that there is an internal $\Hom$ functor
\[\Hom \colon \left(\cS^{\cK}\right)^{\op} \times
\cS^{\cK} \to \cS^{\cK}
\]
 and a natural isomorphism $\cS^{\cK}(X
\boxtimes Y, Z) \iso \cS^{\cK}(X, \Hom(Y,Z))$. This internal $\Hom$ can be
defined via the end construction
\[ \Hom(Y,Z)(\bld{n}) = \int_{\bld{k}\in\cK}
\Map(Y(\bld{k}),Z(\bld{n}\concat\bld{k})).\]

\subsection{Free and semi-free
  \texorpdfstring{$\cK$}{K}-spaces}\label{subs:free-semi-free}
Given an object $\bld k$ in $\cK$, let us write $\cK(\bld k)$ for the monoid of endomorphisms of $\bld k$ and $\cS^{\cK(\bld k)}$ for the category of spaces with left $\cK(\bld k)$-action. The categories $\cS$,
$\cS^{\cK(\bld{k})}$, and $\cS^{\cK}$ are related by various
adjunctions that can be summarized as follows:

\begin{equation}\label{eq:two_adjunctions} \xymatrix{ \cS
    \ar@<.5ex>@/^/[drr]^(.3){F_{\bld{k}}^{\cK}} \ar@<-0.5ex>[d]_{\cK(\bld{k})\times -} \\
    \cS^{\cK(\bld{k})} \ar@<0.5ex>[rr]^(.45){G_{\bld{k}}^{\cK}} \ar@<-0.5ex>[u] & &
    \cS^{\cK} \ar@<0.5ex>[ll]^(.55){\Ev_{\bld{k}}^{\cK}} 
    \ar@<.5ex>@/_/[ull]^(.7){\Ev_{\bld{k}}^{\cK}}  \ar@<0.5ex>[r]^{\colim_{\cK}} &
    \ar@<0.5ex>[l]^{\const_{\cK}}\cS }
\end{equation}
Here $\const_{\cK}$ takes a space to the corresponding constant $\cK$-space,
$\colim_{\cK}$ is its left adjoint, $\cK(\bld{k})\times -$ is
the free $\cK(\bld{k})$-space functor, the unlabeled forgetful functor
is its right adjoint, the two instances of $\Ev_{\bld{k}}^{\cK}$ are
the evaluations of a $\cK$-space at $\bld{k}$ considered as
a space or a $\cK(\bld{k})$-space, and $F_{\bld{k}}^{\cK}$ and
$G_{\bld{k}}^{\cK}$ are the corresponding left adjoints. Explicitly,
for a space $K$ and a $\cK(\bld{k})$-space $L$,
\begin{equation}\label{eq:free-functors}
 F_{\bld{k}}^{\cK}(K) = \cK(\bld{k},-) \times K \quad\textrm{ and
}\quad G_{\bld{k}}^{\cK}(L) = \cK(\bld{k},-)
\times_{\cK(\bld{k})} L
\end{equation}
where the expression for $G_{\bld k}^{\cK}(L)$ indicates the coequalizer of the evident diagram. 

\begin{lemma} \label{lem:product_of_free_Jspaces} There is a natural
  isomorphism
  \begin{align*}
  &F_{\bld{k}}^{\cK}(K) \boxtimes
    F_{\bld{k'}}^{\cK}(K') \iso F_{\bld{k}\concat\bld{k'}}^{\cK}
    (K\times K')\\
\intertext{for each pair of spaces $K$ and $K'$, and a natural isomorphism}  
&G_{\bld{k}}^{\cK}(L) \boxtimes G_{\bld{k'}}^{\cK}(L')\iso
    G_{\bld{k}\concat\bld{k'}}^{\cK}
    (\cK(\bld{k}\concat\bld{k'}) \times_{\cK(\bld{k})\times\cK(\bld{k'})}L\times
    L')
    \end{align*} 
     for each $\cK(\bld k)$-space $L$ and each $\cK(\bld k')$-space $L'$.\qed
\end{lemma}

\section{\texorpdfstring{$\cI$}{I}-spaces and symmetric spectra}\label{sec:I-space-section}
Let $\cI$ be the category whose objects are the finite sets
$\bld{n}=\{1, \dots, n\}$ for $n\geq 0$ ($\bld 0$ is the empty set) and whose morphisms are the injections. The usual ordered concatenation of ordered sets $\concat$ makes this a symmetric monoidal category with unit 
$\bld 0$. The symmetry isomorphism 
$\chi_{m,n}\colon\bld{m}\concat\bld{n}\to \bld{n}\concat\bld{m}$ is the shuffle moving the first $m$ elements past the last $n$ elements.

\subsection{The category of \texorpdfstring{$\cI$}{I}-spaces}
Let $\cS^{\cI}$ be the category of $\cI$-spaces, equipped with the symmetric monoidal structure $(\cS^{\cI},\boxtimes, \bld 1_{\cI})$ inherited from $\cI$. The unit $\bld 1_{\cI}=\cI(\bld 0,-)$ can be identified with the terminal 
$\cI$-space $*$. By definition, an \emph{$\cI$-space monoid} is a monoid in 
$\cS^{\cI}$. We say that a map of $\cI$-spaces $X\to Y$ is an 
\begin{itemize}
\item
$\cI$-equivalence if the induced map of homotopy colimits $X_{h\cI}\to 
Y_{h\cI}$ is a weak homotopy equivalence,
\item
$\cI$-fibration if it is a level fibration and the diagram
\begin{equation}\label{eq:I-fibration}
\xymatrix@-1pc{
X(\bld m) \ar[r] \ar[d]&X(\bld n)\ar[d]\\
Y(\bld m)\ar[r] & Y(\bld n)
}
\end{equation}
is homotopy cartesian for all morphisms $\bld m\to\bld n$ in $\cI$,
\item
 cofibration if it has the left lifting property with respect to maps of 
$\cI$-spaces that are level acyclic  fibrations.
\end{itemize}

These classes specify a model structure on $\cS^{\cI}$ as we show in Proposition \ref{prop:projective-I-model-structure} below. We shall refer to this as the \emph{projective $\cI$-model structure}. 
There also is a \emph{positive projective $\cI$-model structure} on $\cS^{\cI}$: Let $\cI_+$ be the full subcategory of $\cI$ obtained by excluding the initial object $\bld 0$. We say that a map $X\to Y$ of $\cI$-spaces is a   
\begin{itemize}
\item
positive $\cI$-fibration if it is a level fibration for the levels corresponding to objects in $\cI_+$ and the diagrams \eqref{eq:I-fibration} are homotopy cartesian for all morphisms in $\cI_+$,
\item
positive cofibration if it has the left lifting property with respect to maps of 
$\cI$-spaces that are level acyclic fibrations for the levels corresponding to objects in $\cI_+$.
\end{itemize}
A more explicit description of the (positive) cofibrations is given in 
Proposition~\ref{prop:latching-characterization}.

\begin{proposition}\label{prop:projective-I-model-structure}
The $\cI$-equivalences, the (positive) $\cI$-fibrations, and the (positive) cofibrations specify a cofibrantly generated proper simplicial model structure on 
$\cS^{\cI}$. These model structures are monoidal and satisfy the monoid axiom.
\end{proposition}
\begin{proof}
The fact that these classes of maps specify a cofibrantly generated model structure is a consequence of Corollary~\ref{cor:I-J-well-structured} together with Proposition~\ref{prop:K-model-str}. These model structures are proper by Corollary~\ref{cor:proper},  simplicial by Proposition~\ref{prop:K-model-is-S-model-str}, monoidal by 
Proposition~\ref{prop:K-pushout-product-axiom}, and satisfy the monoid axiom by Proposition~\ref{prop:monoid-axiom}.
\end{proof}
It follows from the definitions that the identity functor on $\cS^{\cI}$ is the left Quillen functor of a Quillen equivalence from the positive projective to the projective $\cI$-model structure, see Proposition \ref{prop:well-structured-comparison}.

\begin{theorem}
The adjunction
$
\colim_{\cI} \colon \cS^{\cI} \rightleftarrows
\cS \thinspace \colon   \! \! \const_{\cI}
$
defines a Quillen equivalence between the (positive) projective $\cI$-model structure on $\cS^{\cI}$ and the usual model structure on $\cS$. 
\end{theorem} 
\begin{proof}
$B\cI$ is contractible so this is a special case of Proposition~\ref{prop:colim-const-Q-adjunction}.
\end{proof}

\begin{remark}
A variant of the above Quillen equivalence is considered in \cite{Lind-diagram}
where also part of Proposition~\ref{prop:I-positive-lift} below is verified. One of the main objectives in \cite{Lind-diagram} is a comparison of the $\cI$-space units of a symmetric ring spectrum with the corresponding construction in the $S$-module setting from~\cite{EKMM}.
\end{remark}

The next result is the main reason for introducing the positive projective model structure. We write $\cC\cS^{\cI}$ for the category of commutative 
$\cI$-space monoids. 

\begin{proposition}\label{prop:I-positive-lift}
The positive projective $\cI$-model structure on $\cS^{\cI}$ lifts to a cofibrantly generated proper simplicial model structure on $\cC\cS^{\cI}$. 
\end{proposition}
\begin{proof}
By Corollary \ref{cor:I-J-well-structured} this is a consequence of 
Corollaries~\ref{cor:K-positive-projective-commutative} and \ref{cor:proper}.
\end{proof}

More generally, we show in Proposition \ref{prop:structured-lift-proposition} that if $\cD$ is any operad in $\cS$, then the positive projective model structure lifts to the category of algebras $\cS^{\cI}[\mathbb D]$ for the associated monad $\mathbb D$ (as usual defined by $\mathbb D(X)=\coprod_{n\geq 0}\mathcal D(n)\times_{\Sigma_n}X^{\boxtimes n}$). If the operad is $\Sigma$-free, then the projective model structure also lifts to $\cS^{\cI}[\mathbb D]$. Thus, for instance the projective model structure lifts to the category of (not necessarily commutative) 
$\cI$-space monoids. 

Recall that an \emph{$E_{\infty}$ operad} $\cD$ is an operad which is $\Sigma$-free and whose spaces are contractible, see 
Section~\ref{sec:structured-section} for details. As we recall in 
Remark \ref{rem:model-str-D-spaces}, the assumption that $\cD$ is 
$\Sigma$-free ensures that the usual model structure on $\cS$ lifts to a model structure on the category of algebras $\cS[\mathbb D]$ for the associated monad $\mathbb D$ on $\cS$.

\begin{theorem}
Let $\cD$ be an $E_{\infty}$ operad and let $\mathbb D$ be the associated monad on $\cS$. Then the positive projective $\cI$-model structure on 
$\cC\cS^{\cI}$ is related to the standard model structure on $\cS[\mathbb D]$  by a chain of Quillen equivalences. 
\end{theorem}
\begin{proof}
This is a special case of Theorem \ref{theorem:K-E-infinity-rectification}.
\end{proof}

Using the explicit Quillen equivalences in the above theorem we can rectify $E_{\infty}$ spaces to strictly commutative $\cI$-space monoids in a precise sense.
\begin{corollary}
Let $X$ be an $E_{\infty}$ space for some $E_{\infty}$ operad. Then there exists a commutative $\cI$-space monoid $A$ and a chain of 
$\cI$-equivalences of $E_{\infty}$ $\cI$-spaces $A\xl{\sim} 
Y\xr{\sim} X$ relating $A$ to the constant $\cI$-space $X$. 
\end{corollary} 
\begin{proof}
Suppose that $X$ is a $\mathbb D$-algebra in $\cS$ for the monad 
$\mathbb D$ associated to an $E_{\infty}$ operad $\cD$. Then the corresponding constant $\cI$-space is a $\mathbb D$-algebra in $\cS^{\cI}$ and we let $Y\to X$ be a cofibrant replacement in the positive projective model structure on $\cS^{\I}[\mathbb D]$. Let $\pi\colon \mathbb D\to\mathbb C$ be the canonical projection onto the commutativity monad $\mathbb C$. By Proposition~\ref{prop:operad-change} this gives rise to a Quillen equivalence 
$\pi_* \colon \cS^{\cI}[\mathbb D] \rightleftarrows
\cC\cS^{\cI} \thinspace \colon    \! \pi^*$ relating the respective positive projective $\cI$-model structures. We let $A=\pi_*(Y)$ and observe that the cofibrancy assumption on $Y$ implies that the counit of the adjunction $Y\to\pi^*(A)$ is an $\cI$-equivalence.  
\end{proof}
In fact, by Lemma~\ref{lem:unit-lemma} we may even choose the chain of 
$\cI$-equivalences in the corollary so that they are level equivalences in positive degrees.

\subsection{The flat model structure on \texorpdfstring{$\cI$}{I}-spaces}
\label{subsec:flat-I-model}
For the applications of the theory it is important to be able to decide whether a particular $\cI$-space is homotopically well-behaved with respect to the 
$\boxtimes$-product. We know from Proposition \ref{prop:projective-I-model-structure} that the cofibrant $\cI$-spaces in the projective $\cI$-model structure have this property, but it is inconvenient to restrict our attention to this class of cofibrant objects. In practice, such cofibrant objects rarely occur naturally and must almost always be manufactured using the small object argument. Also, a commutative $\cI$-space monoid that is cofibrant in the lifted positive projective $\cI$-model structure on $\cC\cS^{\cI}$ will not in general have an underlying $\cI$-space that is cofibrant in the projective $\cI$-model structure on $\cS^{\cI}$. Thus, we need another argument to ensure that the cofibrant objects in $\cC\cS^{\cI}$ are homotopically well-behaved with respect to the $\boxtimes$-product. This motivates introducing the flat 
$\cI$-model structure which is the purpose of this section.

Given an object $\bld n$ in $\cI$, we write $(\cI\downarrow \bld n)$ for the comma category of objects in $\cI$ over $\bld n$ and $\partial(\cI\downarrow\bld n)$ for the full subcategory whose  objects $\bld m\to\bld n$ are not isomorphisms. Composing with the forgetful functor $(\cI\downarrow\bld n)\to \cI$, an $\cI$-space $X$ gives rise to a diagram indexed by $\partial(\cI\downarrow\bld n)$. The \emph{$\bld n$th latching space} of $X$ is defined by $L_{\bld n}(X)=\colim_{\partial(\cI\downarrow\bld n)}X$.

\begin{definition}
A map of $\cI$-spaces $X\to Y$ is a 
\emph{flat cofibration} if the induced map $X(\bld n)\cup_{L_{\bld n}(X)}
L_{\bld n}(Y)\to Y(\bld n)$ is a cofibration in $\cS$ for all $\bld n$. It is a
\emph{positive flat cofibration} if it is a flat cofibration and in addition  
$X(\bld 0)\to Y(\bld 0)$ is an isomorphism. 
\end{definition}

We say that a map of $\cI$-spaces is a (positive) flat $\cI$-fibration if it has the right lifting property with respect to the class of (positive) flat cofibrations that are $\cI$-equivalences. A more explicit description of the flat $\I$-fibrations is given in Section~\ref{sec:K-model-structure}.

\begin{proposition}\label{prop:flat-I-model-structure}
The $\cI$-equivalences, the (positive) flat $\cI$-fibrations, and the (positive) flat cofibrations specify a cofibrantly generated proper simplicial model structure on 
$\cS^{\cI}$. These model structures are monoidal and satisfy the monoid axiom. 
\end{proposition}
\begin{proof}
By Corollary~\ref{cor:flat-I-J-well-structured} and the remarks preceding it, the fact that these classes of maps specify a cofibrantly generated model structure is a consequence of Proposition~\ref{prop:K-model-str}.
The model structures are proper by Corollary~\ref{cor:proper}, simplicial by 
Proposition~\ref{prop:K-model-is-S-model-str}, monoidal by 
Proposition~\ref{prop:K-pushout-product-axiom}, and satisfy the monoid axiom by Proposition~\ref{prop:monoid-axiom}.
\end{proof}

We shall refer to this as the \emph{(positive) flat model structure} on 
$\cS^{\cI}$ and the cofibrant objects will be called \emph{flat $\cI$-spaces}. 
It is proved in Proposition~\ref{prop:well-structured-comparison} that the identity functor is the left Quillen functor in a Quillen equivalence from the (positive) projective model structure to the (positive) flat model structure on 
$\cS^{\cI}$. In particular, an $\cI$-space which is cofibrant in the projective model structure is also flat. One of the convenient properties of a flat $\cI$-space $X$ is that the endofunctor $X\boxtimes(-)$ on $\cS^{\cI}$ preserves $\cI$-equivalences; this is proved in 
Proposition~\ref{prop:boxtimes-flat-invariance}.

It is useful to reformulate the flat cofibration condition in terms of the well-known Reedy cofibrations of cubical diagrams.  
For an object $\bld n$ in $\cI$, let $\cP(\bld n)$ denote the category with objects the subsets of $\bld n$ and morphisms the inclusions. Functors 
$C\colon \cP(\bld n)\to\cS$ may be viewed as $\bld n$-cubical diagrams. The partial ordering of the objects in $\cP(\bld n)$ gives rise to the usual Reedy model structure on the category of $\bld n$-cubical diagrams, see e.g.\ \cite[Ch.\ 15]{Hirschhorn-model}. We shall only make use of the cofibration part of this structure. Given an $\bld n$-cubical diagram $C$ and a subset $V$ of $\bld n$, the \emph{$V$th-latching space}  is defined by 
$L_V(C)=\colim_{U\varsubsetneq V}C(U)$. A map of $\bld n$-cubical diagrams $C\to D$ is said to be a cofibration if the induced map 
$C(V)\cup_{L_V(C)}L_V(D)\to D(V)$ is a cofibration in $\cS$ for all subsets 
$V$ of $\bld n$. In particular, an $\bld n$-cube $C$ is cofibrant if and only if it is a cofibration cube in the sense of Goodwillie \cite{Goodwillie-calculus}, 
that is, the map $L_V(C)\to C(V)$ is a cofibration for each subset $V$.

Cubical diagrams arise from $\cI$-spaces in the following way: The category 
$\cP(\bld n)$ maps isomorphically onto the skeletal subcategory of 
$(\cI\downarrow \bld n)$ given by the objects $\bld m\to\bld n$ that are order preserving. Composing with the forgetful functor to $\cI$, an $\cI$-space thus gives rise to an $\bld n$-cubical diagram for all $\bld n$. 
It is clear from the definitions that a map of $\cI$-spaces $X\to Y$ is a flat cofibration if and only if the induced map of $\bld n$-cubical diagrams is a cofibration for all $\bld n$.   

\begin{proposition}\label{prop:simplicial-flatness}
In the simplicial setting an $\cI$-space $X$ is flat if and only if each morphism $\bld m\to \bld n$ induces a cofibration $X(\bld m)\to X(\bld n)$ and for each diagram of the following form (with maps induced by the evident order preserving morphisms)
\begin{equation}\label{eq:flat-criterion}
\xymatrix@-1pc{
X(\bld m)\ar[r]\ar[d] & X(\bld m\sqcup\bld n)\ar[d]\\
X(\bld l\sqcup\bld m)\ar[r] & X(\bld l\sqcup\bld m\sqcup\bld n)
}
\end{equation}
the intersection of the images of $X(\bld l\sqcup\bld m)$ and 
$X(\bld m\sqcup\bld n)$ in $X(\bld l\sqcup\bld m\sqcup\bld n)$ equals the image of $X(\bld m)$. 
\end{proposition}
\begin{proof}
First notice that an $\bld n$-cubical diagram of simplicial sets is a cofibration cube if and only if (i) each inclusion $U\subseteq V$ induces a cofibration $C(U)\to C(V)$, and (ii) for each pair of subsets $U$ and $V$, the intersection of the images of $C(U)$ and $C(V)$ in $C(U\cup V)$ equals the image of $C(U\cap V)$. One can check this inductively using the following principle: Let $V$ be a finite set and $C$ a $V$-cubical diagram. Let $U$ be the subset obtained by removing a point from $V$. Then we may view $C$ as a map of $U$-cubical diagrams $C\colon D\to E$ and there is a pushout diagram 
\[
\xymatrix@-1pc{
L_U(D) \ar[r]\ar[d] & D(U)\ar[d]\\
L_U(E) \ar[r] & L_V(C).
}
\]  
With this description of a cofibration cube it is clear that the cubical diagrams associated to an $\cI$-space are cofibration cubes precisely when the conditions in the lemma are satisfied.
\end{proof}
 
In the topological setting we cannot state the obvious analogue of the above flatness criterion since we lack a sufficiently general gluing principle for topological cofibrations. Instead we have the following weaker result which is proved by a similar argument.  
 
\begin{proposition}
In the topological setting an $\cI$-space $X$ is flat provided that each of the spaces $X(\bld n)$ is a CW-complex, each morphism $\bld m\to\bld n$ induces an isomorphism of $X(\bld m)$ onto a subcomplex of $X(\bld n)$, and for each diagram of the form \eqref{eq:flat-criterion} the intersection of the images of $X(\bld l\sqcup\bld m)$ and $X(\bld m\sqcup\bld n)$ in 
$X(\bld l\sqcup\bld m\sqcup\bld n)$ equals the image of $X(\bld m)$.
\qed
\end{proposition}

\begin{example}
The $\cI$-spaces $BO$ and $B\GL(R)$ in Examples~\ref{ex:BO-example} and \ref{ex:BGL(R)-example} are flat. The $\I$-space $X^{\bullet}$ in 
Example~\ref{ex:X-bullet-example} is flat in the simplicial setting and is flat in the topological setting if we assume that $X$ is a based CW-complex. None of these $\cI$-spaces are cofibrant in the projective model structure.
\end{example}

\begin{remark}
In the topological setting there also is a weaker $h$-cofibration notion of flatness which is characterized by a condition analogous to that in 
Proposition~\ref{prop:simplicial-flatness}. This is the flatness criterion used in \cite{Blumberg-C-S_THH-Thom,Schlichtkrull-infinite}, but it is not the right notion in the present setting of cofibrantly generated model categories.   
\end{remark}

The next result is one of the main reasons for considering the flat model structure.

\begin{proposition}\label{prop:I-underlying-flat}
\begin{enumerate}[(i)]
\item
The positive flat model structure on $\cS^{\cI}$ lifts to a cofibrantly generated proper simplicial model structure on $\cC\cS^{\cI}$. 
\item
Suppose that $A$ is a commutative $\cI$-space monoid which is cofibrant in the lifted model structure in (i). Then the underlying $\cI$-space of $A$ is flat. 
\end{enumerate}
\end{proposition}
\begin{proof}
The statement in (i) that the positive flat model structure lifts to $\cC\cS^{\cI}$ is a consequence of Corollary \ref{cor:flat-I-J-well-structured} and Proposition~\ref{prop:structured-lift-proposition}. Properness follows from Corollary~\ref{cor:proper} and the claim about the simplicial structure is verified in Proposition~\ref{prop:CSK-simplicial}. The statement in (ii) is a special case of Corollary~\ref{cor:comm-underlying-cofibrant}. \end{proof}

As remarked at the beginning of the section, the analogues result fails for the (positive) projective model structure. 

The flat model structure on $\cI$-spaces is analogous to the flat model structure on symmetric spectra established in \cite{HSS} and 
\cite{Shipley-convenient} (we use the terminology introduced by Schwede \cite{Schwede-SymSp}; the flat model structure on $\Spsym$ is what is called the $S$-model structure in  \cite{HSS} and \cite{Shipley-convenient}). In Proposition \ref{prop:cScI-Spsym-adjunction} we make this analogy precise by establishing a Quillen adjunction relating the two model structures. 

\subsection{Recollections on symmetric spectra}
When discussing symmetric spectra we shall frequently consider spheres indexed by finite sets and isomorphisms between them induced by bijections. 
As explained in Section~\ref{sec:preliminaries} we simultaneously work in a topological and a simplicial setting and we write $\cS_*$ for the category of based spaces. 
Given a finite set $X$, let $S^X$ be the smash product $\bigwedge_{x\in X}S^1$ defined as the quotient space of the $X$-fold product 
$\prod_{x\in X}S^1$ by the subspace where one of the components equals the base point. For a morphism $\alpha\colon\bld m\to\bld n$ in $\cI$, we write $\bld n-\alpha$ for the complement of $\alpha(\bld m)$ in $\bld n$. There is a canonical extension of $\alpha$ to a bijection  $\bld m\concat(\bld n-\alpha)\to\bld n$ (which is the inclusion of $\bld n-\alpha$) and this gives rise to the isomorphism
\begin{equation}\label{eq:reindexing-basic} S^{\bld{m}} \sm
S^{\bld{n}-\alpha}\xrightarrow{\iso} S^{\bld{n}}.
\end{equation} 
Restricting to morphisms $\alpha \in \cI(\bld{n},\bld{n})=\Sigma_n$,
this defines the usual left $\Sigma_n$-action on $S^n$. For a pair of morphisms $\alpha\colon\bld{l}\to\bld{m}$ and $\beta\colon\bld{m}\to\bld{n}$
there is a canonical bijection 
$(\bld{m}-\alpha)\concat(\bld{n}-\beta)\to\bld{n}-\beta\alpha$, obtained by applying $\beta$ to the elements in $\bld m-\alpha$, 
and an associated isomorphism
\begin{equation}\label{eq:reindexing-composite} S^{\bld{m}-\alpha} \sm
S^{\bld{n}-\beta} \xrightarrow{\iso} S^{\bld{n}-\beta\alpha}.
\end{equation} 
Given morphisms $\alpha\colon\bld{m}\to\bld{n}$ and
$\alpha'\colon\bld{m'}\to\bld{n'}$, there is a canonical identification of
$(\bld{n}\concat\bld{n'})-(\alpha\concat\alpha')$ with  
$(\bld{n}-\alpha) \concat (\bld{n'}-\alpha')$ and therefore an isomorphism
\begin{equation}\label{eq:reindexing-product}
S^{(\bld{n}\concat\bld{n'})-(\alpha\concat\alpha')} \iso
S^{\bld{n}-\alpha} \sm S^{ \bld{n'}-\alpha'}.
\end{equation}

Recall from \cite{HSS} and \cite{MMSS} that a symmetric spectrum $E$ is a spectrum (in $\cS_*$) with structure maps $E_m\wedge S^1\to 
E_{m+1}$ such that the $m$th space $E_m$ has a left $\Sigma_m$-action and the iterated structure maps $E_m\wedge S^n\to E_{m+n}$ are 
$\Sigma_m\times \Sigma_n$-equivariant. Given a morphism 
$\alpha\colon\bld m\to\bld n$ in $\cI$ there is an induced structure map 
$\alpha_*\colon E_m\wedge S^{\bld n-\alpha}\to E_n$ defined as follows: Choose a bijection 
$\beta\colon\bld l\to \bld n-\alpha$ for an object $\bld l$ in $\cI$ and let 
$\{\alpha,\beta\}\colon \bld m\concat\bld l\to\bld n$ be the resulting bijection. Then $\alpha_*$ is defined by
\[
\alpha_*\colon E_m\wedge S^{\bld n-\alpha}\xr{\eins\wedge\beta^{-1}}
E_m\wedge S^{\bld l}\to E_{m+l}\xr{\{\alpha,\beta\}_*}E_n
\] 
which is independent of the choice of $\beta$. With this convention, the subset inclusion $\iota_{\bld{m}}\colon\bld{m}\to\bld{m+1}$ induces the structure map $E_m \sm S^1 \to E_{m+1}$ and the endomorphisms of  
$\bld m$ induce the left $\Sigma_m$-action on $E_m$. See~\cite[\S 3.1]{Schlichtkrull-Thom_symm} and~\cite{Schwede-SymSp} for more details on this perspective on symmetric spectra.

The functor $\Ev_m \colon \Spsym \to \cS_*$ sending a symmetric
spectrum $E$ to $E_m$ has a left adjoint $F_m\colon \cS_* \to
\Spsym$. It can be defined explicitly as
\begin{equation}\label{eq:free_sym_sp} F_m(K)_n = \bigvee_{\alpha \in \cI(\bld{m},\bld{n})} K \sm
S^{\bld{n}-\alpha}.\end{equation} Here we use the notation $S^{\bld{n}-\alpha}$ to both keep track of the dimension of the sphere and the different copies of it. A morphism $\beta\colon\bld{n}\to\bld{p}$ in $\cI$ induces
the structure map $\beta_*\colon F_m(K)_n\sm S^{\bld{p}-\beta} \to 
F_m(K)_p$. This maps the wedge  summand indexed by 
$\alpha\colon \bld{m}\to\bld{n}$ to the wedge summand indexed by $\beta\alpha$ via the isomorphism
\[ K \sm S^{\bld{n}-\alpha} \sm S^{\bld{p}-\beta} \to K \sm
S^{\bld{p}-\beta\alpha},\] specified by $\beta$ as in
\eqref{eq:reindexing-composite}.  

It is proved in \cite{HSS} that the smash product of symmetric spectra makes $\Spsym$ a symmetric monoidal category with unit the sphere spectrum $\mathbb S$. Using the above notation, the smash product $E\sm E'$ of a pair of symmetric spectra $E$ and $E'$ can be described explicitly in degree $n$ by
\[ (E\sm E')_{n} =
\colim_{\alpha\colon\bld{k}\concat\bld{k'}\to\bld{n}}E_k\sm
E'_{k'}\sm S^{\bld{n}-\alpha}.\] The colimit is taken over the comma category
$(\concat \downarrow \bld{n})$, and a morphism
\[ (\gamma,\gamma')\colon(\bld{k},\bld{k'},
\bld{k}\concat{\bld{k'}}\xrightarrow{\alpha}
\bld{n})\to(\bld{l},\bld{l'},\bld{l}\concat\bld{l'}
\xrightarrow{\beta}\bld{n})\] in this category (where by definition $\alpha=
\beta(\gamma\concat\gamma')$) 
induces the map
\[E_{k}\sm E'_{k'} \sm S^{\bld{n}-\alpha}
\to E_{k} \sm S^{\bld{l}-\gamma} \sm E_{k'} \sm S^{\bld{l'}-\gamma'} 
\sm S^{\bld{n}-\beta}
\to E_{l}\sm E_{l'} \sm S^{\bld{n}-\beta}\]
in the colimit system, again
utilizing the isomorphisms \eqref{eq:reindexing-composite} and
\eqref{eq:reindexing-product}.

Coming back to free symmetric spectra, we can use the above to get an explicit description of the isomorphism
\begin{equation}\label{eq:smash_of_free_iso} 
  F_m(K) \sm F_{m'}(K')\xr{\cong} F_{m+m'}(K \sm K')
\end{equation} 
of~\cite[Proposition 2.2.6(1)]{HSS}. In spectrum degree $n$ this is the map from the colimit over $(\concat \downarrow \bld{n})$ which for each object 
$(\bld{k},\bld{k}',\alpha\colon\bld{k}\concat\bld{k}'\to\bld n)$ takes the wedge summand indexed by $\beta\colon\bld m\to\bld k$ and  
$\beta'\colon\bld{m}'\to\bld{k}'$ to the wedge summand indexed by 
$\gamma(\beta\concat\beta')\colon \bld{m}\concat\bld{m}'\to\bld n$ via the isomorphism
\[
K\wedge S^{\bld k-\beta}\wedge K'\wedge S^{\bld{k}'-\beta'}\wedge 
S^{\bld n-\alpha}\to K\wedge K'\wedge S^{\bld n-\alpha(\beta\concat\beta')}
\]
induced by \eqref{eq:reindexing-composite} and \eqref{eq:reindexing-product}.
Under the isomorphism \eqref{eq:smash_of_free_iso}, the symmetry isomorphism of the smash product of $F_m(K)$ and $F_{m'}(K')$ corresponds  to the map of free symmetric spectra 
$F_{m+m'}(K\sm K') \to F_{m'+m}(K' \sm K)$ that maps the wedge summand indexed by $\alpha\colon\bld{m}\concat\bld{m}'\to\bld{n}$ 
to the wedge summand indexed by 
$\alpha\chi_{m,m'}\colon\bld{m}'\concat\bld{m}\to\bld n$ via the isomorphism
\begin{equation}\label{eq:symmetry_iso_spsym} 
K\wedge K' \sm S^{\bld{n}-\alpha} \to K'\wedge K \sm
  S^{\bld{n} - \alpha \chi_{m,m'}}
\end{equation}
that flips the $K$ and $K'$ factors and is the identity on 
$S^{\bld n-\alpha}=S^{\bld n-\alpha\chi_{m,m'}}$.

\subsection{\texorpdfstring{$\cI$}{I}-spaces and symmetric spectra} 
An ordinary ring has an underlying monoid which in turn contains the group of units as displayed in the diagram of adjoint functors
\[
\big\{\textrm{(comm.) groups}\big\}
\rightleftarrows
\big\{\textrm{(comm.) monoids}\big\}
\rightleftarrows
\big\{\textrm{(comm.) rings}\big\}.
\]
We wish to model a topological version of these adjunctions using $\cI$-space monoids and to use this to define the units of a symmetric ring spectrum.  

By adjointness, maps of symmetric spectra $F_n(S^n)\to F_m(S^m)$ are in 
one-to-one correspondence with maps $S^n\to F_m(S^m)_n$. For a morphism 
$\alpha\colon\bld m\to\bld n$, let $\alpha^*\colon F_n(S^n)\to F_m(S^m)$ be the map which is adjoint to
\[
S^{\bld n}\xl{\cong}S^{\bld m}\wedge S^{\bld n-\alpha}\to 
\bigvee_{\beta\colon\bld m\to\bld n}S^{\bld{m}}\wedge S^{\bld n-\beta}. 
\]
The first map is the isomorphism \eqref{eq:reindexing-basic} induced by 
$\alpha$ and the second map is the inclusion of  $S^{\bld m}\wedge S^{\bld n-\alpha}$ as the wedge summand indexed by $\alpha$. 
With the explicit descriptions of free symmetric spectra and smash
products given above, it is easy to verify the following lemma. 
(It can also be deduced from Lemma \ref{lem:J-and-free-spsym} below.)
\begin{lemma}\label{lem:I-and-free-spsym}
The free symmetric spectra on spheres assemble to a strong symmetric
monoidal functor
$ F_{-}(S^{-}) \colon \cI^{\op} \to \Spsym, \ \bld{n} \mapsto
F_n(S^n)$. \qed
\end{lemma}

Let $X$ be an $\cI$-space. As in the general situation considered in 
Section~\ref{sec:diagram-spaces-symmetric-spectra} we write 
$\mathbb S^{\cI}[X]$ for the symmetric spectrum defined as the coend of the ($\cI^{\op}\times\cI$)-diagram 
$F_m(S^m)\wedge X(\bld n)_+$ where $(-)_+$ denotes a disjoint base point. Given a symmetric spectrum $E$ we write $\Omega^{\cI}(E)$ for the $\cI$-space defined by $\Map_{\Spsym}(F_{-}(S^{-}),E)$. 
Let $\cC\Spsym$ denote the category of commutative symmetric ring spectra. An application of Proposition \ref{prop:structured-adjunction-SK-Spsym} then
provides the two adjoint pairs of functors
\begin{equation}\label{eq:cScI-Spsym-adjunction-copied}
  \mS^{\cI}[-] \colon
  \cS^{\cI} \rightleftarrows \Spsym \colon
  \Omega^{\cI}
  \qquad \text{and}   \qquad 
  \mS^{\cI}[-] \colon \cC\cS^{\cI} \rightleftarrows \cC\Spsym \colon
  \Omega^{\cI},
\end{equation}
and more generally an adjunction relating the categories of 
$\mathbb D$-algebras $\cS^{\cI}[\mathbb D]$ and $\Spsym[\mathbb D]$ for any operad $\cD$ with associated monad $\mathbb D$. Checking from the definitions, we find that 
\[
\mS^{\cI}[X]_n=S^n\wedge X(\bld n)_+
\quad\textrm{ and } \quad \Omega^{\cI}(E)(\bld{n})=
\Omega^{n}(E_n).
\]

Recall that we use the term \emph{flat model structure} for the model structure on $\Spsym$ which is called the $S$-model structure in \cite{HSS} and \cite{Shipley-convenient}.

\begin{proposition}\label{prop:cScI-Spsym-adjunction}
\begin{enumerate}[(i)]
\item
The first adjunction in \eqref{eq:cScI-Spsym-adjunction-copied} is a Quillen adjunction with respect to the (positive) projective and (positive) flat 
$\cI$-model structures on $\cS^{\cI}$ and the corresponding (positive) projective and (positive) flat stable model structures on $\Spsym$.
\item
The second adjunction in \eqref{eq:cScI-Spsym-adjunction-copied} is a Quillen adjunction with respect to the positive projective and positive flat 
$\cI$-model structures on $\cC\cS^{\cI}$ and the corresponding positive projective and positive flat stable model structures on $\cC\Spsym$.
\end{enumerate}
\end{proposition}
\begin{proof}
This is a consequence of Proposition \ref{prop:K-spaces-Spsym-adjunction} and the descriptions of the respective model structures on $\Spsym$ given in 
\cite{HSS,MMSS,Shipley-convenient}; see also 
\cite{Schwede-SymSp}. In the flat case the main point is that the diagonal 
$\Sigma_n$-action on $F_n(S^n)$, acting both on $\cI(\bld n,-)$ and $S^n$, is free away from the base point.  
\end{proof}

The $\cI$-space monoid of units $A^{\times}$ associated to an $\cI$-space monoid $A$ is defined by letting $A^{\times}(\bld n)$ be the union of the components in $A(\bld n)$ that represent units in the monoid $\pi_0(A_{h\I})$.
It follows immediately from the definitions that if $A$ is (positive) fibrant, then $A^{\times}$ is also (positive) fibrant, and that if $A$ is commutative, then $A^{\times}$ is also commutative. 
We say that an $\cI$-space monoid $A$ is \emph{grouplike} if $A_{h\cI}$ is a grouplike monoid in $\cS$. Clearly $A$ is grouplike if and only if $A^{\times}=A$, which implies that the functor $A\mapsto A^{\times}$ from (commutative)  $\cI$-space monoids to (commutative) grouplike $\I$-space monoids is a right adjoint of the inclusion functor.  

\begin{definition}
Let $R$ be a  symmetric ring spectrum. The $\cI$-space units of $R$ is the grouplike $\cI$-space monoid $\GLoneI(R)=\Omega^{\cI}(R)^{\times}$. 
\end{definition}

The functor $R\mapsto \GLoneI(R)$ is defined for all symmetric ring spectra $R$ and provides a right adjoint of the functor that to a (commutative) grouplike $\cI$-space monoid $A$ associates the (commutative) symmetric ring spectrum $\mathbb S^{\cI}[A]$. However, one should keep in mind that 
$\GLoneI(R)$ only represents the ``correct'' homotopy type of the units when $R$ is positive fibrant (or at least semistable in the sense of \cite{Shipley-THH}).   

\begin{proposition}
If $R$ is a (positive) fibrant symmetric ring spectrum, then the monoid homomorphism $\pi_0(\GLoneI(R)_{h\cI})\to \pi_0(\Omega^{\cI}(R)_{h\cI})$ realizes the inclusion of units $\pi_0(R)^{\times}\to\pi_0(R)$.\qed
\end{proposition}

\section{\texorpdfstring{$\cJ$}{J}-spaces and symmetric spectra}\label{sec:J-space-section}
Here we introduce the category $\cJ$ and discuss the homotopy theory of $\cJ$-spaces and the relation to symmetric spectra. We end the section with a sample application to topological logarithmic structures. 

\subsection{The category \texorpdfstring{$\cJ$}{J}}\label{subs:cat-J}
First we give an explicit description of the category $\cJ$. After this, in Proposition \ref{prop:Grayson-Quillen-J} below, we exhibit $\cJ$ as Quillen's localization construction on the category of finite sets and bijections.
\begin{definition} The objects of the category $\cJ$ are pairs 
$(\bld{n_1},\bld{n_2})$ of objects in $\cI$ and a morphism
  $(\bld{m_1},\bld{m_2}) \to (\bld{n_1},\bld{n_2})$ is a triple 
  $(\beta_1,\beta_2, \sigma)$ with $\beta_1\colon \bld {m_1}\to \bld{n_1}$ and 
  $\beta_2\colon \bld{m_2}\to\bld{n_2}$ morphisms in $\cI$, and 
  $\sigma\colon\bld{n_1}-\beta_1 \to \bld{n_2} - \beta_2$ a bijection
    identifying the complement of $\beta_1(\bld{m_1})$ in $\bld{n_1}$ with
    the complement of $\beta_2(\bld{m_2})$ in $\bld{n_2}$.
  Given composable morphisms
  \[ (\bld{l_1},\bld{l_2}) \xrightarrow{(\alpha_1, \alpha_2,\rho)}(\bld{m_1},\bld{m_2})\xrightarrow{(\beta_1, \beta_2,\sigma)}(\bld{n_1},\bld{n_2}),\] 
  the first two entries of their composite $(\gamma_1,\gamma_2,\tau)$ are 
  $\gamma_1=\beta_1\alpha_1$ and $\gamma_2=\beta_2 \alpha_2$. It
   remains to specify a bijection $\tau \colon \bld{n_1} - \beta_1\alpha_1
  \to \bld{n_2}-\beta_2\alpha_2$. The set $\bld{n_i} -\beta_i\alpha_i$ is the  
  disjoint union of $\bld{n_i}-\beta_i$ and
  $\beta_i(\bld{m_i}-\alpha_i)$ for $i=1,2$, and we define

  \[\tau(s)=\begin{cases} \sigma(s) &\textrm{ if } s \in
    \bld{n_1} - \beta_1 \quad\textrm{ and } \\ \beta_2(\rho(t))
    &\textrm{ if } s= \beta_1(t) \in \beta_1(\bld{m_1} -
    \alpha_1). \end{cases}\] 
    Slightly imprecisely, we refer to
  $\tau$ as $\sigma \cup \beta_2 \rho \beta_1^{-1}$. To see that
  $\cJ$ is indeed a category we have to verify that composition is
  associative. Consider the composable morphisms
  \[
  (\bld{k_1},\bld{k_2})\xr{(\alpha_1,\alpha_2,\rho)}(\bld{l_1},\bld{l_2})
  \xr{(\beta_1,\beta_2,\sigma)}(\bld{m_1},\bld{m_2})
  \xr{(\gamma_1,\gamma_2,\tau)}(\bld{n_1},\bld{n_2}).
  \]
  Associativity is clear for the injections in the first two entries. For the bijections one checks that
\[
\tau\cup\gamma_2\sigma\gamma_1^{-1}\cup \gamma_2\beta_2\rho(\gamma_1\beta_1)^{-1}=\tau\cup\gamma_2(\sigma\cup\beta_2\rho\beta_1^{-1})\gamma_1^{-1}.
\]   
\end{definition}
Let $\concat \colon
  \cJ \times \cJ \to \cJ$ be the functor defined on objects by
  \[(\bld{m_1},\bld{m_2}) \concat (\bld{n_1},\bld{n_2}) =
  (\bld{m_1}\concat\bld{n_1},\bld{m_2}\concat\bld{n_2}),\]
  and on morphisms by
  $(\alpha_1,\alpha_1,\rho)\concat(\beta_1,\beta_2,\sigma) =
  (\alpha_1\concat\beta_1,\alpha_2\concat\beta_2,\rho\concat\sigma)$, where 
  $\rho\concat \sigma$ is the bijection induced by $\rho$ and $\sigma$.
Recall that a permutative category is a symmetric monoidal category with strict unit and strict associativity, see for example  
\cite[Definition 3.1]{Elmendorf-M_infinite-loop}. The fact that $\I$ is permutative easily implies that the same holds for $\cJ$.

\begin{proposition} The data $(\cJ, \concat, (\bld{0},\bld{0}))$ defines a permutative category with  symmetry isomorphism
  \[(\chi_{m_1,n_1},\chi_{m_2,n_2},\eins_{\emptyset}) \colon
  (\bld{m_1},\bld{m_2})\concat(\bld{n_1},\bld{n_2}) \to
  (\bld{n_1},\bld{n_2})\concat(\bld{m_1},\bld{m_2}).
  \eqno\qed
  \]
\end{proposition}
It is easy to see that there is a strong symmetric monoidal diagonal
functor $\Delta\colon \cI \to \cJ$ with
$\Delta(\bld{n})=(\bld{n},\bld{n})$ and $\Delta(\alpha\colon
\bld{m}\to\bld{n})=(\alpha,\alpha,\eins_{\bld{n}-\alpha})$. Many
constructions in connection with $\cI$ and $\cJ$ will be related
through $\Delta$.

\begin{proposition}\label{prop:Grayson-Quillen-J}
The category $\cJ$ is isomorphic to Quillen's localization construction $\Sigma^{-1}\Sigma$ on the category $\Sigma$ of finite sets and bijections. 
\end{proposition}
\begin{proof}
  Let $\Sigma \subset \cI$ be the subcategory of finite sets and
  bijections with the symmetric monoidal structure inherited from $\cI$. 
  Quillen's localization construction~\cite[p.\ 219]{Grayson-higher} on $\Sigma$
  is the category $\Sigma^{-1}\Sigma$ whose objects are pairs $(\bld{n_1},\bld{n_2})$ of objects in $\cI$, and whose
   morphisms  from $(\bld{m_1},\bld{m_2})$ to $(\bld{n_1},\bld{n_2})$ are  
   isomorphism classes of tuples
  \[ \big((\bld{m_1},\bld{m_2}), (\bld{n_1},\bld{n_2}), \bld{l},
  (\bld{m_1}\concat\bld{l},\bld{m_2}\concat\bld{l})
  \xrightarrow{(\alpha_1,\alpha_2)}(\bld{n_1},\bld{n_2})\big).\]
  Here $\bld l$ is an object in $\Sigma$ and $(\alpha_1,\alpha_2)$ is a  
  morphism in $\Sigma\times\Sigma$.    
  An isomorphism of tuples is given by a morphism
  $\sigma\colon\bld{l}\to\bld{l}$ in $\Sigma$ such that
  \[\xymatrix@-1pc{
    (\bld{m_1}\concat\bld{l},\bld{m_2}\concat\bld{l}) 
    \ar[rr]^{(\eins_{\bld{m_1}}\concat\sigma, 
    \eins_{\bld{m_2}}\concat\sigma)}
    \ar[dr]_{(\alpha_1,\alpha_2)} && (\bld{m_1}\concat\bld{l},
    \bld{m_2}\concat\bld{l}) \ar[dl]^{(\alpha_1',\alpha_2')} \\
    & (\bld{n_1},\bld{n_2}) }\] commutes. Notice, that whereas in 
    \cite{Grayson-higher} $\Sigma^{-1}\Sigma$ is defined using the monoidal
    left action of $\Sigma$ on itself, we here use the right action instead. The
    resulting categories are canonically isomorphic, but our conventions are
    more convenient when defining the components of a $\cJ$-space
    monoid (see Section \ref{subsec:J-components} below).   
    
    The desired isomorphism  
    $\Sigma^{-1}\Sigma \to \cJ$ is defined by sending a morphism 
    represented by $(\alpha_1,\alpha_2)$ as above to the morphism
  \[ \big(\alpha_1|_{\bld{m_1}}, \alpha_2|_{\bld{m_2}},
  (\bld{n_1}-\alpha_1({\bld{m_1}}))\xrightarrow{({\alpha_1}|_{\bld{l}})^{-1}}
  \bld{l} \xrightarrow{\alpha_2|_{\bld{l}}} 
  (\bld{n_2}-\alpha_2(\bld{m_2}))\big).\] 
  This does not depend on the choice of representative
  $(\alpha_1, \alpha_2)$.
\end{proof}
The arguments of~\cite[p.\ 224]{Grayson-higher} and the
Barratt-Priddy-Quillen Theorem therefore determine the homotopy type of the classifying space of $\cJ$.
\begin{corollary}\label{cor:BPQ-corollary}
  The classifying space $B\cJ$ is homotopy equivalent to $Q(S^0)$.\qed
\end{corollary}

\begin{remark}
 As pointed out to the authors, Kro~\cite{Kro-orthogonal}
  considered the analogue of the category $\cJ$ for orthogonal spectra
  in order to define a symmetric monoidal fibrant replacement functor
  for orthogonal spectra. This application does not carry over to
  symmetric spectra, cf.\ \cite[Remark 3.4]{Kro-orthogonal}. 
  The analogues to our applications of $\cJ$ in the
  orthogonal context are not addressed in~\cite{Kro-orthogonal},
  although it is potentially interesting to consider diagram spaces
  indexed by the category Kro describes.
\end{remark}

\subsection{The category of $\cJ$-spaces}
Let $\cS^{\cJ}$ be the category of $\cJ$-spaces, equipped with the symmetric monoidal structure $(\cS^{\cJ},\boxtimes, \bld 1_{\cJ})$ inherited from $\cJ$. Notice that, contrary to the situation for $\cI$-spaces, the unit $\bld 1_{\cJ}=\cJ((\bld{0},\bld{0}), -)$ is not isomorphic to the terminal $\cJ$-space $*$. By definition, a \emph{$\cJ$-space monoid} is a monoid in $\cS^{\cJ}$. We say that a map of 
$\cJ$-spaces $X\to Y$ is a 
\begin{itemize}
\item
$\cJ$-equivalence if the induced map of homotopy colimits $X_{h\cJ}\to 
Y_{h\cJ}$ is a weak homotopy equivalence,
\item
$\cJ$-fibration if it is a level fibration and the diagram
\begin{equation}\label{eq:J-fibration}
\xymatrix@-1pc{
X(\bld m_1,\bld m_2) \ar[r] \ar[d]&X(\bld n_1,\bld n_2)\ar[d]\\
Y(\bld m_1,\bld m_2)\ar[r] & Y(\bld n_1,\bld n_2)
}
\end{equation}
is homotopy cartesian for all morphisms $(\bld m_1,\bld m_2)\to 
(\bld n_1,\bld n_2)$ in $\cJ$,
\item
cofibration if it has the left lifting property with respect to maps of 
$\cJ$-spaces that are level acyclic fibrations. 
\end{itemize}
These classes specify a model structure on $\cS^{\cJ}$ as we show in Proposition \ref{prop:projective-J-model-structure} below. We shall refer to this as the \emph{projective $\cJ$-model structure}. There also is a \emph{positive projective $\cJ$-model structure} on $\cS^{\cJ}$ as we discuss next. Let $\cJ_+$ be the full subcategory of $\cJ$ with objects $(\bld n_1,\bld n_2)$ such that 
$|\bld n_1|\geq 1$. We say that a map $X\to Y$ of $\cJ$-spaces is a   
\begin{itemize}
\item
positive $\cJ$-fibration if it is a level fibration for the levels corresponding to objects in $\cJ_+$ and the diagrams \eqref{eq:J-fibration} are homotopy cartesian for all morphisms in $\cJ_+$,
\item
positive cofibration if it has the left lifting property with respect to maps of 
$\cJ$-spaces that are level acyclic fibrations for the levels corresponding to objects in $\cJ_+$.
\end{itemize}
A more explicit description of the (positive) cofibrations is given in 
Proposition~\ref{prop:latching-characterization}.

\begin{proposition}\label{prop:projective-J-model-structure}
The $\cJ$-equivalences, the (positive) $\cJ$-fibrations, and the (positive) cofibrations specify a cofibrantly generated proper simplicial model structure on 
$\cS^{\cJ}$. These model structures are monoidal and satisfy the monoid axiom.
\end{proposition}
\begin{proof}
The fact that these classes of maps specify a cofibrantly generated model structure is a consequence of Corollary~\ref{cor:I-J-well-structured} together with Proposition~\ref{prop:K-model-str}. These model structures are proper by Corollary~\ref{cor:proper}, simplicial by Proposition~\ref{prop:K-model-is-S-model-str},  monoidal by 
Proposition~\ref{prop:K-pushout-product-axiom}, and satisfy the monoid axiom by Proposition~\ref{prop:monoid-axiom}.
\end{proof}
It is clear from the definitions that the identity functor on $\cS^{\cJ}$ is the left Quillen functor of a Quillen equivalence from the positive projective to the projective $\cJ$-model structure, see Proposition \ref{prop:well-structured-comparison}.

For the next result we equip the category $\cS/B\cJ$ of spaces over $B\cJ$ with the standard model structure inherited from the usual model structure on 
$\cS$. 
\begin{theorem}
There is a chain of Quillen equivalences relating $\cS^{\cJ}$ with the projective $\cJ$-model structure to $\cS/B\cJ$ with the standard model structure.
\end{theorem} 
\begin{proof}
This is a special case of Theorem~\ref{thm:K-spaces-over-BK}.
\end{proof}
Any $\cJ$-space is naturally augmented over the terminal $\cJ$-space
$*$. On the level of homotopy categories the above adjunction takes a
$\cJ$-space $X$ to the induced map $X_{h\cJ}\to *_{h\cJ}=B\cJ$. As
discussed in the introduction, this justifies interpreting
$\cJ$-spaces as graded objects.

We write $\cC\cS^{\cJ}$ for the category of commutative $\cJ$-space monoids. 

\begin{proposition}
The positive projective $\cJ$-model structure on $\cS^{\cJ}$ lifts to a cofibrantly generated proper simplicial model structure on $\cC\cS^{\cJ}$. 
\end{proposition}
\begin{proof}
By Corollary \ref{cor:I-J-well-structured} this is a special case of 
Corollary~\ref{cor:K-positive-projective-commutative}.
\end{proof}

As for  $\cI$-spaces both the projective and the positive projective $\cJ$-model structures lift to the category of (not necessarily commutative) $\cJ$-space monoids; this is a consequence of 
Proposition~\ref{prop:structured-lift-proposition}.

Let again $\cE$ denote the Barratt-Eccles operad and $\mathbb E$ the associated monad on $\cS^{\cJ}$. As we recall in 
Lemma~\ref{lem:Barratt-Eccles-action}, the fact that 
$\cJ$ is permutative implies that $B\cJ$ is an $\mathbb E$-algebra, so 
the standard model structure on the category $\cS[\mathbb E]$ of 
$\mathbb E$-algebras in $\cS$ lifts to a model structure on the category $\cS[\mathbb E]/B\cJ$ of $\mathbb E$-algebras over $B\cJ$. We shall refer to this as the standard model structure on $\cS[\mathbb E]/B\cJ$.

\begin{theorem}
There is a chain of Quillen equivalences relating the positive projective 
$\cJ$-model structure on $\cC\cS^{\cJ}$ to the standard model structure on 
$\cS[\mathbb E]/B\cJ$.
\end{theorem}
\begin{proof}
This is a special case of Theorem~\ref{thm:comm-graded-theorem}.
\end{proof}
For a commutative $\cJ$-space monoid $A$, the homotopy colimit $A_{h\cJ}$ is canonically an $\mathbb E$-algebra. On the level of homotopy categories the above adjunction takes $A$ to the induced map of $\mathbb E$-algebras 
$A_{h\cJ}\to B\cJ$. 

\begin{lemma}\label{lem:BJ-homotopy-cartesian}
If $X$ is a fibrant $\cJ$-space, then the commutative square 
\[
\xymatrix@-1pc{
X(\bld n_1,\bld n_2) \ar[r]\ar[d] & X_{h\cJ}\ar[d] \\
\{(\bld n_1,\bld n_2)\} \ar[r] &B\cJ
}
\]
is homotopy cartesian for every object $(\bld n_1,\bld n_2)$ in $\cJ$.
If $X$ is positive fibrant, then the square is homotopy cartesian when
$|\bld n_1|\geq 1$. 
\end{lemma}
\begin{proof}
We first prove the result in the simplicial setting. Since by definition a fibrant 
$\cJ$-space is homotopy constant, the first statement is an immediate consequence of \cite[IV Lemma 5.7]{Goerss_J-simplicial}. Using that $\cJ_+$ is homotopy cofinal in $\cJ$, a similar argument gives the  second statement in the lemma. The topological versions of these statements can be reduced to the simplicial versions by applying the singular complex functor; see Remark \ref{rem:puppe-remark} for details.
\end{proof}

Let us write $\{\pm1\}$ for the group with two elements. 
We know from Corollary~\ref{cor:BPQ-corollary} that $\pi_1(B\cJ,*)$ is isomorphic to $\{\pm1\}$ for every choice of base point $*$. It will be convenient to have an explicit description of this isomorphism. If we view 
$\{\pm1\}$ as a symmetric monoidal category with a single object, then the sign function defines a symmetric monoidal functor $\sgn\colon\Sigma\to \{\pm1\}$ and therefore a functor 
\begin{equation}\label{eq:sgn}
\sgn\colon\cJ\cong\Sigma^{-1}\Sigma\to\{\pm1\}^{-1}\{\pm1\}\cong \{\pm1\}.
\end{equation}
In particular, an endomorphism $(\alpha_1,\alpha_2)$ of an object 
$(\bld n_1,\bld n_2)$ in $\cJ$ is mapped to $\sgn(\alpha_2)\sgn(\alpha_1^{-1})$. Thus, if $|\bld n_1|\geq 2$ or $|\bld n_2|\geq 2$, then we can represent the non-trivial element of $\pi_1(B\cJ,(\bld n_1,\bld n_2))$ by any endomorphism $(\alpha_1,\alpha_2)$ such that 
$\sgn(\alpha_1)=-\sgn(\alpha_2)$.

\begin{corollary}\label{cor:sign-action} 
If $X$ is a fibrant (respectively a positive fibrant) $\cJ$-space, then 
$\pi_0(X(\bld n_1,\bld n_2))$ has a canonical action of $\pi_1(B\cJ,(\bld n_1,\bld n_2))$ for all objects $(\bld n_1,\bld n_2)$ (respectively for all objects such that $|\bld n_1|\geq 1$). An element of $\pi_1(B\cJ,(\bld n_1,\bld n_2))$ represented by an endomorphism $(\alpha_1,\alpha_2)$ acts as 
$\pi_0(X(\alpha_1,\alpha_2))$.\qed 
\end{corollary}

\subsection{Components and units of $\cJ$-space monoids}
\label{subsec:J-components}
Our first task is to decide what kind of object the components of a $\cJ$-space should be. 

\begin{definition}\label{def:graded-signed-monoid}
A graded signed monoid $M$ is a collection of $\{\pm 1\}$-sets $M_t$ for 
$t \in \mathbb Z$, together with a unit $e\in M_0$ and maps $\mu_{s,t}\colon M_s\times M_t\to M_{s+t}$ for all $s,t\in \mathbb Z$. The multiplication maps $\mu_{s,t}$ are assumed to be associative, unital, and 
$(\{\pm1\}\!\times\!\{\pm1\})$-equivariant, where $\{\pm1\}\!\times\!\{\pm1\}$ acts on $M_{s+t}$ through the product. We say that $M$ is graded commutative if $\mu_{s,t}(a,b)=(-1)^{st}\mu_{t,s}(b,a)$ for all $a\in M_s$ and $b\in M_t$.
\end{definition} 

This notion is placed in the general context of diagram categories as
follows. Let $\tilde \cJ$ be the product category $\{\pm1\}\times
\mathbb Z$ where we view $\mathbb Z$ as a discrete category with only
identity morphisms and $\{\pm1\}$ as a category with a single object.
We view $\{\pm1\}\times \mathbb Z$ as a product of monoidal categories
and define a symmetric monoidal structure on $\tilde \cJ$ by
specifying the isomorphisms $(-1)^{mn}\colon m+n\to n+m$. The
symmetric monoidal structure of $\tilde\cJ$ induces a symmetric
monoidal structure on the category of set valued $\tilde\cJ$-diagrams
and a monoid $M$ in this diagram category is the same thing as a
graded signed monoid as defined above. Moreover, $M$ is commutative as
a $\tilde\cJ$-diagram monoid if and only if it is graded commutative
as a graded signed monoid.

Next observe that there is a functor $\Sgn\colon \cJ\to\tilde\cJ$
which on objects takes $(\bld n_1,\bld n_2)$ to the integer $n_2-n_1$
and on morphisms is given by the functor $\sgn$ in \eqref{eq:sgn}.
This becomes a strong symmetric monoidal functor when we specify the
isomorphisms
\[
(-1)^{m_1(n_2-n_1)}\colon \Sgn(\bld m_1,\bld m_2)+\Sgn(\bld n_1,\bld n_2)\to 
\Sgn(\bld m_1\sqcup\bld n_1,\bld m_2\sqcup\bld n_2).
\]
For this to work it is important that we have defined $\Sigma^{-1}\Sigma$ (and hence $\sgn$ in~\eqref{eq:sgn}) using the monoidal right action of $\Sigma$ on itself, cf.\ the proof of Proposition \ref{prop:Grayson-Quillen-J}. Our sign conventions are motivated by the comparison to homotopy groups of symmetric ring spectra in Proposition~\ref{prop:realizing-homotopy-groups} below.

The connected components of $\cJ$ are the full subcategories $\cJ_t$,
for $t\in \mathbb Z$, with objects $(\bld n_1,\bld n_2)$ such that
$n_2-n_1=t$. Let $\cN_t$ be the subcategory of $\cJ_t$ with the same
objects and morphisms $(\iota_1,\iota_2,\chi)$, where $\iota_i$ is a subset inclusion  of the form $\iota_i\colon \bld n_i\to\bld n_i \sqcup \bld1$ and $\chi$ is the unique bijection identifying the
complements (thus, $\mathcal N_t$ is isomorphic to the ordered set of
natural numbers).  Given a $\cJ$-space $X$ we define $\pi_{0,t}(X)$ to
be the set
\[
\pi_{0,t}(X)=\colim_{\cN_t}\{ \dots\to \pi_0(X(\bld n_1,\bld n_2))\to
\pi_0(X(\bld n_1\sqcup \bld 1,\bld n_2 \sqcup \bld 1))\to\dots \}
\]
and write $\pi_0(X)$ for the $\mathbb Z$-graded set
$\{\pi_{0,t}(X)\colon t\in \mathbb Z\}$. Recall from Corollary
\ref{cor:sign-action} that if $X$ is positive fibrant, then $\pi_0(X(\bld
n_1,\bld n_2))$ has a canonical $\{\pm 1\}$-action for $|\bld n_1|\geq 1$. This gives rise to
a $\{\pm 1\}$-action on $\pi_{0,t}(X)$ for each $t$ such that
$\pi_0(X)$ defines a $\tilde \cJ$-diagram.

Now suppose that $A$ is a (positive) fibrant $\cJ$-space monoid. Then
we claim that $\pi_0(A)$ has a uniquely determined structure as a
graded signed monoid such that the canonical map $A\to \pi_0(A)\circ
\Sgn$ is a map of $\cJ$-diagram monoids. Indeed, it easily follows
from the definitions that the maps
\[
\pi_0(A(\bld m_1,\bld m_2))\times\pi_0(A(\bld n_1,\bld n_2))
\xr{(-1)^{m_1(n_2-n_1)}\mu}
\pi_0(A(\bld m_1\sqcup \bld n_1,\bld m_2\sqcup \bld n_2))
\]
(where $\mu$ denotes the multiplication in $A$) give rise to the
required multiplication maps
$\pi_{0,s}(A)\times\pi_{0,t}(A)\to\pi_{0,s+t}(A)$. The unit is
represented by the image in $\pi_0(A(\bld 0,\bld 0))$ of the monoidal
unit $*\to A(\bld 0,\bld 0)$. 

We summarize the properties of $\pi_0(A)$ in the next proposition.
\begin{proposition}
If $A$ is a (positive) fibrant $\cJ$-space monoid, then $\pi_0(A)$ inherits the structure of a graded signed monoid. If $A$ is commutative, then $\pi_0(A)$ is graded commutative. \qed
\end{proposition}

Since for a commutative $\cJ$-space monoid $A$ the homotopy colimit
$A_{h\cJ}$ is an $E_{\infty}$ space, it is clear that the monoid of components 
$\pi_0(A_{h\cJ})$ is commutative. In general, this monoid has to be different from the underlying ungraded monoid of $\pi_0(A)$ because graded commutativity does not become commutativity when forgetting the grading. 
By Lemma \ref{lem:BJ-homotopy-cartesian} and Corollary \ref{cor:sign-action} 
we have the following description. 
\begin{corollary}\label{cor:J-hocolim-components}
Let $A$ be a $\cJ$-space monoid and $A\to \bar A$ a (positive) fibrant replacement. Then the monoid $\pi_0(A_{h\cJ})$ is isomorphic to the quotient of the underlying ungraded monoid of $\pi_0(\bar A)$ by the action of $\{\pm1\}$. \qed
\end{corollary}

The $\cJ$-space monoid of units $A^{\times}$ associated to a $\cJ$-space monoid $A$ is defined by choosing a fibrant replacement $A\to \bar A$ and letting $A^{\times}(\bld n_1,\bld n_2)$ be the union of the components in $A(\bld n_1,\bld n_2)$ that represent units in the graded signed monoid 
$\pi_0(\bar A)$. It is easy to see that $A^{\times}(\mathbf n_1, \mathbf n_2)$ is independent of the choice of fibrant replacement. In order to see that this definition of $A^{\times}$ actually produces a $\cJ$-space monoid it is convenient to give an equivalent description of 
$A^{\times}(\mathbf n_1, \mathbf n_2)$. Notice first that 
$\pi_0(A_{h\cJ})$ is naturally isomorphic to
$\colim_{(\mathbf n_1,\mathbf n_2)\in\cJ}\pi_0(A(\mathbf n_1,\mathbf n_2))$ (the analogous statement holds for any diagram space indexed by a small category). By Corollary~\ref{cor:J-hocolim-components} we can therefore characterize $A^{\times}(\mathbf n_1, \mathbf n_2)$ as the union of the components in $A(\mathbf n_1, \mathbf n_2)$ that represent units in 
$\pi_0(A_{h\cJ})$. It is now clear that $A^{\times}$ has a unique $\cJ$-space monoid structure such that $A^{\times}\to A$ is a fibration of $\cJ$-space monoids. If $A$ is commutative, then so is $A^{\times}$.

\begin{lemma}\label{lem:unit-inclusion}
If $A$ is a (positive) fibrant $\cJ$-space monoid, then $A^{\times}\to A$  realizes the inclusion of units $\pi_0(A)^{\times}\to\pi_0(A)$.\qed
\end{lemma}

\begin{lemma}
Let $A$ be a $\cJ$-space monoid and let $A\to \bar A$ be a fibrant replacement. Then the following conditions are equivalent: (i) $\pi_0(A_{h\cJ})$ is a group, (ii) $\pi_0(\bar A)$ is a group, and (iii) $A^{\times}=A$.
\end{lemma}
\begin{proof}
By Corollary \ref{cor:J-hocolim-components}  (i) implies (ii), by definition, (ii) implies (iii), and by the second description of $A^{\times}$ (iii) implies (i).  
\end{proof}
A $\cJ$-space monoid $A$ is said to 
be \emph{grouplike} if it satisfies the equivalent conditions in the above lemma. 

\begin{proposition}\label{prop:J-grouplike-adjunction}
The functor $A\mapsto A^{\times}$ from $\cJ$-space monoids to grouplike 
$\cJ$-space monoids is a right adjoint of the inclusion functor. 
\end{proposition}
\begin{proof}
Let $A\to \bar A$ be a fibrant replacement and notice that the induced map of units $A^{\times}\to \bar A^{\times}$ is the pullback along the inclusion $\bar A^{\times}\to \bar A$. This implies that $A^{\times}\to\bar A^{\times}$ is also a fibrant replacement and hence, by Lemma~\ref{lem:unit-inclusion}, that 
$A^{\times}$ is indeed grouplike. 
Consequently $(A^{\times})^{\times}=A^{\times}$ which gives the adjunction statement. 
\end{proof}

\subsection{$\cJ$-spaces and symmetric spectra}
An ordinary $\mathbb Z$-graded ring has an underlying graded signed monoid obtained by forgetting the additive structure except for the action of $\{\pm 1\}$ in each degree. 
There are pairs of adjoint functors
\[
\left\{\begin{gathered}
\text{graded (comm.)}\\
\text{signed groups}
\end{gathered}\right\}
\rightleftarrows
\left\{\begin{gathered}
\text{graded (comm.)}\\
\text{signed monoids}
\end{gathered}\right\}
\rightleftarrows
\left\{\begin{gathered}
\text{graded (comm.)}\\
\text{rings}
\end{gathered}\right\}.
\]
and we wish to model a topological version of this using $\cJ$-space monoids.

The category $\cJ$ is designed to be an organizing device for maps
between all the free symmetric spectra on spheres 
$F_{n_1}(S^{n_2})$, just as $\cI$ encodes maps between the symmetric spectra $F_n(S^n)$ for different 
$n$. By adjunction, maps of symmetric spectra 
$F_{n_1}(S^{n_2})\to F_{m_1}(S^{m_2})$ are in 
one-to-one correspondence with maps $S^{n_2} \to F_{m_1}(S^{m_2})_{n_1}$. For a morphism
  $(\beta_1,\beta_2,\sigma)\colon(\bld{m_1},\bld{m_2})\to(\bld{n_1},\bld{n_2})$ in $\cJ$, let
  \begin{equation}\label{eq:FS-induced-maps} (\beta_1,\beta_2,\sigma)^*
    \colon F_{n_1}(S^{n_2})\to F_{m_1}(S^{m_2})
  \end{equation} be the map that is adjoint to
  \[ S^{\bld{n_2}} \xleftarrow{\iso} S^{\bld{m_2}} \sm S^{\bld{n_2}-\beta_2}
  \xleftarrow{\iso} S^{\bld{m_2}}\sm S^{\bld{n_1}-\beta_1}
  \hookrightarrow
  \bigvee_{\beta\colon\bld{m_1}\to\bld{n_1}}S^{\bld{m_2}}\sm
  S^{\bld{n_1}-\beta}.\]
The first map is the isomorphism \eqref{eq:reindexing-basic} induced by 
$\beta_2$, the second is the isomorphism induced by $\sigma$, and the last map is the inclusion of $S^{\bld m_2}\wedge S^{\bld n_1-\beta_1}$ as the wedge summand indexed by $\beta_1$.

\begin{lemma} \label{lem:J-and-free-spsym}
  With the maps defined in \eqref{eq:FS-induced-maps},
  \[ F_{-}(S^{-})\colon \cJ^{\op} \to \Spsym, \quad
  (\bld{n_1},\bld{n_2})\mapsto F_{n_1}(S^{n_2}) \] is a strong symmetric monoidal functor.
\end{lemma}

\begin{proof} 
We first check that $F_{-}(S^{-})$ actually defines a functor. Consider a morphism $(\beta_1,\beta_2,\sigma)
\colon(\bld{m_1},\bld{m_2})\to(\bld{n_1},\bld{n_2})$ and the induced map
\[
(\beta_1,\beta_2,\sigma)^*\colon \bigvee_{\gamma\colon\bld{n_1}\to \bld p}
S^{\bld{n_2}}\wedge S^{\bld p-\gamma}\to\bigvee_{\delta\colon\bld{m_1}\to
\bld p}S^{\bld{m_2}}\wedge S^{\bld p-\delta}
\]
in spectrum degree $p$. This map takes the wedge summand $S^{\bld{n_2}}\wedge S^{\bld p-\gamma}$ indexed by 
$\gamma\colon\bld{n_1}\to \bld{p}$ to the wedge summand $S^{\bld{m_2}}\wedge S^{\bld p-\gamma\beta_1}$ indexed by $\gamma\beta_1\colon\bld{m_1}\to \bld p$. On these wedge summands
the isomorphism 
$S^{\bld{n_2}}\wedge S^{\bld p-\gamma}\to S^{\bld{m_2}}\wedge 
S^{\bld p-\gamma\beta_1}$ is then induced by the chain of bijections 
\[
\bld{n_2}\concat(\bld p-\gamma) \ot \bld{m_2}\concat(\bld{n_2}-\beta_2)\concat (\bld p-\gamma)\ot\bld{m_2}\concat (\bld{n_1}-\beta_1)\concat(\bld p-\gamma)\to\bld{m_2}\concat(\bld p-\gamma\beta_1).
\]
The first bijection is induced by $\beta_2$ (as in \eqref{eq:reindexing-basic}), the second by $\sigma$, and the third by $\gamma$ (as in \eqref{eq:reindexing-composite}). For a composable pair of morphisms
\[
(\bld{l_1},\bld{l_2})\xr{(\alpha_1,\alpha_2,\rho)}
(\bld{m_1},\bld{m_2})\xr{(\beta_1,\beta_2,\sigma)}(\bld{n_1},\bld{n_2})
\]
there is a commutative diagram of bijections
\[
\xymatrix{
& \bld{m_2}\concat(\bld{p}-\gamma\beta_1)\ar[dr]^{(\alpha_1,\alpha_2,\rho)^*} 
& \\
\bld{n_2}\concat(\bld{p}-\gamma)\ar[ur]^{(\beta_1,\beta_1,\sigma)^*}
\ar[rr]^{(\beta_1\alpha_1,\beta_2\alpha_2,
\sigma\cup \beta_2\rho\beta_1^{-1})^*}& & 
\bld{l_2}\concat(\bld{p}-\gamma\beta_1\alpha_1)
}
\] 
which gives the functoriality of $F_{-}(S^{-})$. The required isomorphism making $F_{-}(S^{-})$ a strong symmetric monoidal functor is provided by
\eqref{eq:smash_of_free_iso}. Checking from the explicit descriptions of the relevant maps given above and in \eqref{eq:smash_of_free_iso}, one sees that this is indeed a natural transformation when both sides are viewed as functors on $\cJ\times \cJ$. Furthermore, it follows from the explicit 
description in \eqref{eq:symmetry_iso_spsym} that the isomorphism 
\eqref{eq:smash_of_free_iso} is compatible with the symmetry isomorphisms for $\cJ$ and $\Spsym$. 
\end{proof}

Let $X$ be a $\cJ$-space and let 
$\mathbb S^{\cJ}[X]$ be the symmetric spectrum defined as the coend of the ($\cJ^{\op}\times\cJ$)-diagram $F_{m_1}(S^{m_2})\wedge 
X(\bld n_1,\bld n_2)_+$. Given a symmetric spectrum $E$ we write 
$\Omega^{\cJ}(E)$ for the $\cJ$-space defined by 
$\Map_{\Spsym}(F_{-}(S^{-}),E)$. 
An application of Proposition~\ref{prop:structured-adjunction-SK-Spsym} to the
strong symmetric monoidal functor $F_{-}(S^{-})$  provides the two adjoint
pairs of functors
\begin{equation}\label{eq:cScJ-Spsym-adjunction-copied}
  \mS^{\cJ}[-] \colon
  \cS^{\cJ} \rightleftarrows \Spsym \colon
  \Omega^{\cJ}
  \qquad \text{and}   \qquad 
  \mS^{\cJ}[-] \colon \cC\cS^{\cJ} \rightleftarrows \cC\Spsym \colon
  \Omega^{\cJ}.
\end{equation}
Checking from the definitions, we find that 
\begin{equation}
\mS^{\cJ}[X]_n=\bigvee_{k\geq 0} S^k\wedge_{\Sigma_k}
X(\bld n,\bld k)_+
\quad\textrm{ and } \quad \Omega^{\cJ}(E)(\bld{n_1},\bld{n_2})=\Omega^{n_2}(E_{n_1}).
\end{equation}

The next result can be deduced from the general criterion in 
Proposition~\ref{prop:K-spaces-Spsym-adjunction} in the same way as the $\cI$-space analogue in Proposition~\ref{prop:cScI-Spsym-adjunction}. 
\begin{proposition}
\begin{enumerate}[(i)]
\item
The first adjunction in \eqref{eq:cScJ-Spsym-adjunction-copied} is a Quillen adjunction with respect to the (positive) projective $\cJ$-model structure on 
$\cS^{\cJ}$ and the (positive) projective stable model structure on $\Spsym$.
\item
The second adjunction in \eqref{eq:cScJ-Spsym-adjunction-copied} is a Quillen adjunction with respect to the positive projective $\cJ$-model structure on $\cC\cS^{\cJ}$ and the positive projective stable model structure on 
$\cC\Spsym$.
\qed
\end{enumerate}
\end{proposition}

Our definition of the signed monoid of components associated to a 
positive fibrant $\cJ$-space monoid is partly motivated by the next result.
 
\begin{proposition}\label{prop:realizing-homotopy-groups}
  Let $R$ be a (positive) fibrant symmetric ring spectrum. Then the
  underlying graded signed monoid of the ring $\pi_*(R)$ is isomorphic
  to $\pi_0(\Omega^{\cJ}(R))$.
\end{proposition}
\begin{proof}
For $p \in \mZ$ the $p$th stable homotopy group of $R$ is 
$\pi_p(R) = \colim_{u} \pi_{p+u}(R_u)$
where $p+u \geq 2$ and the colimit is taken over the 
maps 
\[ \pi_{p+u}(R_u) \to \pi_{p+u+1} (R_u \sm S^1) \to \pi_{p+u+1}
(R_{u+1}). \] Setting $m_1 = u$ and $m_2 = p+u$, the chain of
isomorphisms
\[ \pi_{p+u}(R_u) \iso \pi_0(\Omega^{p+u}R_u) \iso
\pi_0(\Omega^{\cJ}(R)(\bld{m_1},\bld{m_2})) \] identifies the terms in
the colimit system defining $\pi_p(R)$ with the terms in the colimit system
defining $\pi_{0,p}(\Omega^{\cJ}(R))$. Since the isomorphisms are
compatible with the structure maps in the colimit system, we get an
isomorphism $\pi_{p}(R) \to \pi_{0,p}(\Omega^{\cJ}(R))$.

Because $R$ is positive fibrant, it is in particular semistable,
and~\cite[4.1 Theorem]{Schwede-homotopy_sym} implies that the action
of $\sigma \in\Sigma_u$ on $\pi_{p+u} (R_u)$ induced by the
$\Sigma_u$-action on $R_u$ coincides with the action of
$\sgn(\sigma)$ on the abelian group $\pi_{p+u}(R_u)$.  This
implies that the $\{\pm1\}$-actions on $\pi_*(R)$ and
$\pi_0(\Omega^{\cJ}(R))$ coincide.

Let  $[x]\in \pi_p(R)$ and $[y]\in\pi_q(R)$ be represented by maps 
$x \colon S^{p+u} \to R_u$ and $y \colon S^{q+v} \to R_v$. It is shown
in~\cite[Proposition I.6.21]{Schwede-SymSp} that the product $[x][y] \in \pi_{p+q}(R)$ is 
represented by $(-1)^{uq}[x\sm y]$, where $x\sm y$ is the
map $S^{p+u+q+v}\to R_{u+v}$ obtained by smashing the representatives
and using the multiplication in $R$.  
Setting $m_1 = u, m_2 = p+u, n_1 = v,$ and $n_2 = q+v$, the above 
isomorphism maps $(-1)^{uq}[x\sm y]$ to $(-1)^{m_1(n_2-n_1)}\mu([x],[y])$.
This proves that the products coincide.
\end{proof}

\begin{definition}
  Let $R$ be a symmetric ring spectrum.
  The \emph{graded units} of $R$ is the grouplike $\cJ$-space monoid 
  $\GLoneJ(R) = (\Omega^{\cJ}R)^{\times}$.
\end{definition}
The functor $R\mapsto \GLoneJ(R)$ is defined for all symmetric ring spectra $R$ and provides a right adjoint of the functor that to a (commutative) grouplike $\cJ$-space monoid $A$ associates the (commutative) symmetric ring spectrum $\mathbb S^{\cJ}[A]$. However, as for the 
$\cI$-space units, $\GLoneJ(R)$ only represents the ``correct'' homotopy type when $R$ is positive fibrant.
In this case we have the following consequence of Lemma \ref{lem:unit-inclusion} and Proposition  \ref{prop:realizing-homotopy-groups}.

\begin{proposition}
Let $R$ be a (positive) fibrant symmetric ring spectrum. Then the map of graded signed monoids $\pi_0(\GLoneJ(R))\to \pi_0(\Omega^{\cJ}(R))$ realizes the inclusion of graded units $\pi_*(R)^{\times}\to \pi_*(R)$. \qed 
\end{proposition}

\subsection{The flat model structure on $\cJ$-spaces}
\label{subsec:flat-J-model}
We here introduce the $\cJ$-space analogue of the flat model structure on 
$\cI$-spaces. Given an object $(\bld n_1,\bld n_2)$ in $\cJ$, let
$\partial(\cJ\!\downarrow(\bld n_1,\bld n_2))$ be the full subcategory of the comma category $(\cJ\!\downarrow(\bld n_1,\bld n_2))$ obtained by excluding the objects $(\bld m_1,\bld m_2)\to(\bld n_1,\bld n_2)$ that are isomorphisms in $\cJ$. The $(\bld n_1,\bld n_2)$th latching space $L_{(\bld n_1,\bld n_2)}(X)$ of a $\cJ$-space $X$ is the colimit of the $\partial(\cJ\!\downarrow(\bld n_1,\bld n_2))$-diagram obtained by composing $X$ with the forgetful functor from $\partial(\cJ\downarrow(\bld n_1,\bld n_2))$ to $\cJ$. A map of $\cJ$-spaces $X\to Y$ is a \emph{flat cofibration} if the induced map 
\[
X(\bld n_1,\bld n_2)\cup_{L_{(\bld n_1,\bld n_2)}(X)}
L_{(\bld n_1,\bld n_2)}(Y)\to Y(\bld n_1,\bld n_2)
\] 
is a cofibration in $\cS$ for all objects $(\bld n_1,\bld n_2)$. It is a 
\emph{positive flat cofibration} if in addition $X(\bld 0,\bld n_2)\to
Y(\bld 0,\bld n_2)$ is an isomorphism for all objects $\bld n_2$ in $\cI$.
We say that a map of $\cJ$-spaces is a (positive) flat $\cJ$-fibration if it has the right lifting property with respect to the class of (positive) flat cofibrations that are $\cJ$-equivalences. As in the case of $\cI$-spaces it is a consequence of our general theory of diagram spaces that together with the 
$\cJ$-equivalences these classes of maps specify a cofibrantly generated simplicial model structure which is proper, monoidal, and satisfies the monoid axiom. We refer to this as the (positive) flat model structure and the cofibrant objects will be called flat $\cJ$-spaces. By 
Proposition~\ref{prop:well-structured-comparison} the identity functor on 
$\cS^{\cJ}$ is the left adjoint in a Quillen equivalence from the (positive) projective $\cJ$-model structure to the (positive) flat $\cJ$-model structure. The next proposition can be derived from our general theory of diagram spaces in the same way as the $\cI$-space analogue in 
Proposition~\ref{prop:I-underlying-flat}. 

\begin{proposition}\label{prop:J-underlying-flat}
\begin{enumerate}[(i)]
\item
The positive flat model structure on $\cS^{\cJ}$ lifts to a cofibrantly generated proper simplicial model structure on $\cC\cS^{\cJ}$. 
\item
Suppose that $A$ is a commutative $\cJ$-space monoid which is cofibrant in the lifted model structure in (i). Then the underlying $\cJ$-space of $A$ is flat. 
\qed
\end{enumerate}
\end{proposition}
By Proposition \ref{prop:boxtimes-flat-invariance} this has the following important implication: If a commutative $\cJ$-space monoid $A$ is cofibrant in the positive flat or projective $\cJ$-model structures on $\cC\cS^{\cJ}$, then the endofunctor 
$A\boxtimes(-)$ preserves $\cJ$-equivalences. 

\begin{remark}
The adjunctions in \eqref{eq:cScJ-Spsym-adjunction-copied} fail to be Quillen adjunctions with respect to the flat model structures. The reason is that the action of $\Sigma_{n_1}\times\Sigma_{n_2}$ on $F_{n_1}(S^{n_2})$ is not free so that condition (ii) in Proposition~\ref{prop:K-spaces-Spsym-adjunction} does not hold. However, there is a ``semi-flat'' model structure on $\cS^{\cJ}$ which is compatible with the flat model structure on 
$\Spsym$: In the notation of Proposition~\ref{prop:K-model-str}
this is defined by letting the subcategory of automorphisms $\cA$ be the subgroups $\Sigma_{n_1}\times\{1_{n_2}\}$ of 
$\Sigma_{n_1}\times\Sigma_{n_2}$.
\end{remark}

\subsection{An application to topological logarithmic structures}
Commutative $\cJ$-space monoids can be used to define a graded version
of John Rognes' notion of topological logarithmic structures
introduced in~\cite{Rognes-TLS}. We explain how this can be done for
the basic definitions by an almost verbatim translation of Rognes'
terminology to the context of $\cJ$-spaces. One advantage of graded
logarithmic structures on commutative symmetric ring spectra is that
they enable us to see the difference between a periodic ring spectrum
and its connective cover. For motivation and background, we refer the
reader to \cite{Rognes-TLS}. More results about the graded log
structures introduced here can be found
in~\cite{Sagave_log-on-k-theory}.

We start by introducing the graded analogue of a pre-log symmetric ring
spectrum~\cite[Definition 7.1]{Rognes-TLS}:
\begin{definition}
  Let $A$ be a commutative symmetric ring spectrum. A \emph{graded
    pre-log structure} on $A$ is a pair $(M,\alpha)$ consisting of a
  commutative $\cJ$-space monoid $M$ and a map
  $\alpha \colon M \to \OmegaJ(A)$ of commutative $\cJ$-space
  monoids. A \emph{graded pre-log symmetric ring spectrum}
  $(A,M,\alpha)$ is a commutative symmetric ring spectrum $A$ with a
  graded pre-log structure $(M,\alpha)$. A map $(f,f^{\flat}) \colon
  (A,M,\alpha) \to (B,N,\beta)$ of graded pre-log symmetric ring
  spectra consists of a map $f\colon A \to B$ of commutative symmetric
  ring spectra and a map $f^{\flat}\colon M \to N$ of commutative
  $\cJ$-space monoids such that $\OmegaJ(f) \alpha = \beta
  f^{\flat}$.
\end{definition}

\begin{example}
  Let $A$ be a commutative symmetric ring spectrum and let $x$ be a
  point in $\OmegaJ(A)(\bld{n}_1,\bld{n}_2)$. By adjunction, $x$ gives
  rise to a map
\[\alpha \colon M = \mC F_{(\bld{n}_1,\bld{n}_2)}^{\cJ}(\pt) = \textstyle
\coprod_{i\geq 0}
\big(F_{(\bld{n}_1,\bld{n}_2)}^{\cJ}(\pt)\big)^{\boxtimes i}/\Sigma_i \to
\OmegaJ(A).\] from the free commutative $\cJ$-space monoid on a point
in degree $(\bld{n_1},\bld{n_2})$. We refer to $(M,\alpha)$ as the
free graded pre-log structure generated by $x$.
\end{example}

For a graded pre-log symmetric ring spectrum $(A,M,\alpha)$, the
commutative $\cJ$-space monoid $\alpha^{-1}(\GLoneJ(A))$ is defined
by the pullback diagram
\[\xymatrix@-1pc{
\alpha^{-1}(\GLoneJ(A)) 
\ar[d]_{\widetilde{\iota}}
\ar[r]^-{\widetilde{\alpha}}& 
\GLoneJ(A)\ar[d]_{\iota}\\
M \ar[r]^-{\alpha}& \OmegaJ(A).}\] 
This enables us to state the analogue of~\cite[Definition 7.4]{Rognes-TLS}:

\begin{definition}
  A graded pre-log structure $(M,\alpha)$ on $A$ is a \emph{graded
    log structure} if the induced map $\widetilde{\alpha}\colon
  \alpha^{-1}(\GLoneJ(A)) \to \GLoneJ(A)$ is a $\cJ$-equivalence.  A
  graded pre-log symmetric ring spectrum $(A,M,\alpha)$ is a {\em
    graded log symmetric ring spectrum} if $(M,\alpha)$ is a graded
  log structure.
\end{definition}

The basic example of a graded log structure on $A$ is the {\em
  trivial graded log structure} $(\GLoneJ(A),\iota)$. The {\em
  logification} of a graded pre-log structure is the pushout $M^{a}$
of the diagram
\[ M \xleftarrow{\widetilde{\iota}} \alpha^{-1}(\GLoneJ(A))
\xrightarrow{\widetilde{\alpha}}\GLoneJ(A) \] together with the induced map
$\alpha^{a} \colon M^{a} \to \OmegaJ(A)$. It comes with a map
$(M,\alpha)\to(M^{a},\alpha^{a})$. As in~\cite[Lemma 7.7]{Rognes-TLS},
one can show that $(M^{a},\alpha^{a})$ is indeed a log structure.

\begin{lemma}\label{lem:criterion-trivial-log}
  Let $(A,M,\alpha)$ be a graded pre-log symmetric ring spectrum. If
  $\alpha$ factors as a composite
  \[ M \xrightarrow{\zeta} \GLoneJ(A) \xrightarrow{\iota} \OmegaJ(A)\]
  with $\zeta$ a map of commutative $\cJ$-space monoids, then
  $(M^{a},\alpha^{a})$ is isomorphic to the trivial log-structure.
\end{lemma}
\begin{proof}
  If $\zeta$ exists, then $\alpha^{-1}(\GLoneJ(A))\iso M$, hence
  $M^a \iso \GLoneJ(A)$. This uses that $\GLoneJ(A) \to
  \OmegaJ(A)$ is an inclusion of path components and hence a
  monomorphism.
\end{proof}

\begin{example}
  Let $KU$ be a positive fibrant model for the periodic complex
  $K$-theory spectrum, and let $f\colon ku \to KU$ be a positive
  fibration exhibiting the connective complex $K$-theory spectrum $ku$
  as the connective cover of $KU$. The Bott class $\pi_2(ku)$ can be
  represented by a point $x \in \OmegaJ(ku)(\bld{1},\bld{3})$. It
  generates a graded pre-log structure
  \[ \alpha \colon M = \mC F_{(\bld{1},\bld{3})}^{\cJ}(\pt) \to
  \OmegaJ(ku)\] which gives rise to a non-trivial log structure on
  $ku$. We compose $\alpha$ with $\OmegaJ(f)$ in order to get the
  induced (inverse image) graded pre-log structure $(f^*M,f^*\alpha)$
  on $KU$. The map $f^* \alpha$ factors through the inclusion of the
  units of $KU$ since the Bott element is invertible in
  $\pi_*(KU)$. By Lemma \ref{lem:criterion-trivial-log}, the
  logification of $(f^*M,f^*\alpha)$ is the trivial log structure.

  In other words, the graded log structure generated by the Bott
  element is trivial on $KU$ and non-trivial on $ku$. This is an
  important feature of graded log structures which can only be
  achieved using the graded units. In the $\cI$-space case, the map
  $\GLoneI(ku) \to \GLoneI(KU)$ is an equivalence, and the Bott class
  being a unit is not detected by $\GLoneI(KU)$. This issue is
  discussed in~\cite[Remark 7.28]{Rognes-TLS}.
\end{example}

\begin{example}
  As pointed out in~\cite[Remark 7.28]{Rognes-TLS}, there is another
  source of interesting graded log structures. Let $E$ be a periodic
  commutative symmetric ring spectrum and let $i\colon e \to E$ be its
  connective cover. On $e$, one can form the \emph{direct image} graded
  log structure $i_*(\GLoneJ(E))$ of the trivial graded log
  structure on $E$. It is defined to be the pullback of
  \[ \OmegaJ(e) \to \OmegaJ(E) \ot \GLoneJ(E).\] 
  This is very much analogous to the situation in algebra, where a
  discrete valuation ring $A$ inherits the log structure $A \setminus
  \{0\}$ as the direct image of the trivial log structure on its
  fraction field, compare for example~\cite[Remark
  2.25]{Rognes-TLS}. Again, forming $i_*(\GLoneJ(E))$ gives
  something non-trivial in the graded case, while the same
  construction in the $\cI$-space case only leads to the trivial log
  structure.
\end{example}

\section{Well-structured index categories}\label{sec:well-structured}
As we have seen in Sections \ref{sec:I-space-section} and \ref{sec:J-space-section} there are several useful model structures on the categories of $\cI$- and $\cJ$-spaces. In order to set up and analyze these model structures in a common framework we here introduce the notion of a \emph{well-structured relative index category}. The main idea can be summarized as follows. Consider in general a small symmetric monoidal category $\cK$ and the category of $\cK$-spaces $\cS^{\cK}$. When defining a model structure 
on $\cS^{\cK}$ we have a choice when specifying the requirements for a map of $\cK$-spaces $X\to Y$ to be a fibration: For each object $\bld k$ in $\cK$ the map $X(\bld k)\to Y(\bld k)$ is equivariant with respect to the action of the endomorphisms of $\bld k$ and we must specify the extend to which these maps are equivariant fibrations. This of course has a dual effect on the cofibrations; the stronger the condition for a map to be a fibration the weaker the condition is to be a cofibration. In practice we shall control this duality by specifying a subcategory $\cA$ of automorphisms in 
$\cK$ and require that the fibrations $X(\bld k)\to Y(\bld k)$ be equivariant with respect to the automorphisms in $\cA$. 

\subsection{Well-structured relative index categories}
Let $(\cK,\sqcup,\bld 0)$ be a small symmetric monoidal category and let $\cA$ be a subcategory of automorphisms. Given an object $\bld k$ in $\cA$ we write $\cA(\bld k)$ for the automorphism group $\cA(\bld k,\bld k)$. We shall always assume that $\cA$ be a \emph{normal subcategory}: for each isomorphism $\alpha\colon \bld k\to\bld l$ in $\cK$ we require that $\bld k$ belongs to $\cA$ if and only if $\bld l$ does, and that in this case conjugation by $\alpha$ gives an isomorphism $\cA(\bld k)\to \cA(\bld l)$ by mapping 
$\gamma$ in $\cA(\bld k)$  to $\alpha\gamma \alpha^{-1}$. We shall also require that 
$\cA$ be \emph{multiplicative} in the sense that  the monoidal structure 
$\sqcup\colon\cK\times \cK\to \cK$ restricts to a functor $\cA\times \cA\to \cA$ (but we do not assume that $\cA$ necessarily contains the unit $\bld 0$ for the monoidal structure). 

Let $\mN_0$ denote the ordered set of natural numbers $0\to1\to2\to\dots $, thought of as a symmetric monoidal category via the additive structure. 

\begin{definition}\label{def:well-structured-relative-index}
  A \emph{well-structured relative index category} is a triple consisting of a 
  symmetric monoidal category $(\cK,\concat, \bld 0)$, a strong
   symmetric monoidal functor $\lambda\colon \cK\to \mN_0$, and a normal and multiplicative subcategory of automorphisms $\cA$ in $\cK$. These data are required to satisfy the following conditions.
  \begin{enumerate}
  \item[(i)]
  A morphism $\bld k\to \bld l$ in $\cK$ is an isomorphism if and only if 
  $\lambda(\bld k)=\lambda(\bld l)$. 
 \item[(ii)]
  For each object $\bld k$ in $\cA$ and each object $\bld l$ in $\cK$, each connected component of the category $(\bld k\concat -\downarrow \bld l)$ has a terminal object.
  \item[(iii)]
  For each object $\bld k$ in $\cA$ and each object $\bld l$ in $\cK$, the canonical right action of 
  $\cA(\bld k)$ on the category $(\bld k\concat -\downarrow \bld l)$ induces a free action on the set of connected components. 
  \item[(iv)]
Let $\cK_{\cA}$ be the full subcategory of $\cK$ generated by the objects in 
$\cA$. We require that the inclusion $\cK_{\cA}\to \cK$ is homotopy cofinal. 
  \end{enumerate}
\end{definition}  
  
Spelling the requirements out in more detail, the condition that $\lambda$ be strong symmetric monoidal means that $\lambda(\bld 0)=0$ and 
$\lambda(\bld k_1\concat \bld k_2)=\lambda(\bld k_1)+\lambda(\bld k_2)$ for all pairs of objects $\bld k_1$ and $\bld k_2$. We think of $\lambda$ as a \emph{degree functor} on $\cK$. The existence of such a degree functor ensures an explicit description of the cofibrations in the model structures we define on $\cS^{\cK}$, see Proposition \ref{prop:latching-characterization}.
In (ii) and (iii) the category $(\bld k\concat -\downarrow\bld l)$ is the comma category whose  objects are pairs $(\bld n,\alpha)$ given by an object $\bld n$ in $\cK$ and a morphism $\alpha\colon \bld k\concat \bld n\to\bld l$ in $\cK$. A morphism $(\bld n,\alpha)\to (\bld n',\alpha')$ is specified by a morphism 
$\gamma\colon \bld n\to \bld n'$ in $\cK$ such that 
$\alpha=\alpha'\circ(1_{\bld k}\concat \gamma)$. For the homotopy cofinality condition (iv) we recall that in general a subcategory $\cB$ of a small category $\cC$ is said to be homotopy cofinal if the comma categories $(c\downarrow \cB)$ have contractible classifying space for each object $c$ in $\cC$.
This is equivalent to the condition that for any $\cC$-diagram $X$ in $\cS$, the canonical map $\hocolim_{\cB}X\to\hocolim_{\cC}X$ is a weak homotopy equivalence, see e.g.\ \cite[Theorem~19.6.13]{Hirschhorn-model}.

The degree functor $\lambda$ on $\cK$ will usually be understood from the 
context and we often use the notation $(\cK,\cA)$ to indicate a well-structured relative index category. Notice that by condition (i) each endomorphism set 
$\cK(\bld k,\bld k)$ is a group of automorphisms. We introduce the notation 
$\cK(\bld k)$ for this group. The normality condition on 
$\cA$ in particular implies that $\cA(\bld k)$ is a normal subgroup of 
$\cK(\bld k)$ for each object $\bld k$ in $\cA$.

We shall later prove that a well-structured relative index category $(\cK,\cA)$ gives rise to an $\cA$-relative $\cK$-model structure on $\cS^{\cK}$ and that this model structure is proper, monoidal, and lifts to the category of structured $\cK$-spaces for any $\Sigma$-free operad. However, in order to lift this model structure to commutative $\cK$-space monoids we need $(\cK,\cA)$ to be \emph{very well-structured} in the following sense. 

\begin{definition}\label{def:very-well-structured}
A well-structured relative index category $(\cK,\cA)$ is \emph{very well-structured} if for each object $\bld k$ in $\cA$, each object $\bld l$ in $\cK$, 
and each $n\geq 1$, the canonical right action of  
$\Sigma_n\ltimes \cA(\bld k)^{\times n}$ on the category 
$(\bld k^{\sqcup n} \sqcup-\downarrow \bld l)$ induces a free action on the set of connected components.
\end{definition} 

Here the group  $\Sigma_n\ltimes \cA(\bld k)^{\times n}$ is the semidirect product of $\Sigma_n$ acting from the right on $\cA(\bld k)^{\times n}$ (also known as the wreath product $\Sigma_n\int \cA(\bld k)$). The action on
$(\bld k^{\sqcup n}\sqcup -\downarrow \bld l)$ is via the homomorphism 
$\Sigma_n\ltimes \cA(\bld k)^{\times n}\to \cK(\bld k^{\bld n})$ that
maps an element $(\sigma;f_1,\dots,f_n)$ with $\sigma\in \Sigma_n$ and 
$f_i\in \cA(\bld k)$ to the composition 
$\sigma_*\circ(f_1\concat\dots\concat f_n)$ where $\sigma_*$ denotes the canonical automorphism of $\bld k^{\concat n}$ determined by the symmetric monoidal structure. 

\begin{remark}
The stronger condition that $\Sigma_n\ltimes \cA(\bld k)^{\times n}$ maps injectively into $\cA(\bld k^{\sqcup n})$ is relevant for whether the forgetful functor from commutative $\cK$-space monoids to $\cK$-spaces preserves cofibrancy. We discuss this in Section~\ref{sec:structured-cofibrant}.
\end{remark}

Specializing to the case where $\cA$ is the discrete subcategory of identity morphisms in $\cK$ we get the notion of a \emph{well-structured index category}. Writing out the details of this we arrive at the following definition. 

\begin{definition}\label{def:well-structured-index}
A \emph{well-structured index category} $\cK$ is a small symmetric monoidal category, equipped with a strong symmetric monoidal functor 
$\lambda\colon\cK\to \mathbb N_0$, such that 
\begin{itemize}
\item
a morphism $\bld k\to \bld l$ in $\cK$ is an isomorphism if and only if 
$\lambda(\bld k)=\lambda(\bld l)$, and
\item
for each pair of objects $\bld k$ and $\bld l$ in $\cK$, each connected component of the category $(\bld k\sqcup-\downarrow\bld l)$ has a terminal object.
\end{itemize}
\end{definition}
 Given a well-structured index category $\cK$ we shall refer to the associated model structure on $\cS^{\cK}$ as the \emph{projective} model structure, see  Definition \ref{def:K-projective-model}. It should be noted that the axioms for the monoidal unit $\bld 0$ in the symmetric monoidal category $\cK$ imply that the homomorphism 
 $\Sigma_n\to \cK(\bld 0^{\sqcup n})$ is trivial for all $n$. Thus, to obtain a very well-structured relative index category and hence a model structure on commutative $\cK$-space monoids, we are forced to specify a subcategory of automorphisms $\cA$ that  does not contain the object $\bld 0$. 

Recall the free functors $F_{\bld k}^{\cK}(*)$ introduced in 
Section~\ref{subs:free-semi-free}.
When analyzing the homotopical properties of the $\boxtimes$-product on 
$\cS^{\cK}$ it will be important to consider $\cK$-spaces of the form 
$F_{\bld k}^{\cK}(*)\boxtimes X$ for an object $\bld k$ and a $\cK$-space 
$X$. The axioms for a well-structured relative index category $\cK$ are partly motivated by the following lemmas.

\begin{lemma}\label{lem:k-Kan-extension}
The $\cK$-space $F_{\bld k}^{\cK}(*)\boxtimes X$ is isomorphic to the left Kan extension of $X$ along the functor $\bld k\concat-\colon \cK\to \cK$. 
\end{lemma}
\begin{proof}
In fact, this result holds for any monoidal category $\cK$. The statement in the lemma means that there is a natural isomorphism 
\begin{equation}\label{eq:left-Kan}
(F_{\bld k}^{\cK}(*)\boxtimes X)(\bld l)\cong 
\colimwithlimits_{\bld k\concat \bld k'\to \bld l}
X(\bld k')
\end{equation}
(where the colimit on the right hand side is over the category $(\bld k\concat -\downarrow \bld l)$) and this is immediate from the universal properties of these constructions.
\end{proof}
Notice in particular, that for $X=*$ the isomorphism \eqref{eq:left-Kan} gives an identification of  $(F_{\bld k}^{\cK}(*)\boxtimes *)(\bld l)$ with the set of connected components of the category $(\bld k\concat -\downarrow\bld l)$.

\begin{lemma}\label{lem:free-right-action}
Let $(\cK,\cA)$ be a well-structured relative index category. Then
the canonical right action of $\cK(\bld k)$ on $F_{\bld k}^{\cK}(*)\boxtimes X$ restricts to a levelwise free $\cA(\bld k)$-action for all objects $\bld k$ in 
$\cA$.
\end{lemma}
\begin{proof}
The projection $X\to *$ onto the terminal $\cK$-space induces a map 
of $\cK$-spaces $F_{\bld k}^{\cK}(*)\boxtimes X\to F_{\bld k}^{\cK}(*)
\boxtimes *$, so it suffices to check that the $\cA(\bld k)$-action on the target is levelwise free. By the observation before the lemma 
$(F_{\bld k}^{\cK}(*)\boxtimes *)(\bld l)$ can be identified with the set of connected components of the category 
$(\bld k\concat -\downarrow \bld l)$, hence the result follows from condition 
(iii) for a well-structured relative index category. 
\end{proof}

The principal examples in this paper are the categories $\cI$ and $\cJ$. 
We define degree functors 
$\lambda\colon \cI\to\mathbb N_0 $ by $\lambda(\bld{k}) =|\bld k|$ and 
$\lambda\colon\cJ\to\mathbb N_0$ by $\lambda(\bld{k}_1,\bld{k}_2)= 
|\bld k_1|$. (Here $|-|$ indicates the cardinality of a finite set). There are other degree functors on $\cJ$ but the above choice is the one that will be relevant for our work. 

\begin{proposition}\label{prop:I-J-well-structured}
Let $\cK$ denote one of the categories $\cI$ or $\cJ$, equipped with the above degree functor. Suppose that $\cA$ is a normal and multiplicative subcategory of automorphisms in $\cK$ such that the inclusion $\cK_{\cA}\to \cK$ is homotopy cofinal. Then $(\cK,\cA)$ is a well-structured relative index category and if all objects of $\cA$ have positive degree, then $(\cK,\cA)$ is very well-structured. 
\end{proposition}
\begin{proof}
We explain the details for $\cJ$; the case of $\cI$ is similar but easier. It
is clear that $\lambda$ is strong symmetric monoidal and that (i) holds. 
For (ii) we first observe that the correspondence $(\alpha_1,\alpha_2,\rho)\mapsto (\alpha_1|_{\bld k_1},\alpha_2|_{\bld k_2})$ defines a bijection between the set of connected components of the category 
$((\bld k_1,\bld k_2)\sqcup-\downarrow (\bld l_1,\bld l_2))$
and the set $\cI(\bld k_1,\bld l_1)\times \cI(\bld k_2,\bld l_2)$. An object 
$(\alpha_1,\alpha_2,\rho)$ is terminal in its connected component if and only it is an isomorphism in $\cJ$ which implies that (ii) holds. Given an object 
$(\bld k_1,\bld k_2)$ in $\cJ$, the automorphism group can be identified with the product $\Sigma_{|\bld k_1|}\times \Sigma_{|\bld k_2|}$ and it is clear from the above description that this acts freely on the set of connected components in $((\bld k_1,\bld k_2)\sqcup-\downarrow (\bld l_1,\bld l_2))$, hence (iii) holds.
Finally, (iv) holds by the assumption on $\cA$.

Now suppose that the objects in $\cA$ have positive degree. In order to check the condition in Definition \ref{def:very-well-structured} for being very 
well-structured we may as well assume that the automorphism groups in $\cA$ are the full automorphism groups in $\cJ$.  Using the above identification in terms of symmetric groups, the homomorphism from the semidirect product in question to the full automorphism group in $\cJ$ is given by concatenation and block permutation in each factor,
\[
\Sigma_n\ltimes (\Sigma_{|\bld k_1|}^{\times n}\times 
\Sigma_{|\bld k_2|}^{\times n})\to
\Sigma_{n|\bld k_1|}\times \Sigma_{n|\bld k_2|}.
\]
This is injective if $|\bld k_1|>0$, so the conclusion follows from the observation above that the right hand side acts freely on the relevant set of connected components.  
\end{proof}

In general, for a well-structured index category $\cK$, we write $\cK_+$ for the full subcategory whose objects have positive degree and $O\cK_+$ for the corresponding discrete subcategory of identity morphisms.
\begin{corollary}\label{cor:I-J-well-structured}
Let $\cK$ denote one of the categories $\cI$ or $\cJ$. Then $\cK$ is a 
well-structured index category and $(\cK,O\cK_+)$ is very well-structured.
\end{corollary}
\begin{proof}
It remains to check that the inclusion $\cK_+\to \cK$ is homotopy cofinal. Thus, given an object $\bld k$ in $\cK$ we must show that the comma category $(\bld k\downarrow \cK_+)$ has contractible classifying space. Choose a morphism $\alpha\colon\bld 0\to\bld l$ in $\cK$ such that $\bld l$ has positive degree and consider the functor $(\bld k\downarrow\cK)\to 
(\bld k\downarrow \cK_+)$ defined by concatenation with $\alpha$ on objects.
We also have the functor $(\bld k\downarrow \cK_+)\to 
(\bld k\downarrow \cK)$ defined by the inclusion of $\cK_+$ in $\cK$ and it is easy to see that $\alpha$ gives rise to natural transformations between the two compositions of these functors and the respective identity functors. Hence it suffices to show that the classifying space of $(\bld k\downarrow\cK)$ is
contractible and this is clear since the identity on $\bld k$ is an initial object.   
\end{proof}
The above corollary is the underlying reason why the corresponding (positive) projective $\cI$- and $\cJ$-model structures on $\cS^{\cI}$ and $\cS^{\cJ}$ have the pleasant properties stated in Propositions \ref{prop:projective-I-model-structure} and \ref{prop:projective-J-model-structure}. 
As discussed in Sections~\ref{subsec:flat-I-model} and 
\ref{subsec:flat-J-model} it is important for the applications that the projective model structures are accompanied by corresponding \emph{flat} model structures. These flat model structures arise by specifying $\cA$ as the full automorphism subcategory in $\cI$ and $\cJ$, respectively. We write $\Sigma$ for the full automorphism subcategory of $\cI$ (that is, the category of finite sets and 
isomorphisms) and $\Sigma\times \Sigma$ for the full automorphism subcategory of $\cJ$. The corresponding automorphism categories of objects of positive degree are then given by $\Sigma_+$ and 
$\Sigma_+\times \Sigma$. 

\begin{corollary}\label{cor:flat-I-J-well-structured}
With the above notation there are well-structured relative index categories 
$(\cI,\Sigma)$ and 
$(\cJ,\Sigma\times\Sigma)$. Restricting to automorphisms of objects of positive degree we get very well-structured relative index categories 
$(\cI,\Sigma_+)$ and $(\cJ,\Sigma_+\times\Sigma)$. 
\end{corollary}
\begin{proof}
The homotopy cofinality condition (v) follows from the proof of 
Corollary~\ref{cor:I-J-well-structured}. 
\end{proof}

We finally consider the $K$-theory example mentioned in the introduction.

\begin{example}\label{ex:K-R-example}
Let $R$ be a ring with invariant basis number and let $\cF_R$ be the category with objects the free $R$-modules $R^n$ and morphisms the isomorphisms between such modules. This is a permutative category under direct sum and we write 
$\cK_R$ for Quillen's localization construction $\cF_R^{-1}\cF_R$ on this category, 
see \cite{Grayson-higher}. In order to give an explicit description of this category we first introduce the category $\cS\cF_R$ of \emph{free split injections} (in Grayson's 
notation \cite{Grayson-higher} this is the category 
$\langle \cF_R,\cF_R\rangle$). The objects of $\cS\cF_R$ are the 
$R$-modules $R^n$ and a morphism $(f,p)\colon R^m\to R^n$ is a pair of 
$R$-linear maps $f\colon R^m\to R^n$ and $p\colon R^n\to R^m$ such that 
$p\circ f$ is the identity on $R^m$ and the cokernel $R^n/\image(f)$ is free.
The category $\cK_R$ has objects all pairs of free $R$-modules 
$(R^{n_1},R^{n_2})$ and a morphism 
\[
((f_1,p_1),(f_2,p_2),\rho)\colon
(R^{m_1},R^{m_2})\to 
(R^{n_1},R^{n_2})
\] 
is a triple given by morphisms $(f_i,p_i)\colon R^{m_i}\to R^{n_i}$ in 
$\cS\cF_R$ for $i=1,2$, and an isomorphism
$\rho\colon R^{n_1}/\image(f_1)\xr{\sim} R^{n_2}/\image(f_2)$ between the corresponding cokernels. Composition of morphisms is defined in the natural way. 

We define a degree functor $\lambda\colon \cK_R\to \mathbb N_0$ by 
$\lambda(R^{n_1},R^{n_2})=n_1$ and claim that this makes $\cK_R$ a well-structured index category. Since we assume that $R$ has invariant basis number it is clear that a morphism in $\cK_R$ is an isomorphism if and only if the domain and codomain have the same degree (and hence are equal). For the second condition we observe that there is a bijective correspondence
\[
\left\{\text{components of  
$((R^{k_1},R^{k_2})\oplus -\downarrow (R^{l_1},R^{l_2}))$}\right\}\simeq 
\cS\cF_R(R^{k_1},R^{l_1})\times \cS\cF_R(R^{k_2},R^{l_2}).
\]
This takes an object defined by the data $(R^{n_i},(f_i,p_i)\colon R^{k_i}\oplus R^{n_i}\to R^{l_i})$ for $i=1,2$, and 
$\rho\colon R^{l_1}/\image(f_1)\cong R^{l_2}/\image(f_2)$ to the pair of morphisms $((\bar f_1,\bar p_1), (\bar f_2,\bar p_2))$, where $\bar f_i$ is the restriction of $f_i$ to $R^{k_i}$ and $\bar p_i$ is the composition of $p_i$ with the projection of $R^{k_i}\oplus R^{n_i}$ onto $R^{k_i}$. Using this it is easy to see that each connected component has a terminal object, hence $\cK_R$ is indeed a well-structured index category.
\end{example}

\section{Model structures on \texorpdfstring{$\cK$}{K}-spaces}
In this section we introduce the various model structures on diagram spaces associated to well-structured relative index categories. We assume that the reader is familiar with the basic theory of cofibrantly generated model categories as presented in \cite[Chapter 11]{Hirschhorn-model} and 
\cite[Section 2.1]{Hovey-model}.

\subsection{Model structures on \texorpdfstring{$G$}{G}-spaces}
\label{subs:G-space-model}
We review some well-known facts about equivariant homotopy theory.
Recall that the category of spaces 
$\cS$ is cofibrantly generated with generating cofibrations $I$ the set of maps 
$\partial \Delta^n\to \Delta^n$ for $n\geq 0$ and generating acyclic cofibrations $J$ the set of maps $\Lambda^n_i\to \Delta ^n$ for $n>0$ and $0\leq i\leq n$, see 
\cite[Section 11.1]{Hirschhorn-model} for details. Here the notation indicates that the generating (acyclic) cofibrations in the topological setting are obtained from those in the simplicial setting by geometric realization.  

Consider a discrete group $G$ and write $\mathcal S^G$ for the category
of left $G$-spaces. It admits several model structures that are of
interest for us. Fix a normal subgroup $A$ in $G$. We say that a map of $G$-spaces $X\to Y$ is an \emph{$A$-relative weak equivalence} (or an 
\emph{$A$-relative fibration}) if the induced map of fixed points $X^H\to Y^H$ is a weak equivalence (or fibration) for all subgroups $H\subseteq A$. We say that a map of $G$-spaces is an $A$-relative cofibration if it has the left lifting property with respect to maps that are $A$-relative weak equivalences and $A$-relative fibrations. Let $I_{(G,A)}$ be the set of maps in $\cS^G$ of the form $G/H\times i$ for $H\subseteq A$ and $i\in I$, and let $J_{(G,A)}$ be the set of maps of the form $G/H\times j$ for $H\subseteq A$ and $j\in J$. 
The following well-known result is an easy consequence of the recognition principle for cofibrantly generated model categories~\cite[Theorem 2.1.19]{Hovey-model}.

\begin{proposition}
Let $A$ be a normal subgroup of $G$. 
The $A$-relative weak equivalences, fibrations, and cofibrations specify a cofibrantly generated model structure on $\cS^G$ with generating cofibrations $I_{(G,A)}$ and generating acyclic cofibrations $J_{(G,A)}$. \qed
\end{proposition}

We shall refer to this as the \emph{$A$-relative model structure} on $\cS^G$.
When $A$ is the trivial group this is also known as the \emph{coarse} (or weak or naive) model structure and when $A=G$ this is sometimes called the \emph{fine} (or strong or genuine) model structure.  

As observed by Shipley \cite{Shipley-convenient} (in the case $A=G$) it is possible to combine the weak equivalences in the coarse model structure with the cofibrations in the $A$-relative model structure to get a cofibrantly generated \emph{mixed model structure} on $\cS^G$. We recall the details of this construction.
Let $EH$ denote the one-sided bar construction $B(H,H,*)$ of a subgroup 
$H$ in $A$, and  let $\pi_H\colon G\times_HEH\to G/H$ be the projection. We write $M(\pi_H)$ for the mapping cylinder of $\pi_H$ and consider the
standard factorization
\begin{equation}\label{eq:pi_H-factorization}
G\times_HEH\xr{j_H}M(\pi_H)\xr{r_H}G/H
\end{equation}
where $j_H$ is an $A$-relative cofibration and $r_H$ is a $G$-equivariant homotopy equivalence.  Let $J'_{(G,A)}$ be the set of morphisms in
$\mathcal S^G$ of the form $j_H\Box i$, where $j_H$ is as above 
(for $H\subseteq A$), $i\in I$, and $j_H \Box i$ is the pushout-product of
$j_H$ and $i$ (see the remarks preceding   
Proposition~\ref{prop:level-S-model}). 
We write $J^{\textit{mix}}_{(G,A)} =J_{(G,A)}\cup J'_{(G,A)}$.  The significance of the set $J^{\textit{mix}}_{(G,A)}$ is explained in the next lemma which is implicit in the proof of \cite[Proposition 1.3]{Shipley-convenient}. Recall that given a group $H$ and an $H$-space $X$, the homotopy fixed points $X^{hH}$ is the space of equivariant maps $\Map_H(EH, X)$ (which is the same thing as the homotopy limit of $X$ viewed as a diagram over the one-object category $H$).

\begin{lemma}\label{flat-injective-lemma}
  A map $X\to Y$ in $\mathcal S^G$ is $J^{\textit{mix}}_{(G,A)}$-injective if
  and only if the induced maps $X^H\to Y^H$ are fibrations and the
  diagrams
\[
\xymatrix@-1pc{
X^H\ar[r]\ar[d] & X^{hH}\ar[d]\\
Y^H\ar[r] & Y^{hH}
}
\]
are homotopy cartesian for all subgroups $H$ in $A$.\qed
\end{lemma}
A map of $G$-spaces $X\to Y$ is said to be an \emph{$A$-relative mixed fibration} if it satisfies the equivalent conditions of Lemma
\ref{flat-injective-lemma}. Arguing as in the proof of \cite[Proposition 1.3]{Shipley-convenient} we get the \emph{$A$-relative mixed model structure on $\cS^G$}:

\begin{proposition}
The coarse weak equivalences, the $A$-relative mixed fibrations, and the $A$-relative cofibrations specify a cofibrantly generated model structure on 
$\cS^G$ with generating cofibrations $I_{(G,A)}$ and generating acyclic cofibrations $J^{\textit{mix}}_{(G,A)}$. \qed   
\end{proposition}

\begin{remark}
For $A=G$ and  $\cS$ the category of simplicial sets, it is easy to check that the $A$-relative (that is, the \emph{fine}) cofibrations are the maps in 
$\cS^{G}$ whose underlying maps in $\cS$ are cofibrations.
\end{remark}

\subsection{The $\cA$-relative level model structure on
  \texorpdfstring{$\cK$}{K}-spaces}\label{subs:level-model}
  \label{subs:level-model-structures} 
Consider now a well-structured relative index category $(\cK,\cA)$ in the sense of Definition \ref{def:well-structured-relative-index}. Recall that for an object $\bld k$ in $\cK$ we write $\cK(\bld k)$ and $\cA(\bld k)$ for the automorphism groups of $\bld k$ in $\cK$ and $\cA$, respectively. 
We say that a map 
$X \to Y$ of $\cK$-spaces is an \emph{$\cA$-relative level equivalence} if  \mbox{$X(\bld{k})\to Y(\bld{k})$} is a weak equivalence of spaces for every object $\bld{k}$ in $\cA$. (This will not lead to confusion with the notion of an $A$-relative weak equivalence of $G$-spaces introduced in Section 
\ref{subs:G-space-model}.)
A map of $\cK$-spaces $X \to Y$ is an \emph{$\cA$-relative level fibration} 
if for all objects $\bld{k}$ in $\cA$ and all subgroups
$H$ in $\cA(\bld{k})$, the map $X(\bld{k})^H\to Y(\bld{k})^H$
is a fibration and the diagram
  \begin{equation}\label{eq:flat-level-fibration-square}
  \xymatrix@-1pc{
    X(\bld{k})^H\ar[r] \ar[d] &  X(\bld{k})^{hH}\ar[d]\\
    Y(\bld{k})^H\ar[r] & Y(\bld{k})^{hH} }
  \end{equation}
is  homotopy cartesian. Finally, a map of $\cK$-spaces is an $\cA$-relative cofibration if it has the left lifting property with respect to maps that are 
$\cA$-relative level equivalences and $\cA$-relative level fibrations.  
  Recall the functors $G_{\bld{k}}^{\cK}$ from Section~\ref{subs:free-semi-free} and let
\[ \begin{split} I^{\textit{level}}_{(\cK,\cA)} &= \{
  G^{\cK}_{\bld{k}}(i) | \bld{k}\in O(\cA) \textrm{ and } i\in
  I_{(\cK(\bld k),\cA(\bld k))}\}\quad \textrm{and}\\
  J^{\textit{level}}_{(\cK,\cA)} &= \{ G^{\cK}_{\bld{k}}(j) |
  \bld{k}\in O(\cA)\textrm{ and }
  j\in J^{\textit{mix}}_{(\cK(\bld k),\cA(\bld k))} \}. \end{split}\] 
Here we write $O(\cA)$ for the set of objects in $\cA$.

\begin{proposition}\label{prop:level-A-relative}
The $\cA$-relative level equivalences, level fibrations, and cofibrations specify a cofibrantly generated model structure on $\cS^{\cK}$ with generating cofibrations $I^{\textit{level}}_{(\cK,\cA)}$ and generating acyclic cofibrations  
$J^{\textit{level}}_{(\cK,\cA)}$.
\end{proposition}
We shall refer to this as the \emph{$\cA$-relative level model structure} on 
$\mathcal S^{\cK}$.

\begin{proof}
  We use the recognition criterion for cofibrantly generated model 
  categories as stated in \cite[Theorem 2.1.19]{Hovey-model}. The
  smallness requirements are satisfied because the generating (acyclic)  
  cofibrations are levelwise cofibrations in $\cS$ and $\cS$ is small relative to
  the cofibrations.
  It follows from the definition that a map $X\to Y$ in $\cS^{\cK}$ is
  $ I^{\textit{level}}_{(\cK,\cA)}$-injective if and only if for all
  $\bld{k}$ in $\cA$ and all subgroups
  $H\subseteq \cA(\bld{k})$ the induced map $X(\bld{k})^H\to
  Y(\bld{k})^H$ is an acyclic fibration. Similarly, Lemma
  \ref{flat-injective-lemma} implies that $X\to Y$ is
  $J^{\textit{level}}_{(\cK,\cA)}$-injective if and only if it is an
  $\cA$-relative level fibration.  From these explicit descriptions it
  is clear that the $ I^{\textit{level}}_{(\cK,\cA)}$-injective maps are
  the $J^{\textit{level}}_{(\cK,\cA)}$-injective maps that are
  $\cA$-relative level equivalences. Furthermore, this has as a formal
  consequence that the class $J^{\textit{level}}_{(\cK,\cA)}\cof$ is
  contained in $I^{\textit{level}}_{(\cK,\cA)}\cof$.
  
  It remains to show that the maps in
  $J^{\textit{level}}_{(\cK,\cA)}\cell$ are $\cA$-relative level
  equivalences. We first show that the maps in 
  $J^{\textit{level}}_{(\cK,\cA)}$ are level equivalences and level cofibrations at all levels, not only those corresponding to objects in $\cA$. Indeed, for a map of the form $G^{\cK}_{\bld{k}}(j)$ for $\bld k$ in $\cA$ and 
$j\in J_{(\cK(\bld k),\cA(\bld k))}$ this easily follows from the explicit description of the functor $G^{\cK}_{\bld{k}}$ in \eqref{eq:free-functors}.
Consider then a map of the form $G^{\cK}_{\bld{k}}(j_H\Box i)$ for $\bld k$ in 
$\cA$ and $j_H\Box i$ in  $J'_{(\cK(\bld k),\cA(\bld k))}$. Since 
$G_{\bld k}^{\cK}$ preserves colimits and tensors (over $\cS$) we can identify this map with the pushout-product $G_{\bld k}^{\cK}(j_H)\Box i$. 
Using that $H$ acts freely from the right on the morphism sets $\cK(\bld k, -)$ (by Lemma~\ref{lem:free-right-action}, letting $X=F_{\bld 0}^{\cK}(*)$) we again conclude from \eqref{eq:free-functors} that 
$G^{\cK}_{\bld{k}}(\pi_H)$ is a level equivalence. Identifying 
$G^{\cK}_{\bld{k}}(M(\pi_H))$ with the mapping cylinder of 
$G_{\bld k}^{\cK}(\pi_H)$ we see that $G^{\cK}_{\bld{k}}(j_H)$ is both a level equivalence and a level cofibration. By the pushout-product axiom for 
$\cS$ we finally conclude that $G_{\bld k}^{\cK}(j_H)\Box i$ is a level equivalence and a level cofibration. By definition, a map in
  $J^{\textit{level}}_{(\cK,\cA)}\cell$ is the transfinite composition
  of a sequence of maps each of which is the pushout of a map in
  $J^{\textit{level}}_{(\cK,\cA)}$. At each level such a map is therefore the
  transfinite composition of a sequence of acyclic cofibrations, hence a weak
   equivalence.
\end{proof}

As promised, there is a more explicit description of the $\cA$-relative cofibrations. Let $\partial\left(\cK \!\downarrow \!\bld{k}\right)$ be the full
subcategory of $(\cK \!\downarrow \!\bld{k})$ whose objects are the
non-isomor\-phisms.  For a $\cK$-space $X$, the $\bld{k}$-th latching
space $L_{\bld k}(X)$ is the colimit of the 
$\partial\left(\cK \!\downarrow \!\bld{k}\right)$-diagram
$(\bld l\to\bld k)\mapsto X(\bld l)$, and the $\bld k$th latching map is the   
canonical $\cK(\bld k)$-equivariant map 
\[ L_{\bld{k}}(X) = \colimwithlimits_{(\bld{l}\to\bld{k}) \in \partial\left(\cK
    \downarrow \bld{k}\right)} X(\bld{l}) \to
\colimwithlimits_{(\bld{l}\to\bld{k}) \in (\cK \downarrow \bld{k})}X(\bld l)
\xrightarrow{\iso} X(\bld{k}).\]
For a map of $\cK$-spaces $X \to Y$, the $\bld{k}$-th latching map
is the $\cK(\bld k)$-equivariant map
\begin{equation}\label{eq:latching-map} 
  L_{\bld{k}}(Y) \cup_{L_{\bld{k}}(X)}X(\bld{k}) \to Y(\bld{k}).
\end{equation}
Recall the notion of an $\cA(\bld k)$-relative cofibration in $\cS^{\cK(\bld k)}$ from Section \ref{subs:G-space-model}.

\begin{proposition}\label{prop:latching-characterization}
  A map of $\cK$-spaces $f\colon X\to Y$ is an $\cA$-relative
  cofibration if and only if the latching map \eqref{eq:latching-map}
  is an $\cA(\bld k)$-relative cofibration for all $\bld k$ in $\cA$ and an isomorphism for all $\bld k$ not in $\cA$. 
\end{proposition} 
\begin{proof}
  Choosing representatives for the isomorphism classes of objects $\bld{k}$ 
  in $\cK$ with common value $\lambda(\bld{k})=n$, one shows by induction
  on $n$ that maps satisfying the stated condition have the left lifting property
  with respect to the maps that are $\cA$-relative level equivalences and  
  $\cA$-relative level fibrations.  Hence maps satisfying the condition  are 
  $\cA$-relative cofibrations.

  For the other direction, one first shows that the generating cofibrations
  satisfy the condition in the proposition. This uses the normality condition on $\cA$. Since the condition is preserved under   
  cobase change, transfinite composition, and retracts, this implies that it holds for all $\cA$-relative cofibrations.
\end{proof}

\begin{remark}\label{rem:simpl-top-K-model} One may also compare the
 simplicial and the topological version of these model structures:
 The geometric realization functor $|- |$ and the singular complex functor 
 $\Sing$ induce an adjunction between $\cK$-diagrams in simplicial sets and (compactly generated weak Hausdorff) topological spaces. It is easy to
 check that this defines a Quillen equivalence with respect to each of the
model structures we consider. Using the $(|-|,\Sing)$-adjunction we can also turn any $\cS$-model structure in the topological setting into a simplicial model structure.
\end{remark}

Recall from \cite[Definition 4.2.18]{Hovey-model} that an $\cS$-model category $\cM$ is a category which is enriched, tensored, and  cotensored over 
$\cS$, and equipped with a model structure such that if $f\colon X\to Y$ is a cofibration in $\cM$ and $g\colon S\to T$ a cofibration in $\cS$, then the pushout-product
\[
f\Box g\colon Y\times S\cup_{X\times S}X\times T\to Y\times T
\] 
is a cofibration in $\cM$ which is acyclic if either $f$ or $g$ is.

\begin{proposition}\label{prop:level-S-model}
  The $\cA$-relative level model structure on 
  $\cS^{\cK}$ is an  $\cS$-model structure.
  \end{proposition}
\begin{proof}
We already know from Lemma \ref{lem:K-spaces-S-category} that $\cS^{\cK}$ is enriched, tensored, and cotensored over $\cS$. For the statements concerning the pushout-product, it suffices by 
\cite[Corollary 4.2.5]{Hovey-model} to consider the generating (acyclic) cofibrations for the respective model structures. Let $\bld k$ be an object in $\cA$, let $f$ be a map in $\cS^{\cK(\bld k)}$, and let $g$ be a map in $\cS$. Then $G_{\bld k}^{\cK}(f)\Box g$ can be identified with  
$G_{\bld k}^{\cK}(f\Box g)$ where $f\Box g$ is the pushout-product in 
$\cS^{\cK(\bld k)}$. If $f$ is a generating $\cA(\bld k)$-relative cofibration and $g$ is a generating cofibration in $\cS$, then $f\Box g$ is an $\cA(\bld k)$-relative cofibration which in turn implies that 
$G_{\bld k}^{\cK}(f\Box g)$ is an $\cA$-relative cofibration. Now suppose that either $f$ or $g$ is a generating acyclic cofibration. Applying the pushout-product axiom for $\cS$ levelwise we see that $G_{\bld k}^{\cK}(f)\Box g$ is then an $\cA$-relative level equivalence. 
 \end{proof}

\subsection{The $\cA$-relative \texorpdfstring{$\cK$}{K}-model structures}
\label{sec:K-model-structure}
Let $\cC$ be a small category and let $X$ be a $\cC$-diagram 
of spaces. We either write $X_{h\cC}$ or $\hocolim_{\cC}X$  for the homotopy colimit of $X$ over $\cC$, defined in the usual way as the realization of the simplicial replacement of the diagram,
\[ 
X_{h\cC}=\hocolim_{\cC}X = \bigg| [s] \mapsto \displaystyle\coprod_{\bld{k_0} \ot \dots
  \ot \bld{k_s}} X(\bld{k_s})\bigg|,
\]
see e.g.\ \cite{Bousfield_K-homotopy,Hirschhorn-model} for details. 
(In the simplicial setting $|\ |$ indicates the diagonal simplicial set.) For homotopy colimits in a general model category, one has to
assume that the diagram $X$ is object-wise cofibrant for this construction 
to capture the correct homotopy type. We do not need this assumption here
because we either work in simplicial sets (where it is automatically satisfied) or in (compactly generated weak Hausdorff) topological spaces, where it follows from \cite[Appendix A]{Dugger_I-hypercovers} that the assumption can be dropped.
We will freely use many standard properties of homotopy colimits as
for example developed in~\cite[Chapter 18]{Hirschhorn-model}.
Moreover, we will use the following result which is reproduced here for easy reference.
\begin{lemma}\label{lem:puppe-lemma}\cite[Proposition
  4.4]{Rezk_SS-simplicial} Let $\cC$ be a small category, let $X\to Y$ be
  a map of $\cC$-diagrams in $\cS$, and let $\alpha \colon \bld{k}\to\bld{l}$ be
  a morphism in $\cC$. Consider the two squares
\[\xymatrix@-1pc{X(\bld{k}) \ar[r] \ar[d]_{X(\alpha)}& 
  Y(\bld{k})\ar[d]^{Y(\alpha)}\\
  X(\bld{l})\ar[r] & Y(\bld{l})}
\qquad\qquad\xymatrix@-1pc{X(\bld{k}) \ar[r] \ar[d]& Y(\bld{k})\ar[d]\\
  X_{h\cC} \ar[r] &Y_{h\cC}.}\] If the left hand square is homotopy
cartesian for every $\alpha$, then the right hand square is homotopy
cartesian for every object $\bld{k}$.\qed
\end{lemma}
\begin{remark}\label{rem:puppe-remark}
In \cite{Rezk_SS-simplicial} the above lemma is only stated for simplicial sets, but the analogous result for (compactly generated weak Hausdorff) topological spaces is an immediate consequence. Indeed, recall that a square diagram of topological spaces is homotopy cartesian if and only if applying the singular complex functor $\Sing$ gives a homotopy cartesian diagram of simplicial sets. Conversely, a square diagram of simplicial sets is homotopy cartesian if and only if the geometric realization is homotopy cartesian. Thus, given a map $X\to Y$ of $\cC$-diagrams of topological spaces such that the left hand squares are homotopy cartesian, the lemma implies that the diagram
\[ 
\xymatrix@-1pc{
\Sing X(\bld k) \ar[r] \ar[d]& \Sing Y(\bld k)\ar[d]\\
(\Sing X)_{h\cC} \ar[r] & (\Sing Y)_{hC}
}
\]
is homotopy cartesian. This in turn implies that the geometric realization is homotopy cartesian and the natural transformation $|\Sing X|\to X$ defines a natural weak equivalence between this realization and the right hand square in the lemma.
\end{remark}

The following definition is central to the rest of the paper.

\begin{definition}  
A map $X \to Y$ of $\cK$-spaces is a \emph{$\cK$-equivalence} if the induced map $X_{h\cK} \to Y_{h\cK}$ is a weak equivalence of spaces.
\end{definition} 

Now let $(\cK,\cA)$ be a well-structured relative index category. We proceed to construct an $\cA$-relative model structure on $\cS^{\cK}$ with the
$\cK$-equivalences as the weak equivalences. The cofibrations for this model structure are the $\cA$-relative cofibrations as characterized in Proposition~\ref{prop:latching-characterization}. Recall that $\cK_{\cA}$ denotes the full subcategory of $\cK$ generated by the objects of $\cA$. 
A map of $\cK$-spaces $X\to Y$ is an \emph{$\cA$-relative 
  $\cK$-fibration} if it is an $\cA$-relative level
fibration with the additional property that every morphism 
 $\alpha \colon \bld{k}\to\bld{l}$ in $\cK_{\cA}$
 induces a homotopy cartesian square 
 \[ 
 \xymatrix@-1pc{
  X(\bld{k})\ar[r] \ar[d] & X(\bld{l})\ar[d]\\
  Y(\bld{k})\ar[r] & Y(\bld{l}). }
\]
Let $\alpha\colon \bld k\to\bld l$ be a morphism in $\cK_{\cA}$ and consider the induced map of $\cK$-spaces $\alpha^*\colon  F_{\bld{l}}^{\cK}(*)\to
F_{\bld{k}}^{\cK}(*)$ defined by precomposition with $\alpha$. 

\begin{lemma}\label{lem:alpha^*-equivalence}
The map $\alpha^*\colon  F_{\bld{l}}^{\cK}(*)\to F_{\bld{k}}^{\cK}(*)$ is a
$\cK$-equivalence.
\end{lemma}
\begin{proof}
By definition of the homotopy colimits, $F_{\bld{k}}^{\cK}(*)_{h\cK}$ and $F_{\bld{l}}^{\cK}(*)_{h\cK}$ can be identified with the classifying spaces of the categories $(\bld k\downarrow \cK)$ and  $(\bld l\downarrow \cK)$. These categories each has an initial object and the map induced by $\alpha^*$ is therefore trivially a weak equivalence.
\end{proof}

We now use the tensor with the interval
in $\cS$ to factor $\alpha^*$ through the mapping cylinder $M(\alpha^*)$ in the usual way,
\begin{equation}\label{eq:alpha-star-fact}
  \alpha^*\colon F_{\bld{l}}^{\cK}(*)\xrightarrow{j_{\alpha}}M(\alpha^*)\xrightarrow{r_{\alpha}} F_{\bld{k}}^{\cK}(*).
\end{equation}
Arguing as in the case of symmetric spectra \cite[Lemma 3.4.10]{HSS}
we see that $j_{\alpha}$ is an $\cA$-relative cofibration and $r_{\alpha}$ is a homotopy equivalence. Let $J_{(\cK,\cA)}'$ be the set of morphisms of the form 
$j_{\alpha}\Box i$ where $j_{\alpha}$ is as above (for $\alpha$ in $\cK_{\cA}$), $i$ is a generating cofibration in $\cS$, and 
$\Box$ denotes the pushout-product map associated to the tensor with an object of $\cS$. 
We define $I_{(\cK,\cA)}=I^{\textit{level}}_{(\cK,\cA)}$
and $J_{(\cK,\cA)}=J^{\textit{level}}_{(\cK,\cA)}\cup J'_{(\cK,\cA)}$.

\begin{proposition}\label{prop:K-model-str}
  The $\cK$-equivalences together with the $\cA$-relative $\cK$-fibrations and 
  the $\cA$-relative cofibrations specify a cofibrantly
  generated model structure on $\cS^{\cK}$ with generating
  cofibrations $I_{(\cK,\cA)}$ and generating acyclic
  cofibrations $J_{(\cK,\cA)}$.
\end{proposition}
We shall refer to this as the \emph{$\cA$-relative $\cK$-model structure} on 
$\cK$-spaces.

\begin{proof}
  We again use the criterion of~\cite[Theorem 2.1.19]{Hovey-model}.
  One can apply Proposition \ref{prop:h-cofibration-properties} (viii)  below to
  see that the smallness requirements are satisfied.
  
  As in the $\cA$-relative level model structure, the
  $I_{(\cK,\cA)}$-injective maps are the maps $X\to Y$
  such that for all objects $\bld{k}$ in $\cA$ and all
  subgroups $H\subseteq \cA(\bld{k})$, the induced map
  $X(\bld{k})^H\to Y(\bld{k})^H$ is an acyclic fibration. Moreover,
  arguing as in the proof of \cite[Lemma 3.4.12]{HSS}, we see
  that $X\to Y$ is $J_{(\cK,\cA)}$-injective if and
  only if it is an $\cA$-relative $\cK$-fibration.
  Thus, a map $X \to Y$ which is $I_{(\cK,\cA)}$-injective is clearly both   
  $J_{(\cK,\cA)}$-injective and a $\cK$-equivalence.
  Suppose then that $f\colon X\to Y$ is $J_{(\cK,\cA)}$-injective
  and a  $\cK$-equivalence. Then $f$ is an $\cA$-relative level  
  fibration and it follows from the homotopy cofinality condition (iv) for a well-structured relative index category that the induced map 
 $X_{h\cK_{\cA}}\to Y_{h\cK_{\cA}}$ is a weak equivalence.  Therefore 
  Lemma \ref{lem:puppe-lemma} implies that $X\to Y$ is also an 
  $\cA$-relative level equivalence, hence $I_{(\cK,\cA)}$-injective.
  
  The last thing to be checked is that the maps in the class
  $J_{(\cK,\cA)}\cell$ also belong to the class
  $I_{(\cK,\cA)}\cof$ and are $\cK$-equivalences. Here
  the first part follows formally from the above discussion. For the
  second part we first observe that the maps in $J_{(\cK,\cA)}$ are 
  $\cK$-equivalences by Lemma \ref{lem:alpha^*-equivalence}.
  We next observe that the functor  $\hocolim_{\cK}$ takes the
   class $I_{(\cK,\cA)}\cof$ to cofibrations in $\cS$. Indeed, since
   $\hocolim_{\cK}$ preserves
  colimits it suffices to check that it takes the elements in
  $I_{(\cK,\cA)}$ to cofibrations in $\cS$ and
  this is easy to check directly. By definition, a map in
  $J_{(\cK,\cA)}\cell$ is the transfinite composition
  of a sequence of maps each of which is a pushout of a map in
  $J_{(\cK,\cA)}$. The induced map 
  $X_{h\cK}\to Y_{h\cK}$ is therefore the 
  transfinite composition of a sequence of
  maps each of which is a pushout of an acyclic cofibration; again
  because $\hocolim_{\cK}$ preserves colimits. The induced map itself
  is therefore also a weak equivalence as had to be shown.
\end{proof}
\begin{remark}
  Dugger studied hocolim model structures on $\cC$-diagrams in a model
  category $\cM$ for a contractible category $\cC$ in~\cite[Theorem
  5.2]{Dugger_replacing-simplicial}. These coincide with the model
  structures of the previous proposition if $\cK$ is contractible and
  $\cA = O\cK$.
\end{remark}

For future reference we spell out the condition for a $\cK$-space to be fibrant in the model structure of Proposition~\ref{prop:K-model-str}.

\begin{proposition}\label{prop:A-relative-fibrant-K-space}
  A $\cK$-space $X$ is fibrant in the $\cA$-relative $\cK$-model
  structure if and only if
  \begin{enumerate}[(i)]
  \item for each $\bld{k}$ in $\cA$ and
    each subgroup $H\subseteq \cA(\bld{k})$ the space
    $X(\bld{k})^{H}$ is fibrant and the map $X(\bld{k})^{H}\to
    X(\bld{k})^{hH}$ is a weak equivalence, and
  \item for each morphism
    $\alpha\colon\bld{k}\to\bld{l}$ in $\cK_{\cA}$ the induced map  
    $X(\bld{k})\to X(\bld{l})$ is a weak equivalence.\qed
  \end{enumerate}
\end{proposition}
Here the fibrancy condition on the spaces $X(\bld k)^H$ is of course automatically satisfied in the topological setting.

\begin{proposition}\label{prop:K-model-is-S-model-str}
 The $\cA$-relative $\cK$-model structure on $\cS^{\cK}$ is an  $\cS$-model structure.
\end{proposition}
\begin{proof}
Let $f\colon X\to Y$ be an $\cA$-relative cofibration in $\cS^{\cK}$ and let 
$g\colon S\to T$ be a cofibration in $\cS$. Then the pushout-product 
$f\Box g$ is an $\cA$-relative cofibration by 
Proposition~\ref{prop:level-S-model} and we must show it to be a 
$\cK$-equivalence if either $f$ is a $\cK$-equivalence or $g$ is a weak equivalence. Using that the homotopy colimit functor preserves colimits and tensors we can identify $(f\Box g)_{h\cK}$ with the pushout-product $f_{h\cK}\Box g$ in $\cS$. Since the homotopy colimit functor also preserves cofibrations (see the proof of Proposition 
\ref{prop:K-model-str}) the result now follows from the pushout-product axiom for $\cS$.
\end{proof}

As discussed in Remark~\ref{rem:simpl-top-K-model}, we can use the 
$(|-|,\Sing)$-adjunction to make the $\cA$-relative $\cK$-model structure a simplicial model structure also in the topological version of the theory. 

In the next proposition we compare the relative $\cK$-model structures associated to different subcategories of automorphisms. Recall the normality and multiplicative conditions on our subcategories of automorphisms stated before Definition \ref{def:well-structured-relative-index}.  

\begin{proposition}\label{prop:well-structured-comparison}
Let $(\cK,\cB)$ be a well-structured relative index category. Suppose that 
$\cA$ is a normal and multiplicative subcategory of automorphisms contained in $\cB$ and that the inclusion $\cK_{\cA}\to \cK_{\cB}$ is homotopy cofinal. Then $(\cK,\cA)$ is a well-structured relative index category and the 
$\cA$- and $\cB$-relative $\cK$-model structures on $\cS^{\cK}$ are Quillen equivalent.
\end{proposition}
\begin{proof}
It is immediate from the definitions that the identity functor on $\cS^{\cK}$ is the left Quillen functor of a Quillen equivalence from the $\cA$-relative 
$\cK$-model structure to the $\cB$-relative $\cK$-model structure.
\end{proof}

Now let us specialize to the case of a well-structured index category $\cK$ as specified in Definition \ref{def:well-structured-index}.

\begin{definition}\label{def:K-projective-model}
Let $\cK$ be a well-structured index category. 
\begin{enumerate}[(i)]
\item
The \emph{projective $\cK$-model structure} on $\cS^{\cK}$ is obtained from Proposition \ref{prop:K-model-str} by letting $\cA$ be the category of identity morphisms $O\cK$ in $\cK$.
\item
Suppose that the full subcategory $\cK_+$ of objects with positive degree is homotopy cofinal in $\cK$. The  \emph{positive projective $\cK$-model structure} on $\cS^{\cK}$ is then obtained from Proposition \ref{prop:K-model-str} by letting $\cA$ be the category of identity morphisms $O\cK_+$ in $\cK_+$. 
\end{enumerate}
\end{definition}
We next describe some features of the projective 
$\cK$-model structure that is not shared by the $\cA$-relative $\cK$-model structures in general. 

\begin{lemma}\label{lem:projective-hocolim}
Let $X$ be a $\cK$-space which is cofibrant in the projective $\cK$-model structure on $\cS^{\cK}$. Then the canonical map $\hocolim_{\cK}X\to
\colim_{\cK}X$ is a weak equivalence. 
\end{lemma}
\begin{proof}
This is proved in the simplicial setting in \cite[Proposition
  18.9.4]{Hirschhorn-model}  and a similar argument applies for
$\cK$-diagrams of topological spaces.
\end{proof}

\begin{proposition}\label{prop:colim-const-Q-adjunction}
Let $\cK$ be a well-structured index category and give $\cS^{\cK}$ the projective $\cK$-model structure. Then the adjunction 
$\colim_{\cK} \colon \cS^{\cK} \rightleftarrows \cS
\thinspace \colon\! \!\const_{\cK}$ is a Quillen adjunction. It is a Quillen equivalence if and only if  $B\cK$ is contractible.
\end{proposition}
\begin{proof}
The adjunction is a Quillen adjunction because $\const_{\cK}$ preserves fibrations and acyclic fibrations. Suppose that $B\cK$ is contractible. To show that the adjunction is a Quillen equivalence we must check that for a cofibrant $\cK$-space $X$ and a fibrant space $Y$, a map $\colim_{\cK}X\to Y$ is a weak equivalence if and only if its adjoint $X\to \const_{\cK}Y$ is a $\cK$-equivalence. These maps fit into a commutative diagram
\[
\xymatrix@-1pc{
X_{h\cK} \ar[r]\ar[d] & (\const_{\cK}Y)_{h\cK} \ar[r]^-{\sim} & B\cK\times Y\ar[d]\\
\colim_{\cK}X \ar[rr] && Y
}
\] 
where the vertical maps are induced by the canonical map from the homotopy colimit to 
the colimit. This gives the result since the vertical maps are weak equivalences by Lemma~\ref{lem:projective-hocolim} and the assumption on $B\cK$. Next assume that the adjunction is a Quillen equivalence. Letting 
$X=F^{\cK}_{\bld 0}(*)$ and $Y=*$ in the above diagram then shows that $B\cK$ is contractible.
\end{proof}

\section{The class of \texorpdfstring{$h$}{h}-cofibrations} \label{subs:h-cofibrations}
In this section $(\cK,\cA)$ again denotes a well-structured relative index category. For the homotopical analysis of the monoidal structure of 
$\mathcal S^{\cK}$ it will be convenient to have available a weaker
notion of cofibrations than the $\cA$-relative cofibrations. Specifically, in the topological setting these will be the classical Hurewicz cofibrations, while in the simplicial setting these will be the levelwise injections. 
We shall use the term \emph{$h$-cofibration} for a map in one of these classes.

Below we state a number of results about the $h$-cofibrations which we
verify in the  topological and simplicial settings in Sections
\ref{topological-h-cofibration} and \ref{simplicial-h-cofibration},
respectively. Given a map  $f \colon X \to Y$ of
$\cK$-spaces, we write $f^{\Box n}\colon Q^n_{n-1}(f) \to Y^{\boxtimes
  n}$ for the $n$-fold iterated pushout product map of $f$. We will
study this construction in more detail in 
Section~\ref{iterated-pushout-section} where we 
show that $f^{\Box n}$ is a $\Sigma_n$-equivariant map of $\cK$-spaces 
with $\Sigma_n$-action.

\begin{proposition}\label{prop:h-cofibration-properties}
\begin{enumerate}[(i)]
\item Every $\cA$-relative cofibration is an $h$-cofibration.
\item The $h$-cofibrations are preserved under cobase change,
  transfinite composition, and retracts.
\item The gluing lemma for $h$-cofibrations and $\cA$-relative level
  equivalences holds.
\item The gluing lemma for $h$-cofibrations and 
$\cK$-equivalences holds.
\item Let $\lambda$ be an ordinal. If $\{X_{\alpha}\colon
  \alpha<\lambda\}$ is a $\lambda$-sequence of $h$-cofibrations, then 
  the canonical map $\hocolim_{\alpha<\lambda}X_{\alpha}\to
  \colim_{\alpha<\lambda}X_{\alpha}$ is a level equivalence.
\item For every $\cK$-space $X$, the functor $X\boxtimes(-)$ sends 
$\cA$-relative cofibrations to $h$-cofibrations.
\item Let $f \colon X \to Y$ be a generating cofibration for the $\cA$-relative 
  $\cK$-model structure and let $Z$ be a $\cK$-space with
  a right $\Sigma_n$-action. Then $ Z \boxtimes_{\Sigma_n} f^{\Box n}$ is an
  $h$-cofibration.
\item Every $\cK$-space is small relative to the
  $h$-cofibrations.
\end{enumerate}
\end{proposition}

The gluing lemmas in (iii) and (iv) are the statements that given a map of diagrams 
\begin{equation}\label{h-cofibrant-diagram}
\xymatrix@-1pc{
Y\ar[d] & X\ar[l] \ar[r]^{i} \ar[d]& Z\ar[d]\\
Y' & X'\ar[l] \ar[r]^{\imath'} & Z'
}
\end{equation}
in which $i$ and $\imath'$ are $h$-cofibrations and the vertical maps are 
 $\cA$-relative level equivalences (respectively $\cK$-equivalences), 
then the  map of pushouts $Y\cup_{X}Z\to Y'\cup_{X'}Z'$ is also 
an $\cA$-relative level equivalence (respectively a $\cK$-equivalence).
In (viii) the term \emph{small} has its usual set theoretical meaning, see e.g.,
\cite[Section 10.4]{Hirschhorn-model}.

\begin{remark}
  Hill, Hopkins and Ravenel~\cite[Definition
  B.15]{Hill-H-R_Kervaire-invariant} define a map $f$ in a model
  category to be \emph{flat} if cobase change along $f$ preserves weak
  equivalences. The previous proposition shows that the
  $h$-cofibrations in $\cS^{\cK}$ are flat in this sense. In
  particular, the flat cofibrations introduced for $\cI$- and
  $\cJ$-spaces in Sections~\ref{sec:I-space-section}
  and~\ref{sec:J-space-section} also satisfy this more general
  flatness condition.
    \end{remark}

We note some immediate consequences of these results which we state explicitly for easy reference.
\begin{corollary} \label{cor:pushout-K-or-level-h} The cobase change
  of a map which is both an $\cA$-relative level equivalence (or $\cK$-equivalence) and an $h$-cofibration is also an $\cA$-relative level equivalence (or $\cK$-equivalence) and an $h$-cofibration.\qed
\end{corollary}
 
\begin{corollary} \label{cor:transfinite-K-or-level-equivalence} If
  $\{X_{\alpha}\colon \alpha<\lambda\}$ is a $\lambda$-sequence in
  $\mathcal S^{\cK}$ such that each of the maps $X_{\alpha}\to
  X_{\alpha+1}$ is an $\cA$-relative level equivalence (or  $\cK$-equivalence)
  and an $h$-cofibration, then the transfinite composition $X_0\to \colim_{\alpha<\lambda}X_{\alpha}$ is also an $\cA$-relative level equivalence
  (or $\cK$-equivalence)  and an $h$-cofibration.\qed
\end{corollary}

\subsection{Topological \texorpdfstring{$h$}{h}-cofibrations}\label{topological-h-cofibration}

In this paragraph we write $\Top= \cS$ for our category of spaces to
emphasize that we work in the context of (compactly generated weak
Hausdorff) topological spaces.  For the proof of Proposition
\ref{prop:h-cofibration-properties} it is convenient to have
available a Str\o m type model structure on $\Top^{\cK}$.

We review the relevant definitions and results.  Consider in general a
small category $\cC$ and write $\Top^{\cC}$ for the category of
$\cC$-spaces. This category is tensored and cotensored over $\Top$ with
the tensor $X\times K$ and cotensor $X^K$ of a $\cC$-space $X$ and a
space $K$ defined by the obvious object-wise constructions. Homotopies
are defined in the usual way using the tensor with the unit interval
$I$.  

A morphism $f\colon X\to Y$ in $\Top^{\cC}$ is a \emph{Hurewicz
  cofibration} (\emph{$h$-cofibration}) if it has the left
lifting property with respect to the map $p_0\colon Z^I\to Z$ (induced
by the inclusion $\{0\}\to I$) for any $\mathcal C$-space $Z$. Writing $M(f)$ for
the mapping cylinder $Y\cup_{X\times\{0\}}X\times I$, this is
equivalent to the condition that the canonical map $M(f)\to Y\times I$
admits a retraction. Similarly, we say that a map in $\Top^{\cC}$ is a
\emph{Hurewicz fibration} (\emph{$h$-fibration}) if it has the
right lifting property with respect to the map $i_0\colon A\to A\times
I$ (again induced by $\{0\}\to I$) for any space $A$. When $\cC$ is
the terminal category these notions reduce to the usual Hurewicz
cofibrations and fibrations in $\Top$. We refer the reader to
\cite{Barthel-R_functorial-factorizations, Cole_many-homotopy, Schwanzl-V_cofibrations} for a general
discussion of homotopy theory in topological categories.  The
following is an extension of Str\o m's model structure~\cite{Strom_homotopy} to arbitrary diagram spaces
in $\Top$.
\begin{proposition}
  The classes of homotopy equivalences, $h$-cofibrations, and
  $h$-fibrations specify a model structure on $\Top^{\cC}$.
\end{proposition}
\begin{proof}
  One may either reuse Str\o m's proof in the setting of $\cC$-spaces,
  or one may apply the general criterion for the existence of such
  model structures formulated by Cole~\cite{Cole_many-homotopy} and
  corrected by
  Barthel-Riehl~\cite{Barthel-R_functorial-factorizations}. In the
  latter approach there are two conditions that must be checked.
  The first is the assumption that $\Top^{\cC}$ is locally bounded which
  is needed to apply~\cite[Corollary
  5.18]{Barthel-R_functorial-factorizations}. By a similar argument as
  in~\cite[Example 5.14]{Barthel-R_functorial-factorizations},
  $\Top^{\cC}$ inherits this property from $\Top$.  The second
  condition is that the $h$-cofibrations are exactly the maps that
  have the left lifting property with respect to $h$-fibrations that
  are homotopy equivalences (i.e., in the formulation of Cole that the
  class of $h$-cofibrations equals the class of strong
  $h$-cofibrations). For this we use the characterization of
  Schw\"anzl-Vogt~\cite[Proposition 3.5(6)] {Schwanzl-V_cofibrations}
  which shows that a map $f\colon X\to Y$ has the left lifting
  property with respect to $h$-fibrations that are homotopy
  equivalences if and only if the induced map $M(f)\to Y\times I$ has
  the left lifting property with respect to all $h$-fibrations. It
  follows from the proof of \cite[Theorem 4]{Strom_note} that the
  latter condition is satisfied if $f$ is an $h$-cofibration (Str\o{}m
  formulates the result for the category of all topological spaces but
  the argument works the same for the category $\Top^{\cC}$).
\end{proof}
\begin{lemma}\label{lem:functors-preserving-h-cof}
If $\cC$ and $\cB$ are small categories and 
$F\colon \Top^{\cC}\to\Top^{\cB}$ is a functor that preserves colimits 
and tensors, then $F$ also preserves $h$-cofibrations
\end{lemma}
\begin{proof}
This easily follows from the mapping
cylinder retract characterization.
\end{proof}
In particular, this applies to the functor
$\hocolim_{\cC}\colon \Top^{\cC}\to \Top$ and the evaluation
functor $\Ev_{\bld{k}}\colon \Top^{\cC}\to \Top$ for an object $\bld{k}$ in
$\cC$.

\begin{proof}[Proof of Proposition \ref{prop:h-cofibration-properties} in
the topological setting]
For (i) it suffices to show that an $\cA$-relative cofibration has the left lifting property with respect to the map $p_0\colon Z^I \to Z$ for any
$\mathcal K$-space $Z$. We know that the $\cA$-relative 
cofibrations have the left lifting property with respect to the acyclic $\cA$-relative $\mathcal K$-fibrations, that is, the maps $X\to Y$ such that $X(\mathbf k)^H\to Y(\mathbf k)^H$ is an acyclic fibration for all objects
$\mathbf k$ in $\cA$ and all subgroups $H\subseteq\cA(\mathbf k)$. The map 
$p_0^H(\mathbf k)$ can be identified with the 
evaluation map $(Z(\mathbf k)^H)^I\to Z(\mathbf k)^H$ and the result follows 
since this is even a homotopy equivalence and a Hurewicz fibration.  

Part (ii) is immediate since the $h$-cofibrations are the cofibrations in a 
model structure.  For the gluing lemma (iii), we are given a
commutative diagram as in (\ref{h-cofibrant-diagram}) in which the vertical 
maps are $\cA$-relative level equivalences. Since the $h$-cofibrations in 
$\Top^{\mathcal K}$ are object-wise $h$-cofibrations, the claim follows from the gluing lemma for $h$-cofibrations and weak equivalences in $\Top$, see e.g.\ \cite[Appendix, Proposition 4.8(b)]{Boardman-V_homotopy-invariant}.
The gluing lemma for the $\cK$-equivalences in (iv) follows for the same reason because $\hocolim$ preserves colimits and $h$-cofibrations.

To show (v), we use that the Str\o m model structure on $\Top^{\cK}$
induces a Reedy model structure on the category of $\lambda$-diagrams
in $\Top^{\cK}$ in which a morphism $X\to Y$ is a weak equivalence or
fibration if and only if $X_{\alpha}\to Y_{\alpha}$ is respectively a
homotopy equivalence or $h$-fibration for all $\alpha<\lambda$. In
this model structure a $\lambda$-sequence is cofibrant if and only if
each map $X_{\alpha}\to X_{\alpha+1}$ is an $h$-cofibration, hence it
follows from general results about Reedy model structures
\cite[Theorem 19.9.1]{Hirschhorn-model} that the canonical map from
the homotopy colimit to the colimit is a homotopy equivalence.

For (vi), we notice that the functor $X\boxtimes(-)$ preserves colimits because
it is a left adjoint. Since it also preserves tensors, 
the claim follows from Lemma \ref{lem:functors-preserving-h-cof}
and part~(i).
In (vii) we may assume that $f$ has the form $G_{\bld{k}}^{\cK}(\cK(\bld{k})/H
\times i)$ for an object $\bld k$ in $\cA$, a subgroup $H\subseteq \cA(\bld k)$, 
and  $i$ a generating cofibration in $\Top$. Using 
Lemma~\ref{lem:product_of_free_Jspaces} we have the identification
\[ 
f^{\Box n} \iso G_{\bld{k}^{\concat n}}^{\cK}(\cK(\bld{k}^{\concat n})/
H^{\times n}) \times i^{\Box n}
\]
where $i^{\Box n}$ is the iterated pushout-product map in $\Top$. The category 
$\Top^{\mathcal K\times\Sigma_n}$ is tensored over the category 
$\Top^{\Sigma_n}$ 
of $\Sigma_n$-spaces and we may view the above map as the tensor of the identity 
on $G_{\bld{k}^{\concat n}}^{\cK}(\cK(\bld{k}^{\concat n})/H^{\times n})$ with $i^{\Box n}$. 
Since the latter is an $h$-cofibration in $\Top^{\Sigma_n}$ (it is the realization of a 
$\Sigma_n$-equivariant map of simplicial sets) and every object in 
$\Top^{\mathcal K\times \Sigma_n}$ is $h$-cofibrant, we conclude from the 
pushout-product axiom for $h$-cofibrations \cite[Corollary 2.9]{Schwanzl-V_cofibrations}
that $f^{\Box n}$ is an $h$-cofibration in $\Top^{\mathcal K\times \Sigma_n}$. 
Applying Lemma \ref{lem:functors-preserving-h-cof} to the functor
$Z\boxtimes_{\Sigma_n}(-)$ from $(\cK \times \Sigma_n)$-spaces to
$\cK$-spaces then proves (vii).

The smallness requirement (viii) follows because $h$-cofibrations in 
$\Top^{\cK}$ are
object-wise $h$-cofibrations and any space is small relative to the $h$-cofibrations
in $\Top$ by \cite[Lemma 2.4.1]{Hovey-model}. This uses that in the category
of compactly generated weak Hausdorff spaces, a Hurewicz
cofibration is a closed inclusion.
\end{proof}

Notice, that the argument actually proves a stronger version of (vi): the functor 
$X\boxtimes (-)$ preserves $h$-cofibrations in general. Similarly, one can show that the map $Z\boxtimes_{\Sigma_n} f^{\Box n}$ in (vii) is an $h$-cofibration provided only that $f$ is an $h$-cofibration.
 
\subsection{Simplicial \texorpdfstring{$h$}{h}-cofibrations}\label{simplicial-h-cofibration}
The situation in the simplicial case is easier than in the topological
case. Here, we say that a map of $\mathcal K$-spaces is an $h$-cofibration 
if it is a levelwise cofibration of simplicial sets. We recall that the functor $\hocolim_{\cK}$ preserves colimits and sends $h$-cofibrations to cofibrations in $\cS$ \cite[Theorem
18.5.1]{Hirschhorn-model}.
\begin{proof}[Proof of Proposition \ref{prop:h-cofibration-properties} in
the simplicial setting]
The generating cofibrations are levelwise cofibrations, so (i)
follows and (ii) is clear.  For (iii), (iv), and (v), the same arguments as in the
topological case apply. For (vi), we first consider a generating $\cA$-relative 
cofibration of the form $G_{\bld k}^{\cK}(\cK(\bld k)/H\times i)$ with
$i$ a generating cofibration in $\cS$. Then $X\boxtimes f$ can be identified with the map $X\boxtimes G_{\bld k}^{\cK}(\cK(\bld k)/H)\times i$ which is clearly a level cofibration. This implies that $X\boxtimes(-)$ maps the generating $\cA$-relative cofibrations and hence all 
$\cA$-relative cofibrations to level cofibrations. An analogous argument using the description of $f^{\Box n}$ derived in the topological setting shows that $Z\boxtimes f^{\Box n}$ is a level cofibration.  Together with the fact that an injective map of $\Sigma_n$-sets induces an injective map on 
$\Sigma_n$-orbits this imply (vii). Finally, (viii) follows as in the topological case from the fact that any simplicial set 
is small relative to the cofibrations. (In fact, the category of $\mathcal K$-spaces is locally presentable in the simplicial setting, so all objects are small relative to the whole category.)
\end{proof}

\section{Monoidal properties of the model structures}\label{sec:monoidal-properties}
In this section we verify the pushout-product axiom and the monoid
axiom for the $\cA$-relative level and $\cK$-model structures on $\cS^{\cK}$ for a well-structured relative index category $(\cK,\cA)$.

We begin by stating an easy lemma which will also be useful in
the analysis of structured diagram spaces in the next section.
The lemma concerns the following situation: We
have a $\cK$-equivalence $X\to Y$ which is $G$-equivariant with
respect to an action of a discrete group $G$ and we would like to
conclude that the induced map of $G$-orbit spaces is again a
$\cK$-equivalence.
\begin{lemma}\label{lem:orbit-lemma}
  Let $G$ be a discrete group and consider a commutative diagram of 
  $G$-objects in $\cS^{\cK}$,
\[
\xymatrix@-1pc{
X\ar[rr]\ar[rd]_p & &Y\ar[dl]^q\\
& E &
}
\]
where we assume that the $G$-action on $E(\bld{k})$ is free for all objects
$\bld{k}$ in $\cK$. In the topological case we make the additional
assumption that either (i) $E$ is the geometric realization of a $\cK$-diagram in $G$-simplicial sets, or (ii) that $G$ is finite and $E$ is object-wise Hausdorff (and not just weak Hausdorff). Then if $X\to Y$ is a $\cK$-equivalence so is the induced map of orbit $\cK$-spaces $X/G\to Y/G$. The analogous statement holds for the 
$\cA$-relative level equivalences. 
\end{lemma}
When applying this to a given map $X\to Y$ we usually only specify the map $q$ in the lemma. It will then be understood that $p$ is defined as the composition.
\begin{proof}
Passing to $G$-orbits commutes with homotopy colimit, so it suffices to check that if $X_{h\cK}\to Y_{h\cK}$ is a weak equivalence,  then 
$X_{h\cK}/G\to Y_{h\cK}/G$ is also a weak equivalence. We shall see that 
the stated conditions on $E$ imply that the projections $X_{h\cK}\to X_{h\cK}/G$ and  $Y_{h\cK}\to Y_{h\cK}/G$ are covering maps. The conclusion in the lemma will then follow from the exact sequence of homotopy groups associated with a covering map.

Let us first consider the simplicial case. The condition that $E$ be object-wise $G$-free implies that $E_{h\cK}$ is $G$-free, hence $X_{h\cK}$ and 
$Y_{h\cK}$ are also $G$-free. The conclusion in the simplicial setting therefore follows from the general fact that the projection onto the orbit space of a free group action is a covering map.  
 
Next consider the topological case and suppose that the assumption (i) in the lemma holds. Then, by the above discussion, $E_{h\cK}\to E_{h\cK}/G$ is the geometric realization of a simplicial covering map, hence a topological covering map \cite[III, Satz 7.6]{Lamotke-semi}. This is equivalent to the condition that $G$ acts properly discontinuously on $E_{h\cK}$: For every point $e$ in $E_{h\cK}$ there is an open neighborhood $U_e$ such that $gU_e \cap U_e \neq \emptyset$ implies $g=1$. Pulling such a neighborhood back via $p_{h\cK}$ or $q_{h\cK}$ shows that $G$ also acts properly discontinuously on $X_{h\cK}$ and 
$Y_{h\cK}$ which gives the result. Now suppose instead that the assumption (ii) in the lemma holds. The condition that $E$ be object-wise Hausdorff implies that $E_{h\cK}$ is also Hausdorff, see e.g.~\cite[Lemma 11.3]{May_geometry}. Since we assume $G$ to be finite this ensures that $G$ acts properly discontinuously on $E_{h\cK}$ and the argument now proceeds as above. 

For the $\cA$-relative level equivalences, we apply the previous argument before taking homotopy colimits. 
\end{proof}

The verification of the pushout-product axiom and the monoid axiom is based on the next proposition which is useful in its own right. 

\begin{proposition}\label{prop:boxtimes-flat-invariance}
  If the $\cK$-space $X$ is $\cA$-relative cofibrant, then the
  functor $X\boxtimes(-)$ preserves $\cA$-relative level equivalences and 
  $\cK$-equivalences.
\end{proposition}
\begin{proof}
  We start with the $\cK$-equivalences and first consider the case
  where $X$ has the form $G_{\bld k}^{\cK}(\cK(\bld k)/H\times L)$ for an object $\bld k$ in $\cA$, a subgroup $H\subseteq \cA(\bld k)$, and a space 
 $L$ with trivial $\cK(\bld k)$-action. Then $X\boxtimes Y$ is isomorphic to $(F^{\cK}_{\bld k}(L)\boxtimes Y)/H$ for any $\cK$-space $Y$ (where $H$ acts trivially on $Y$).  We know from 
 Lemma~\ref{lem:k-Kan-extension} that 
  $(F^{\cK}_{\bld{k}}(*)\boxtimes Y)(\bld{m})$ is naturally isomorphic to 
  $\colim_{\bld{k}\concat\bld{l}\to\bld{m}}Y(\bld{l})$; the colimit calculated over  
  the category $(\bld{k}\sqcup -\downarrow \bld{m})$.
  The colimit decomposes as a coproduct indexed by the components of
  this category and the same holds for the analogous homotopy colimit. 
  Each of the connected components of the category in question has a
  terminal object by condition (ii) for a well-structured relative index category.  
  Hence the Kan extension and the homotopy Kan extension of $Y$ along the
   functor $\mathbf k\sqcup-$  are equivalent, i.e., the canonical map
\[
\hocolimwithlimits_{\bld{k}\concat\bld{l}\to\bld{m}}Y(\bld{l})
\to \colimwithlimits_{\bld{k}\concat\bld{l}\to\bld{m}}Y(\bld{l})
\]  
is a weak equivalence for each $\bld{m}$. Given a map $Y\to Z$ we therefore have a commutative diagram
\[
\xymatrix@-1pc{
(F_{\bld{k}}^{\cK}(L)\boxtimes Y)_{h\cK} \ar[rr] & &
(F_{\bld{k}}^{\cK}(L)\boxtimes Z)_{h\cK}\\
L\times \displaystyle{\hocolimwithlimits_{\bld{m}\in \cK}}\ 
\displaystyle{\hocolimwithlimits_{\bld{k}\concat\bld{l}\to\bld{m}}} 
Y(\bld l) \ar[rr]\ar[u]_{\sim}\ar[d]^{\sim} & &
L\times \displaystyle{\hocolimwithlimits_{\bld{m}\in \cK}}\ 
\displaystyle{\hocolimwithlimits_{\bld{k}\concat\bld{l}\to\bld{m}}} 
Z(\mathbf l) \ar[u]_{\sim}\ar[d]^{\sim}\\
L\times Y_{h\cK} \ar[rr] & &L\times Z_{h\cK}
}
\]
in which the vertical maps are weak equivalences as indicated. (The
vertical maps in the bottom part of the diagram are the canonical weak
equivalences associated to the homotopy colimit of a homotopy Kan
extension, see e.g.~\cite[Lemma 1.4]{Schlichtkrull-infinite}.) Thus,
if $Y\to Z$ is a $\cK$-equivalence, then the map on the top of the
diagram is a weak equivalence. Consider then the canonical map 
$F_{\bld{k}}^{\cK}(L)\boxtimes Z\to F_{\bld{k}}^{\cK}(*)\boxtimes * $
induced by the projections onto the terminal objects $*$ in $\cS$ and
$\cS^{\cK}$. It follows from Lemma \ref{lem:free-right-action}  that the
$H$-action on the target of this map is object-wise
free, hence Lemma \ref{lem:orbit-lemma} ensures that the top map in
the above diagram induces a weak equivalence of
$H$-orbit spaces.  This is exactly the statement of
the proposition when $X$ has this special form.
  
For the next step we assume that $X_0 \boxtimes(-)$ preserves
$\cK$-equivalences and that $X_1$ arises from $X_0$ 
as the pushout obtained by attaching an $\cA$-relative generating cofibration. Then parts (iv) and (vi) of 
Proposition~\ref{prop:h-cofibration-properties} and the preceding  argument show that $X_1 \boxtimes(-)$ also preserves $\cK$-equivalences.

In general, any $\cA$-relative cofibrant $\mathcal K$-space 
$X$ is a retract of a cell complex constructed from the 
generating $\cA$-relative cofibrations, hence we may assume that $X$ is itself such a cell complex. This means that there is an ordinal $\lambda$ and a $\lambda$-sequence 
$\{X_\alpha\colon \alpha<\lambda\}$ such that $X_0=\emptyset$, $X=\colim_{\alpha<\lambda}X_{\alpha}$, and each of the maps 
$X_{\alpha}\to X_{\alpha+1}$ is the pushout of a generating $\cA$-relative 
cofibration. By an inductive argument based on the above we conclude that each of the functors $X_{\alpha}\boxtimes (-)$ preserves 
$\cK$-equivalences and the conclusion now follows from parts (v) and (vi) of Proposition \ref{prop:h-cofibration-properties}
and the homotopy invariance of homotopy colimits.

The statement for the $\cA$-relative level equivalences follows from an analogous induction argument where again the basic case is a consequence of  Lemma \ref{lem:free-right-action}.
\end{proof}

\subsection{The pushout-product axiom} Recall that given maps 
$f_1\colon X_1\to Y_1$ and $f_2\colon X_2\to Y_2$ in $\cS^{\cK}$, 
the pushout-product is the induced map
\[
f_1\Box f_2\colon Y_1\boxtimes X_2\cup_{X_1\boxtimes X_2}X_1\boxtimes
Y_2\to Y_1\boxtimes Y_2.
\]
We say that a model structure on $\cS^{\cK}$ satisfies the pushout-product axiom if given two cofibrations $f_1$ and $f_2$, the pushout-product $f_1 \Box f_2$ is a cofibration that is in addition acyclic if $f_1$ or $f_2$ is. By definition, see \cite{Schwede_S-algebras}, a monoidal model category is a closed symmetric monoidal category equipped with a model structure that satisfies the pushout-product axiom. 

\begin{proposition}[The pushout-product
  axiom]\label{prop:K-pushout-product-axiom}
  The $\cA$-relative level and $\cK$-model structures on 
  $\cS^{\cK}$ satisfy the pushout-product axiom.
\end{proposition}
\begin{proof}
  It suffices by~\cite[Lemma 3.5]{Schwede_S-algebras} to consider the generating (acyclic) cofibrations for the model structures. We start with the 
  $\cA$-relative $\cK$-model structure. For $s=1,2$, assume we
  are given generating cofibrations 
  $f_s =G_{\bld{k_s}}^{\cK}(\cK(\bld{k_s})/H_s \times i_s)$ with $\bld k_s$ in 
  $\cA$, $H_s\subseteq \cA(\bld k_s)$, and $i_s$ a generating cofibration in 
  $\cS$. Then Lemma \ref{lem:product_of_free_Jspaces} implies that $f_1  
  \Box f_2$ is isomorphic to 
\[
G_{\bld{k_1}\concat\bld{k_2}}^{\cK}(\cK(\bld{k_1}\concat\bld{k_2})/(H_1
\times H_2) \times i_1 \Box i_2), \] where $i_1\Box i_2$ is the pushout product in 
$\mathcal S$.  The latter is a cofibration because the pushout-product axiom in $\cS$ holds and the condition that $\cA$ be a multiplicative subcategory of automorphisms therefore implies that $f_1\Box f_2$ is an $\cA$-relative 
cofibration. Suppose then that $f_1\colon X_1\to Y_1$ is a generating acyclic 
$\cA$-relative cofibration and that $f_2\colon X_2\to Y_2$ is a generating 
$\cA$-relative cofibration. Consider the diagram
\[
\xymatrix@-1pc{
Y_1\boxtimes X_2 \ar[d]_{\id} & X_1\boxtimes X_2 \ar[l] \ar[r]\ar[d] & X_1\boxtimes Y_2 \ar[d]\\
Y_1\boxtimes X_2 & Y_1\boxtimes X_2 \ar[l]_{\id}\ar[r]& Y_1\boxtimes Y_2
}
\]  
and notice that $f_1\Box f_2$ may be identified with the induced map
of pushouts. Since $X_2$ and $Y_2$ are $\cA$-relative cofibrant by definition, it follows from Proposition \ref{prop:boxtimes-flat-invariance} that the vertical maps are $\cK$-equivalences. The map of pushouts is therefore also a 
$\cK$-equivalence by Proposition \ref{prop:h-cofibration-properties}(iv).

The cofibrations for the $\cA$-relative level model structure
are the same as for the $\cA$-relative $\cK$-model structure and the 
second part of the argument follows from the same argument used above for the $\cK$-equivalences. 
\end{proof}

\subsection{The monoid axiom}
Recall from ~\cite{Schwede_S-algebras} that a
monoidal model category $\cC$ satisfies the monoid axiom if the
following holds: If $\cM$ denotes the class of maps of the form
$X\boxtimes (Y \to Z)$ with $X$ an arbitrary object and $Y \to Z$
an acyclic cofibration, then any map obtained from $\cM$
by cobase change and transfinite composition is a weak equivalence
in $\cC$. 

\begin{proposition}[The monoid axiom]\label{prop:monoid-axiom}
 The $\cA$-relative level and $\cK$-model structures on 
  $\cS^{\cK}$ satisfy the monoid axiom.
\end{proposition}
\begin{proof}
  For each of the model structures it suffices by 
  \cite[Lemma 3.5]{Schwede_S-algebras} to consider the subclass of $\cM$ given  
  by the maps $X\boxtimes (Y\to Z)$ where $Y\to Z$ is a generating acyclic  
  cofibration. We start with the $\cA$-relative $\cK$-model structure. Let 
  $f\colon Y\to Z$ be a generating acyclic $\cA$-relative cofibration
  and let $X$ be an arbitrary $\cK$-space. We choose a cofibrant replacement
  $X' \to X$ and consider the diagram
  \[
  \xymatrix@-1pc{
    X'\boxtimes Y \ar[r]\ar[d] & X'\boxtimes Z\ar[d]\\
    X\boxtimes Y \ar[r] & X\boxtimes Z.  }
  \]
  Since $Y$ and $Z$ are $\cA$-relative cofibrant it follows
  from Proposition \ref{prop:boxtimes-flat-invariance} that the
  vertical and the upper horizontal maps are 
  $\cK$-equivalences, hence the same holds for $X \boxtimes f$.  Proposition
  \ref{prop:h-cofibration-properties}(vi) implies that $X \boxtimes f$ is also 
  an $h$-cofibration and since the property of being both a $\cK$-equivalence and an $h$-cofibration in preserved under pushouts and transfinite composition by Corollaries \ref{cor:pushout-K-or-level-h}
  and \ref{cor:transfinite-K-or-level-equivalence}, we get the result for the 
  $\cA$-relative  $\cK$-model structure. An analogous argument gives the 
  $\cA$-relative level version.
\end{proof}

\section{Structured diagram spaces}\label{sec:structured-section}
By an operad $\mathcal D$ in $\cS$ we understand a sequence of spaces
$\mathcal D(n)$ in $\cS$ for $n\geq 0$ such that $\mathcal D(0)=*$, 
there is a unit $\{1\}\to\mathcal D(1)$, each of the spaces $\mathcal D(n)$
comes equipped with an action of $\Sigma_n$, and there are structure
maps
\[
\gamma\colon \mathcal D(k)\times \mathcal
D(j_1)\times\dots\times\mathcal D(j_k)\to \mathcal D(j_1+\dots+j_k)
\]
satisfying the defining relations listed in~\cite[Definition
1.1]{May_geometry}.

\begin{definition}\label{def:Sigma-free}
  An operad $\mathcal D$ in $\cS$ is $\Sigma$-free if the
  $\Sigma_n$-action on $\mathcal D(n)$ is free for all $n$. In the
  topological setting $\cS=\Top$ we make the additional
  assumption that the spaces of a $\Sigma$-free operad be Hausdorff
  (not just weak Hausdorff).
\end{definition}
With this terminology, an $E_{\infty}$ operad in the sense of
\cite{May_geometry} is the same thing as a $\Sigma$-free operad
$\mathcal D$ (in the topological setting) such that each of the spaces
$\mathcal D(n)$ is contractible. An operad $\mathcal D$ in $\cS$ gives
rise to a monad $\mathbb D$ on $\cS^{\cK}$ in the usual way by letting
\begin{equation}\label{eq:monoad-from-operad}
\mathbb D(X)=\coprod_{n\geq 0}\mathcal D(n)\times_{\Sigma_n}X^{\boxtimes n}.
\end{equation}
Here $X^{\boxtimes 0}$ denotes the unit $\bld1_{\cK}$
for the monoidal structure on $\cS^{\cK}$. We write $\cS^{\cK}[\mathbb
D]$ for the category of $\mathbb D$-algebras in $\cS^{\cK}$. By a {\em
  structured $\cK$-space} we understand a $\cK$-space equipped with
such an algebra structure.
\begin{lemma}\label{structured-cocomplete}
  The category $\cS^{\cK}[\mathbb D]$ is complete and cocomplete and
  the forgetful functor from $\cS^{\cK}[\mathbb D]$ to $\cS^{\cK}$
  preserves limits and filtered colimits.
\end{lemma}
\begin{proof}
  By~\cite[Proposition 4.3.6]{Borceux_Handbook2}, the category
  $\cS^{\cK}[\mathbb D]$ is complete and cocomplete provided that the functor
  $\mathbb D \colon \cS^{\cK} \to \cS^{\cK}$ preserves filtered
  colimits.  By the definition of $\mathbb D$
  in~\eqref{eq:monoad-from-operad}, this reduces to showing that
  $(-)^{\boxtimes n}$ preserves filtered colimits. Since $\cS^{\cK}$
  is closed monoidal, the iterated $\boxtimes$-product preserves
  colimits as a functor $(\cS^{\cK})^{\times n} \to
  \cS^{\cK}$. Combining this with the fact that the diagonal functor
  associated with a filtered category is final in the sense
  of~\cite[Section IX.3]{MacLane_working} shows the claim.
  The forgetful functor preserves limits by~\cite[Proposition
  4.3.1]{Borceux_Handbook2}, and~\cite[Proposition~4.3.2]{Borceux_Handbook2} shows that it also preserves
  filtered colimits since $\mathbb D \colon \cS^{\cK} \to \cS^{\cK}$
  does.
\end{proof}

Given a well-structured relative index category $(\cK,\cA)$, we say that the 
$\cA$-relative level or the $\cA$-relative $\cK$-model structure on $\cS^{\cK}$ lifts to $\cS^{\cK}[\mathbb D]$ if there exists a model
structure on $\cS^{\cK}[\mathbb D]$ whose weak equivalences and
fibrations are the weak equivalences and fibrations of the underlying
$\cK$-spaces. Since our model structures on $\cS^{\cK}$ are
cofibrantly generated, there is a well-known strategy for constructing
such a lift: Let us generically write $I_{\cK}$ for the generating
cofibrations and $J_{\cK}$ for the generating acyclic cofibrations for the
given model structure on $\cS^{\cK}$. Then we let $\mathbb D(I_{\cK})$ and
$\mathbb D(J_{\cJ})$ be the corresponding sets of morphisms in
$\cS^{\cK}[\mathbb D]$, obtained by applying the functor $\mathbb D$,
and we may ask if these sets satisfy the conditions in the recognition
principle for cofibrantly generated model categories~\cite[Theorem
2.1.19]{Hovey-model}.  This will not always be the case as for instance not every $\cA$-relative $\cK$-model structure lifts to the category of
commutative $\cK$-space monoids $\cC\cS^{\cK}$ (which is the same as 
$\cS^{\cK}[\mathbb C]$ where $\mathbb
C$ is the monad associated to the commutativity operad $\mathcal C$
with $\mathcal C(n)=*$ for all $n$). Roughly speaking, for this
strategy to be successful, the $\Sigma_n$-action on $\mathcal
D(n)\times X^{\boxtimes n}$ must be sufficiently free and this can be
obtained by imposing either a freeness assumption on the operad spaces
$\mathcal D(n)$ or on the iterated $\boxtimes$-products 
$X^{\boxtimes  n}$. This is reflected in the following lifting result where we use the notion of a very well-structured relative index category introduced in Definition~\ref{def:very-well-structured}.

\begin{proposition}\label{prop:structured-lift-proposition}
The $\cA$-relative level and $\cK$-model structures on 
$\cS^{\cK}$ lift to model structures on $\cS^{\cK}[\mathbb D]$ provided that either (i) $\mathcal D$ is $\Sigma$-free, or (ii) $(\cK,\cA)$ is very 
well-structured. Under one of these assumptions, the lifted model structure on $\cS^{\cK}[\mathbb D]$ is cofibrantly generated with  generating (acyclic) cofibrations obtained by applying $\mathbb D$ to the generating (acyclic) cofibrations for the given model structure on $\cS^{\cK}$.
\end{proposition}

We shall also use the terms \emph{$\cA$-relative level} and \emph{$\cA$-relative $\cK$-model structure} for the lifted model structures on 
$\cS^{\cK}[\mathbb D]$.

 For the proof of Proposition~\ref{prop:structured-lift-proposition}
 we need some results on pushouts in $\cS^{\cK}[\mathbb D]$.  Suppose
 we are given a $\mathbb D$-algebra $A$, a map of $\cK$-spaces
 $f\colon X\to Y$, and a map of $\cK$-spaces $X\to A$. Consider the
 associated pushout diagram in $\cS^{\cK}[\mathbb D]$:
\begin{equation}\label{D-structured-pushout}
\xymatrix@-1pc{
\mathbb D(X)\ar[r]^{\mathbb D(f)}\ar[d] &\mathbb D(Y)\ar[d]\\
A\ar[r]^{\bar f} &B
}
\end{equation}

\begin{lemma}\label{structured-h-cofibration-pushout}
  Let $\mathcal D$ be any operad in $\cS$ and suppose that the map $f$
  in \eqref{D-structured-pushout} is an $\cA$-relative cofibration. Then 
  $\bar f$ is an $h$-cofibration.
\end{lemma}

\begin{lemma}\label{structured-acyclic-pushout}
   Suppose that the map $f$ in (\ref{D-structured-pushout}) is a
   generating acyclic cofibration for the $\cA$-relative level model structure (respectively the $\cA$-relative $\cK$-model structure). Then $\bar f$ is an 
   $\cA$-relative level equivalence (respectively a $\cK$-equivalence) of the underlying $\cK$-spaces provided that either (i) $\mathcal D$ is 
   $\Sigma$-free, or (ii) $(\cK,\cA)$ is very well-structured.
\end{lemma}

The proofs of these lemmas are based on a careful analysis of this
kind of pushout diagrams and will be given in Section
\ref{structured-pushout-proofs}.

\begin{proof}[Proof of Proposition \ref{prop:structured-lift-proposition}]
  We use the criterion for lifting model structures to categories of
  algebras formulated in~\cite[Lemma 2.3]{Schwede_S-algebras} (which
  in turn is an easy consequence of the general recognition principle
  for cofibrantly generated model categories \cite[Theorem
  2.1.19]{Hovey-model}). Thus, we must check the following two conditions:
\begin{enumerate}[(i)]
\item If $K$ is the set of generating cofibrations or generating
  acyclic cofibrations for the given model structure on
  $\cS^{\cK}$, then the domains of the maps in $\mathbb D(K)$ are
  small relative to $\mathbb D(K)\cell$.
\item If $J_{\cK}$ is the set of generating acyclic cofibrations for the given
 model structure on $\cS^{\cK}$, then the maps in 
 $\mathbb D(J_{\cK})\cell$ are
  weak equivalences in the same model structure. 
  
\end{enumerate}
The first condition follows from the fact that if $X$ is any
$\cK$-space, then $\mathbb D(X)$ is small relative to $\mathbb
D(K)\cell$. Indeed, by Lemma \ref{structured-h-cofibration-pushout}, a
relative $\mathbb D(K)$-cell complex in $\cS^{\cK}[\mathbb D]$ is the
transfinite composition of a sequence whose underlying maps in
$\cS^{\cK}$ are $h$-cofibrations. Since we know from Lemma
\ref{structured-cocomplete} that the forgetful functor from
$\cS^{\cK}[\mathbb D]$ to $\cS^{\cK}$ preserves filtered colimits, the
above smallness claim follows from the adjointness property and the
fact that by Proposition \ref{prop:h-cofibration-properties}(viii) all
$\cK$-spaces are small relative to the $h$-cofibrations.

For the second condition we use that by Lemma
\ref{structured-acyclic-pushout} a relative $\mathbb D(J_{\cK})$-cell
complex is the transfinite composition of a sequence of maps each of
which is a weak equivalence in the given model structure on $\cS^{\cK}$. Since these maps are also $h$-cofibrations by 
Lemma~\ref{structured-h-cofibration-pushout}, it follows from 
Corollary~\ref{cor:transfinite-K-or-level-equivalence} that the transfinite
composition is a weak equivalence.
\end{proof}

\begin{remark}\label{rem:model-str-D-spaces}
Let $\mathcal K$ be the terminal category $*$, viewed as a well-structured index category. For a $\Sigma$-free operad $\mathcal D$,  the  model structure in Proposition  \ref{prop:structured-lift-proposition} is the usual model structure on $\cS[\mathbb D]$, see for example \cite[Section~4]{Berger_M-axiomatic}.
\end{remark}

\begin{remark}
Recall from Remark \ref{rem:simpl-top-K-model} the adjoint functors $|-|$ and $\Sing$ relating the simplicial and topological versions of $\cK$-spaces. For a monad $\mathbb D$ associated to an operad $\cD$ in simplicial sets, write $|\mathbb D|$ for the monad associated to the topological operad $|\cD|$ obtained by geometric realization. The fact that $|-|$ is strong monoidal and $\Sing$ is (lax) monoidal easily implies that the above model structures on $\mathbb D$-algebras in the simplicial setting are Quillen equivalent to the corresponding model structures on $|\mathbb D|$-algebras in the topological setting whenever they exist. (See the remarks at the beginning of Section \ref{subs:structured-comma} for a general discussion of such adjunctions). This implies in particular that $|-|$ and $\Sing$ gives rise to Quillen equivalences between the relevant model structures on (commutative) $\cK$-space monoids in the simplicial and topological settings. 
\end{remark}

By definition of the lifted model structures we have the following structured version of the Quillen equivalence in Proposition \ref{prop:well-structured-comparison}.

\begin{proposition}\label{prop:operad-well-structured-comparison}
Let $(\cK,\cB)$ be a well-structured relative index category. Suppose that 
$\cA$ is a normal and multiplicative subcategory of automorphisms contained in $\cB$ and that the inclusion $\cK_{\cA}\to \cK_{\cB}$ is homotopy cofinal. Then $(\cK,\cA)$ is a well-structured relative index category and the following hold:
\begin{enumerate}[(i)]
\item
If $\cD$ is a $\Sigma$-free operad, then the $\cA$- and $\cB$-relative $\cK$-model structures on $\cS^{\cK}[\mathbb D]$ are Quillen equivalent.

\item
If $\cD$ is any operad and $(\cK,\cB)$ is very well-structured, 
then $(\cK,\cA)$ is also very well-structured and the $\cA$- and $\cB$-relative $\cK$-model structures on $\cS^{\cK}[\mathbb D]$ are Quillen equivalent.  \qed 
\end{enumerate}
\end{proposition}

It follows from general results for operads acting on objects in suitable symmetric monoidal $\cS$-model categories that the model structures on $\cS^{\cK}[\mathbb D]$ considered above can be viewed as $\cS$-model categories in a canonical way. A detailed account of how this works in a topological setting can be found in~\cite[VII.2]{EKMM}. In order to simplify the discussion we shall only consider the case of greatest interest to us: the simplicial model structure on the category of commutative 
$\cK$-space monoids $\cC\cS^{\cK}$. In this case the cotensor is defined on
the underlying $\cK$-spaces whereas the tensor and simplicial mapping spaces are defined by
\[ A \tensor K = | [n] \mapsto A^{\boxtimes K_n} | \qquad \text{and}
\qquad \Map(A,B) = \big\{[n]\mapsto \cC\cS^{\cK}(A\tensor
\Delta^{n},B)\big\} 
\] 
for $A$ and $B$ in $\cC\cS^{\cK}$ and $K$ a simplicial set. 

\begin{proposition}\label{prop:CSK-simplicial}
Let $(\cK,\cA)$ be a very 
well-structured index category. Then the $\cA$-relative $\cK$-model structure on $\cC\cS^{\cK}$
is simplicial.
\end{proposition}
\begin{proof}
The condition for being a simplicial model category can be expressed in terms of the 
  cotensor structure (see e.g.~\cite[Lemma 4.2.2]{Hovey-model}). Using this, the result follows from the fact that the $\cA$-relative $\cK$-model structure on $\cS^{\cK}$ is simplicial by Proposition \ref{prop:K-model-is-S-model-str} and the remarks following that proposition. 
\end{proof}
Now we specialize to a well-structured index category $\cK$ with homotopy cofinal inclusion $\cK_+\to\cK$  and the problem of lifting the positive projective $\cK$-model structure in Definition \ref{def:K-projective-model} to $\cC\cS^{\cK}$. By definition, $(\cK,O\cK_+)$ is very well-structured if $\Sigma_n$ acts freely on the set of connected components of $(\bld k^{\sqcup n}\sqcup-
\downarrow \bld l)$ for each pair of objects $\bld k$ and $\bld l$ in $\cK_+$.

\begin{corollary}\label{cor:K-positive-projective-commutative}
Let $\cK$ be a well-structured index category and suppose that the inclusion 
$\cK_+\to \cK$ is homotopy cofinal and that $(\cK,O\cK_+)$ is very 
well-structured. Then the positive projective $\cK$-model structure on 
$\cS^{\cK}$ lifts to a simplicial model structure on $\cC\cS^{\cK}$. \qed
\end{corollary}
\subsection{Change of operads}\label{operad-change-section}
In this section we analyze how our categories of structured diagram
spaces behave under change of operads. Thus, consider a map of operads
$\Phi\colon \mathcal D\to\mathcal E$ and the associated map of monads
$\Phi\colon \mathbb D\to\mathbb E$. This gives rise to a pair of
adjoint functors
\[
\Phi_*\colon \cS^{\cK}[\mathbb D]\rightleftarrows 
\cS^{\cK}[\mathbb E]\thinspace\colon\Phi^*
\] 
where the right adjoint $\Phi^*$ is defined by pulling an $\mathbb
E$-algebra structure on an object back to a $\mathbb D$-algebra
structure via $\Phi$. The $\mathbb E$-algebra $\Phi_*(A)$ associated to a
$\mathbb D$-algebra $A$ with structure map $\xi\colon\mathbb D(A)\to A$
can be represented by a (reflexive) coequalizer diagram in $\mathcal A$,
\[
\xymatrix{
\mathbb E\mathbb D(A)\ar@<0.5ex>[r]^{\partial_0}\ar@<-.5ex>[r]_{\partial_1}& \mathbb E(A) 
\ar[r] &\Phi_*(A),
}
\]
where $\partial_0=\mathbb E(\xi)$ and $\partial_1=\mu_A\circ \mathbb
E(\Phi_A)$, see Section \ref{coequalizer-section} for further details. We say that $\Phi$ is a \emph{weak equivalence} if each of the maps $\Phi_n\colon \mathcal D(n)\to\mathcal E(n)$ is a weak equivalence. The next proposition shows that a weak equivalence of operads induces a Quillen equivalence whenever the lifted model structures on 
$\cS^{\cK}[\mathbb D]$ and $\cS^{\cK}[\mathbb E]$ are defined.

\begin{proposition}\label{prop:operad-change}
Let $(\cK,\cA)$ be a well-structured relative index category and let 
$\Phi\colon \mathcal D\to\mathcal E$ be a weak equivalence of operads. Suppose that either (i) both $\cD$ and $\mathcal E$ are 
$\Sigma$-free, or (ii) $(\cK,\cA)$ is very well-structured. Then the adjoint functor pair $(\Phi_*,\Phi^*)$ defines a Quillen equivalence between the 
$\cA$-relative $\cK$-model structures on $\cS^{\cK}[\mathbb D]$ and 
$\cS^{\cK}[\mathbb E]$.
\end{proposition}

The main technical point in establishing this proposition is the
homotopical analysis of the unit $A\to\Phi^*\Phi_*(A)$ of the
adjunction. The proof of the below lemma requires the same kind of
analysis as the proofs of Lemmas \ref{structured-h-cofibration-pushout} and \ref{structured-acyclic-pushout}, and will be given in Section 
\ref{subs:unit-lemma-section}.

\begin{lemma}\label{lem:unit-lemma}
  Suppose that $\Phi\colon \mathcal D\to\mathcal E$ is a weak
  equivalence and that $A$ is a $\mathbb D$-algebra which is cofibrant in the 
  $\cA$-relative $\cK$-model structure on $\cS^{\cK}[\mathbb D]$. Then the
   unit for the adjunction $A\to\Phi^*\Phi_*(A)$ is a level equivalence provided
    that either (i) the operads $\mathcal D$ and 
  $\mathcal E$ are $\Sigma$-free, or (ii) $(\cK,\cA)$ is very well-structured.
\end{lemma}

\begin{proof}[Proof of Proposition \ref{prop:operad-change}]
  It is clear that $(\Phi_*,\Phi^*)$ is a Quillen adjunction since
  $\Phi^*$ preserves fibrations and acyclic fibrations. Given a
  cofibrant object $A$ in $\cS^{\cK}[\mathbb D]$ and a fibrant object
  $B$ in $\cS^{\cK}[\mathbb E]$, we must show that a map $\Phi_*(A)\to
B$ is a $\cK$-equivalence if and only if the adjoint map
  $A\to \Phi^*(B)$ is. The latter map admits a factorization
  \[
  A\to\Phi^*\Phi_*(A)\to \Phi^*(B)
  \]
  so the result follows from Lemma \ref{lem:unit-lemma} and the 2 out
  of 3 property for $\cK$-equivalences.
\end{proof}
Specializing to a well-structured index category $\cK$ such that 
$(\cK,O\cK_+)$ is very well-structured, we use the above to rectify 
$E_{\infty}$ objects in $\cS^{\cK}$ to strictly commutative monoids. 

\begin{corollary}\label{cor:positive-projective-rectification}
Let $\cD$ be an $E_{\infty}$ operad. Suppose that $\cK$ is a well-structured index category such that $\cK_+\to\cK$ is homotopy cofinal and 
$(\cK,O\cK_+)$ is very well-structured. Then the projective $\cK$-model structure on 
$\cS^{\cK}[\mathbb D]$ is related to the positive projective $\cK$-model structure on $\cC\cS^{\cK}$ by a chain of Quillen equivalences.
\end{corollary}
\begin{proof}
 Writing $\pi\colon \mathbb D\to \mathbb C$ for the canonical weak equivalence to the commutativity monad $\mathbb C$, we have a chain of adjoint functors
\[
\xymatrix@-1pc{
(\cC\cS^{\cK},\text{positive projective $\cK$-model structure})
\ar@<1.0ex>[d]^{\pi^*}\\
(\cS^{\cK}[\mathbb D],\text{positive projective $\cK$-model structure})
\ar@<1.0ex>[u]^{\pi_*} \ar@<-1.0ex>[d]\\
(\cS^{\cK}[\mathbb D],\text{projective $\cK$-model structure}).
\ar@<-1.0ex>[u]
}
\]
According to Propositions \ref{prop:operad-well-structured-comparison}  
and \ref{prop:operad-change}, these Quillen adjunctions specify the required Quillen equivalences.
\end{proof}

\begin{remark}\label{rem:levelwise-version}
The fact that by Lemma \ref{lem:unit-lemma} the unit  $A\to\Phi^*\Phi_*(A)$ for the adjunction  is a level equivalence  implies that there is a levelwise version of Proposition \ref{prop:operad-change}. We shall not use this and leave the details to the interested reader.
\end{remark}

\section{Verification of structured diagram lemmas}
\label{sec:verification}
Let again $\cD$ be an operad in $\cS$ and let $\mathbb D$ be the associated monad on $\cS^{\cK}$. In order to prove the technical lemmas on structured diagram spaces stated in Section~\ref{sec:structured-section}, we shall analyze pushout diagrams in $\cS^{\cK}[\mathbb D]$ of the form
\begin{equation}\label{verification-structured-pushout}
\xymatrix@-1pc{
\mathbb D(X)\ar[r]^{\mathbb D(f)}\ar[d] &\mathbb D(Y)\ar[d]\\
A\ar[r]^{\bar f} &B
}
\end{equation}
for a map $f\colon X\to Y$ in $\cS^{\cK}$. For this we essentially follow Elmendorf-Mandell~\cite{Elmendorf-M_infinite-loop} by introducing a filtration of the induced map $\bar f$ such that the passage from the $(k-1)$th to the $k$th stage of the filtration can be expressed in terms of a certain 
$\Sigma_k$-equivariant $\cK$-space $U_k^{\mathbb D}(A)$ constructed from 
$A$, together with data derived from the map $f$. Since we shall in fact need a refined version of the filtration considered in \cite{Elmendorf-M_infinite-loop} we have included an exposition of this material in Appendix \ref{app:analysis-operad-algebras}. In the below proposition we extract the facts about the functors $U_k^{\mathbb D}$ needed for the proofs of the above mentioned lemmas. Here $(\cS^{\cK})^{\Sigma_k}/\cD(k)$ denotes the category of 
$\Sigma_k$-equivariant $\cK$-spaces over the constant $\cK$-space $\cD(k)$ and $Q_{i-1}^i(f)$ is the domain of the $i$-fold iterated pushout-product map as in Section \ref{subs:h-cofibrations}. The statements in the proposition that are not obvious from the definitions appear as Lemma~\ref{freeshifted}, 
Lemma~\ref{terminal-projection}, and Proposition~\ref{prop:cat-pushout-filtration} in Appendix~\ref{app:analysis-operad-algebras}.

\begin{proposition}\label{prop:pushout-filtration}
There exists a sequence of functors $U_k^{\mathbb D}\colon \cS^{\cK}[\mathbb D]\to (\cS^{\cK})^{\Sigma_k}/\cD(k)$ for $k\geq 0$, such that the following hold:

\begin{enumerate}[(i)]
\item
$U_0^{\mathbb D}$ is the forgetful functor to $\cS^{\cK}$.
\item
The functors $U_k^{\mathbb D}$ preserve filtered colimits.
\item
$U_k^{\mathbb D}(\mathbb D(X))$ is isomorphic to $\coprod_{n\geq 0}\cD(n+k)\times_{\Sigma_n}X^{\boxtimes n}$ for any $\cK$-space $X$.
\item
For a pushout diagram of the type in (\ref{verification-structured-pushout}) 
there is a natural sequence of $\Sigma_k$-equivariant $\cK$-spaces 
\[
U_k^{\mathbb D}(A)=F_0 U_k^{\mathbb D}(B)\to 
F_1 U_k^{\mathbb D}(B)\to\dots\to F_iU_k^{\mathbb D}(B)\to\dots
\]
such that
$\colim_i F_iU_k^{\mathbb D}(B)=U_k^{\mathbb D}(B)$, the transfinite composition of the sequence equals $U_k^{\mathbb D}(\bar f)$, and
there are $\Sigma_k$-equivariant pushout diagrams in $\cS^{\cK}$,
\[
\xymatrix@-1pc{
  U_{i+k}^{\mathbb D}(A)\boxtimes_{\Sigma_i}Q_{i-1}^i(f) \ar[rr]\ar[d]& &U_{i+k}^{\mathbb D}(A)\boxtimes_{\Sigma_i}Y^{\boxtimes i} \ar[d]\\
  F_{i-1} U_k^{\mathbb D}(B) \ar[rr]& & F_iU_k^{\mathbb D}(B)
}
\]
for all $i\geq 1$.\qed
\end{enumerate}
\end{proposition}

For $k=0$, the filtration in (iv) is the filtration of $\bar f$ considered in 
\cite{Elmendorf-M_infinite-loop} (for symmetric spectra). 

\begin{example}\label{example:U_k-commutative}
If $\mathbb C$ is the monad associated to the commutativity operad, then 
$U_k^{\mathbb C}(A)=A$ with trivial $\Sigma_k$-action (see also Example 
\ref{ex:commutativity-example} for more details). 
\end{example}

\subsection{The proofs of Lemmas  \ref{structured-h-cofibration-pushout} 
and  \ref{structured-acyclic-pushout}}\label{structured-pushout-proofs}
\begin{proof}[Proof of Lemma \ref{structured-h-cofibration-pushout}]
  Applying the filtration from Proposition \ref{prop:pushout-filtration} to
  $B=U_0^{\mathbb D}(B)$ and writing $F_iB$ for the filtration terms
  $F_iU_0^{\mathbb D}(B)$, we get a sequence of pushout diagrams
  \[
  \xymatrix@-1pc{ U_{i}^{\mathbb D}(A)\boxtimes_{\Sigma_i}Q_{i-1}^i(f)
    \ar[rr] \ar[d]&&
    U_{i}^{\mathbb D}(A)\boxtimes_{\Sigma_i}Y^{\boxtimes i}\ar[d]\\
    F_{i-1}B\ar[rr] && F_iB }
  \]
  in $\cS^{\cK}$ such that $A=F_0B$ and $B=\colim_iF_iB$. Proposition
  \ref{prop:h-cofibration-properties}(vii) implies that the upper
  horizontal map in each of these diagrams is an $h$-cofibration, hence  the
  same holds for the maps $F_{i-1}B\to F_iB$ and therefore also for
  the transfinite composition $A\to B$ by Proposition  
  \ref{prop:h-cofibration-properties}(ii).
\end{proof}

\begin{proof}[Proof of Lemma \ref{structured-acyclic-pushout}]
  We begin with the $\cA$-relative $\cK$-model structure.  Using the filtration from Proposition \ref{prop:pushout-filtration} and arguing as in the 
  proof of Lemma \ref{structured-h-cofibration-pushout}, we conclude from
  Corollaries \ref{cor:pushout-K-or-level-h} and
 \ref{cor:transfinite-K-or-level-equivalence} that it is sufficient
  to show that
  \[
  U_i^{\mathbb D}(A)\boxtimes _{\Sigma_i}Q^i_{i-1}(f)\to U_i^{\mathbb
    D}(A)\boxtimes_{\Sigma_i}Y^{\boxtimes i}
  \]
  is a $\cK$-equivalence for all $i$. By the pushout-product
   axiom, Proposition \ref{prop:K-pushout-product-axiom}, we know that
    $Q_{i-1}^i(f)\to Y^{\boxtimes i}$ is an $\cA$-relative cofibration and a 
  $\cK$-equivalence. Hence it follows from the monoid axiom, Proposition
  \ref{prop:monoid-axiom}, that
  \[
  U_i^{\mathbb D}(A)\boxtimes Q_{i-1}^i(f)\to U_i^{\mathbb
    D}(A)\boxtimes Y^{\boxtimes i}
  \]
  is also a $\cK$-equivalence. Thus, it remains to show that the
  stated conditions on $\mathcal D$ and $(\cK,\cA)$ imply that the same holds
  for the induced map of $\Sigma_i$-orbits. Suppose first that
  $\mathcal D$ is $\Sigma$-free. We know from Proposition
  \ref{prop:pushout-filtration} that there is a $\Sigma_i$-equivariant map
  $U_i^{\mathbb D}(A)\to \mathcal D(i)$ onto the constant $\cK$-space
  $\mathcal D(i)$. From this we get the $\Sigma_i$-equivariant map
  \[
  U_i^{\mathbb D}(A)\boxtimes Y^{\boxtimes i}\to \mathcal
  D(i)\boxtimes *\cong \mathcal D(i)\times(*\boxtimes *)\to \mathcal
  D(i)\times *
  \]
  onto the constant $\cK$-space $\cD(i)$ 
  and since the $\Sigma_i$-action on the latter is object-wise free,
  the result follows from Lemma \ref{lem:orbit-lemma}. 

  Suppose then that $(\cK,\cA)$ is very well-structured. We first observe that there
  is a map $Y\to G^{\cK}_{\bld{k}}(\cK(\bld{k})/H)$ for some object $\bld{k}$ 
  in $\cA$ and subgroup $H\subseteq \cA(\bld{k})$.  From this we get the 
  $\Sigma_i$-equivariant map
  \[
  U_i^{\mathbb D}(A)\boxtimes Y^{\boxtimes i}\to *\boxtimes
  G^{\cK}_{\bld{k}}(\cK(\bld k)/H)^{\boxtimes i} \iso \left(* \boxtimes
  F^{\cK}_{\bld{k}^{\concat i}}(*)\right)/H^{\times i}.
  \]
    Using the isomorphism \eqref{eq:left-Kan}, the value of the $\cK$-space on the right at an object $\bld{m}$ in $\cK$ can be identified with the
  $H^{\times i}$-orbits of the set of connected components in
  the comma category $(-\concat\bld{k}^{\concat i}\downarrow
  \bld{m})$.  The $(\Sigma_i\ltimes H^{\times i})$-action on the set of
   connected components is free by the definition of a very  well-structured
    index category. Hence the $\Sigma_i$-action on the target is object-wise free and the result again follows from Lemma \ref{lem:orbit-lemma}.
  
For the $\cA$-relative level model structures we use the analogous argument with $\cA$-relative level equivalences instead of $\cK$-equivalences. 
\end{proof} 

\subsection{The proof of Lemma \ref{lem:unit-lemma}}
\label{subs:unit-lemma-section}
As in Section \ref{operad-change-section} we consider a map of operads
$\Phi\colon \mathcal D\to\mathcal E$ and the adjoint functor pair
$
\Phi_*\colon \cS^{\cK}[\mathbb D]\rightleftarrows 
\cS^{\cK}[\mathbb E]\thinspace\colon\Phi^*
$.
We first record the action of $\Phi_*$ on free $\mathbb D$-algebras which is a formal consequence of the fact that a composition of left adjoints is again a left adjoint.
\begin{lemma}
For a free $\mathbb D$-algebra $\mathbb D(X)$ we have
$\Phi_*\mathbb D(X)=\mathbb E(X)$.\qed
\end{lemma}

\begin{proof}[Proof of Lemma \ref{lem:unit-lemma}]
  Recall that we assume operads to be reduced such that the unit
  $\bld 1_{\cK}$ for the monoidal structure on $\cS^{\cK}$ is the initial
  object in $\cS^{\cK}[\mathbb D]$.  We may assume without loss of
  generality that $A$ is a cell complex in $\cS^{\cK}[\mathbb D]$
  in the sense that it is the colimit of a $\lambda$-sequence
  $\{A_{\alpha}\colon \alpha<\lambda\}$ (for an ordinal $\lambda$)
  such that $A_0=\bld1_{\cK}$ and $A_{\alpha}\to A_{\alpha+1}$ is a 
  pushout in $\cS^{\cK}[\mathbb D]$ of a generating cofibration 
  $\mathbb D(f_{\alpha})\colon \mathbb D(X_{\alpha})\to 
  \mathbb D(Y_{\alpha})$ (where $f_{\alpha}\colon X_{\alpha}\to Y_{\alpha}$ 
  is a generating $\cA$-relative cofibration in $\cS^{\cK}$). By Lemma 
 \ref{structured-cocomplete}, $A$ is also the colimit of the $\cK$-spaces
 $A_{\alpha}$ in the underlying category $\cS^{\cK}$. 
 
 Let us write $A'=\Phi_*(A)$ and
 $A_{\alpha}'=\Phi_*(A_{\alpha})$. Since $\Phi_*$ preserves colimits,
 $A'$ is the colimit of the $\lambda$-sequence $\{A'_{\alpha}\colon
 \alpha<\lambda\}$ in $\cS^{\cK}[\mathbb E]$ and 
 $A_{\alpha}'\to A_{\alpha+1}'$ is a pushout of the generating cofibration 
$\mathbb E(f_{\alpha})\colon \mathbb E(X_{\alpha})\to 
  \mathbb E(Y_{\alpha})$.
Again $A'$ is the colimit of the $\cK$-spaces $A_{\alpha}'$ in the underlying
category $\cS^{\cK}$. Since by
Lemma~\ref{structured-h-cofibration-pushout} these are
$\lambda$-sequences of $h$-cofibrations, it suffices by
Proposition~\ref{prop:h-cofibration-properties}(v) to show that $A_{\alpha}\to
A_{\alpha}'$ is a level equivalence in $\cS^{\cK}$ for all $\alpha$.
In order to set up an inductive argument using Proposition
\ref{prop:pushout-filtration}, we shall in fact prove the following stronger
statement: for each $\alpha<\lambda$ the map $U_k^{\mathbb
  D}(A_{\alpha})\to U_k^{\mathbb E}(A_{\alpha}')$ is a level
equivalence for all $k$. The result then follows by setting $k=0$.
Thus, given $\beta<\lambda$, we must show that if the statement holds
for all $\alpha<\beta$, then it also holds for $\beta$.  For $\beta=0$
we conclude from Proposition~\ref{prop:pushout-filtration}(iii) 
(using that $\mathbb D(\emptyset)=
\bld 1_{\cK}$ and $\mathbb E(\emptyset)=\bld1_{\cK}$) that the
maps in question can be identified with the maps $\mathcal D(k)\times
\bld 1_{\cK}\to \mathcal E(k)\times \bld1_{\cK}$ induced by $\Phi$. 
These maps
are level equivalences by assumption. If $\beta$ is a limit ordinal,
then it follows from Proposition~\ref{prop:h-cofibration-properties}(v)
and the definition of a $\lambda$-sequence that the statement for
$\alpha<\beta$ implies the statement for $\beta$. Thus, it remains to
consider the case where $\beta$ is a successor ordinal,
$\beta=\alpha+1$.  For the inductive step we use the filtration from
Proposition \ref{prop:pushout-filtration}(iv) and observe that there is a commutative diagram
\[
\xymatrix@-1pc{U_k^{\mathbb D}(A_{\alpha}) \ar@{=}[r]\ar[d]
  &F_0U_k^{\mathbb D}(A_{\alpha+1})\ar[r] \ar[d] &
  F_1U_k^{\mathbb D}(A_{\alpha+1})\ar[r] \ar[d] & F_2U_k^{\mathbb D}(A_{\alpha+1})\ar[r] \ar[d] &  \dots \\
  U_k^{\mathbb E}(A'_{\alpha}) \ar@{=}[r] &F_0U_k^{\mathbb
    E}(A'_{\alpha+1})\ar[r] & F_1U_k^{\mathbb E}(A'_{\alpha+1})\ar[r]
  & F_2U_k^{\mathbb E}(A'_{\alpha+1})\ar[r] & \dots }
\]
for each $k\geq 0$. Since the horizontal maps are $h$-cofibrations, by
Proposition \ref{prop:h-cofibration-properties}(ii) and (vii), it suffices to show
that the vertical maps are level equivalences for all $i$. It follows
from the naturality of the filtrations that the map in filtration degree $i$ can be identified with the map of pushouts
induced by the map of diagrams
\[
\xymatrix@-1pc{
F_{i-1}U_k^{\mathbb D}(A_{\alpha+1}) \ar[d] &
U_{i+k}^{\mathbb D}(A_{\alpha})\boxtimes_{\Sigma_i}Q^i_{i-1}(f)
\ar[l] \ar[r] \ar[d] & U_{i+k}^{\mathbb D}(A_{\alpha})\boxtimes_{\Sigma_i}Y^{\boxtimes i} \ar[d] \\
F_{i-1}U_k^{\mathbb E}(A'_{\alpha+1})  & U_{i+k}^{\mathbb E}(A'_{\alpha})\boxtimes_{\Sigma_i}Q^i_{i-1}(f)
\ar[l] \ar[r] & U_{i+k}^{\mathbb E}(A'_{\alpha})\boxtimes_{\Sigma_i}Y^{\boxtimes i}.
}
\]
By induction on $i$ we may assume that the vertical map on the left is
a level equivalence and since the horizontal maps on the right are
$h$-cofibrations, again by Proposition \ref{prop:h-cofibration-properties}(vii),
it is sufficient to show that the two other vertical maps are level equivalences. We know from the pushout-product axiom, 
Proposition \ref{prop:K-pushout-product-axiom},
that the $\cK$-spaces $Q_{i-1}^{i}(f)$ and $Y^{\boxtimes i}$ are $\cA$-relative cofibrant. It therefore follows from the induction hypothesis on $\alpha$ and Proposition~\ref{prop:boxtimes-flat-invariance} that these two vertical maps 
are level equivalences before passing to $\Sigma_i$-orbits. Using Lemma
\ref{lem:orbit-lemma} as in the proof of 
Lemma \ref{structured-acyclic-pushout} we conclude that the induced maps
of $\Sigma_i$-orbits are also level equivalences.
\end{proof}

\section{Properness for \texorpdfstring{$\cK$}{K}-spaces}
Recall that a model category is \emph{left proper} if every pushout of a weak equivalence along a cofibration is a weak equivalence and \emph{right proper} if every pullback of a weak equivalence along a fibration is a weak equivalence. A model category is said to be \emph{proper} if it is both left and right proper. This is a desirable property that a model category may or may not have. In this section we discuss properness for our model categories of 
$\cK$-spaces for a well-structured index category $(\cK,\cA)$. 
The main result is Corollary \ref{cor:proper} which states that the $\cA$-relative $\cK$-model structure on $\cS^{\cK}$ is proper and that the lifted model structure on $\cC\cS^{\cK}$ is proper when $(\cK,\cA)$ is very 
well-structured.

\subsection{Right properness}
It is an immediate consequence of right properness for the category of spaces $\cS$ that the $\cA$-relative level model structure on $\cS^{\cK}$ is right proper. In order to establish right properness for the $\cA$-relative $\cK$-model structure we need the following lemma. 

\begin{lemma}\label{lem:hocolim-pullback}
Let $\cC$ be a small category and let $\cS^{\cC}$ be the associated diagram category. Then the functor $\hocolim_{\cC}\colon \cS^{\cC}\to \cS$ preserves pullbacks. 
\end{lemma}
\begin{proof}
Consider a pullback square of $\cC$-spaces
\[
\xymatrix@-1pc{
A\ar[r] \ar[d] & B\ar[d]\\
C\ar[r] & D.
}
\]
 Since realization of simplicial spaces preserves pullback
  diagrams it is enough to show that the diagram
\[\xymatrix@-1pc{\displaystyle\coprod_{\bld{c_0} \ot \dots \ot \bld{c_r}}A(\bld{c_r}) \ar[r] \ar[d]&
  \displaystyle\coprod_{\bld{c_0} \ot \dots \ot \bld{c_r}}B(\bld{c_r}) \ar[d]\\
  \displaystyle\coprod_{\bld{c_0} \ot \dots \ot \bld{c_r}}C(\bld{c_r}) \ar[r]&
  \displaystyle\coprod_{\bld{c_0} \ot \dots \ot \bld{c_r}}D(\bld{c_r}) }\]
of simplicial replacements is a pullback diagram in each simplicial degree  $r$ and this is clear from the definition. 
\end{proof}

\begin{proposition}
The $\cA$-relative $\cK$-model structure on $\cS^{\cK}$ is right proper.
\end{proposition}
\begin{proof}
Given a pullback square of $\cK$-spaces
\[
\xymatrix@-1pc{
A\ar[r]^{\bar f}\ar[d]_{\bar g} & B\ar[d]^g\\
C\ar[r]^f & D
}
\]
in which $g$ is an $\cA$-relative $\cK$-fibration and $f$ is a 
$\cK$-equivalence, we must show that $\bar f$ is a 
$\cK$-equivalence. Applying the homotopy colimit functor over 
the subcategory $\cK_{\cA}$ to the diagram we get a diagram of spaces
\[
\xymatrix@-1pc{
\hocolim_{\cK_{\cA}}A\ar[r] \ar[d] &
\hocolim_{\cK_{\cA}} B\ar[d]\\
\hocolim_{\cK_{\cA}}C\ar[r] & 
\hocolim_{\cK_{\cA}}D
}
\]
which we claim is homotopy cartesian. Indeed, since we know from Lemma 
\ref{lem:hocolim-pullback} that this is a pullback diagram it suffices to show that the fibers of each of the vertical maps are weakly equivalent to the corresponding homotopy fibers. The latter can be deduced from Lemma \ref{lem:puppe-lemma} and the definition of an $\cA$-relative 
$\cK$-fibration. The fact that the diagram is homotopy cartesian in turn implies that the vertical maps induce weak equivalences of the horizontal homotopy fibers which, by the homotopy cofinality condition (iv) in the definition of a well-structured relative index category, implies the statement of the proposition.
\end{proof}
Recall that homotopy cartesian squares can be conveniently treated in
right proper model categories (see e.g.\ ~\cite[II.\S
8]{Goerss_J-simplicial} for details). The proof of the previous
proposition can now be reinterpreted to give the following corollary.

\begin{corollary}
  A square diagram in $\cS^{\cK}$ is homotopy cartesian with respect to the
  $\cA$-relative $\cK$-model structure if and only if the associated
  square of homotopy colimits is homotopy cartesian in $\cS$. \qed
\end{corollary}

Consider now an operad $\cD$ and the category of algebras $\cS^{\cK}[\mathbb D]$ for the corresponding monad $\mathbb D$. Since pullbacks in 
$\cS^{\cK}[\mathbb D]$ are defined in terms of the underlying $\cK$-spaces, right properness of a model structure on $\cS^{\cK}$ carries over to the lifted model structure on $\cS^{\cK}[\mathbb D]$ when the latter is defined.

\begin{corollary}
The $\cA$-relative $\cK$-model structures on $ \cS^{\cK}[\mathbb D]$ considered in Proposition \ref{prop:structured-lift-proposition} are right proper.\qed
\end{corollary}

\subsection{Left properness}
It is an immediate consequence of the left properness of $\cS$ that the $\cA$-relative level model structure on $\cS^{\cK}$ is left proper. By parts (i) and (iv) of Proposition \ref{prop:h-cofibration-properties} the same then holds for the 
$\cA$-relative $\cK$-model structure. 

\begin{proposition}
The $\cA$-relative $\cK$-model structure on $\cS^{\cK}$ is left proper.\qed
\end{proposition}

Next we consider left properness for the category $\cC\cS^{\cK}$ of commutative $\cK$-space monoids.  
The proof in this case is based on the following lemma where
we use the notation $B\boxtimes _AC$ for the pushout of a diagram of
commutative $\cK$-space monoids $B\leftarrow A\to C$.

\begin{lemma}\label{lem:AB-lemma}
Let $(\cK,\cA)$ be a very well-structured relative index category and let 
$A\to B$ be a cofibration in the $\cA$-relative $\cK$-model structure on $\cC\cS^{\cK}$.
Then the functor $B\boxtimes_A(-)$ preserves
 $\cK$-equivalences of commutative $\cK$-space monoids under $A$.
\end{lemma}
\begin{proof}
  Let $C\to C'$ be a $\cK$-equivalence of commutative $\mathcal
  K$-space monoids under $A$. Suppose first that $A\to B$ has the form
  $\mathbb C(X)\to \mathbb C(Y)$ for a generating $\cA$-relative
 cofibration $f\colon X\to Y$ in $\mathcal S^\cK$, where as usual 
  $\mathbb C$ denotes the monad associated to the commutativity operad. 
   We write
  $D$ and $D'$ for the pushouts $\mathbb C(Y)\boxtimes_{\mathbb
    C(X)}C$ and $\mathbb C(Y)\boxtimes_{\mathbb C(X)}C'$, and consider
  the associated filtration terms $F_iD = F_iU_0^{\mathbb C}(D)$ and
  $F_iD'= F_iU_0^{\mathbb C}(D')$ from Proposition
  \ref{prop:pushout-filtration}. Since these are filtrations by
  $h$-cofibrations of $\cK$-spaces, it suffices to prove that
  $F_iD\to F_iD'$ is a $\cK$-equivalence for all $i$. It
  follows from Proposition~\ref{prop:pushout-filtration} that $F_{i+1}D\to
  F_{i+1}D'$ can be identified with the map obtained from the diagram
\[
\xymatrix@-1pc{
  F_iD \ar[d]& C\boxtimes Q_{i-1}^i(f)/\Sigma_i \ar[l] \ar[r] \ar[d]& C\boxtimes Y^{\boxtimes i}/\Sigma_i \ar[d]\\
  F_iD' & C'\boxtimes Q_{i-1}^i(f)/\Sigma_i \ar[l] \ar[r] &
  C'\boxtimes Y^{\boxtimes i}/\Sigma_i }
\]   
by evaluating the pushouts horizontally. Here we use that by Example
\ref{example:U_k-commutative} the functor $U_i^{\mathbb C}$
takes a commutative $\cK$-space monoid to its underlying
$\cK$-space with trivial $\Sigma_i$-action. Proceeding by
induction, we assume that $F_iD\to F_iD'$ is a 
$\cK$-equivalence and it remains to show that so are the other two
vertical maps. We know from Proposition
\ref{prop:boxtimes-flat-invariance} that these maps are
$\mathcal K$-equivalences before taking $\Sigma_i$-orbits and arguing as in the second half of the proof of Lemma \ref{structured-acyclic-pushout} the
same then holds for the maps of $\Sigma_i$-orbits.

In the general case we may assume that $A\to B$ is a relative
cell complex in the sense that there exists a $\lambda$-sequence
$\{B_{\alpha}\colon \alpha<\lambda\}$ of commutative $\mathcal
K$-space monoids (for some ordinal~$\lambda$) such that $A=B_0$, $B
\iso \colim_{\alpha < \lambda}B_{\alpha}$, and $B_{\alpha}\to
B_{\alpha+1}$ is obtained by cobase change from a map of the form
$\mathbb C(X_{\alpha})\to \mathbb C(Y_{\alpha})$ where $X_{\alpha}\to
Y_{\alpha}$ is a generating $\cA$-relative cofibration in 
$\mathcal S^{\cK}$. Lemma~\ref{structured-cocomplete} and 
Proposition~\ref{prop:D-h-cofibration} imply that $B_{\alpha}\boxtimes_A C$ and
$B_{\alpha}\boxtimes_A C'$ are $\lambda$-sequences of $h$-cofibrations
of the underlying $\cK$-spaces. It follows from an inductive
argument based on the special case considered in the beginning of the
proof that $B_{\alpha}\boxtimes _AC\to B_{\alpha}\boxtimes_A C'$ is a
$\cK$-equivalence for each $\alpha$. This implies the statement
of the lemma.
\end{proof}

\begin{proposition}\label{prop:left-properness}
 Let $(\cK,\cA)$ be a very well-structured relative index category. Then the 
 $\cA$-relative  $\cK$-model structure on
  $\mathcal C\mathcal S^{\cK}$ is left proper.
\end{proposition}
\begin{proof}
  Let $A\to B$ be an $\cA$-relative cofibration and let $A\to C$ be a 
  $\cK$-equivalence, both in $\mathcal C\mathcal S^{\cK}$. The
  cobase change of $A \to C$ along $A \to B$ can be identified with
  the map $B\boxtimes_A A\to B\boxtimes _A C$ and the result therefore
  follows from Lemma \ref{lem:AB-lemma}.
\end{proof}

\begin{corollary}\label{cor:proper}
\begin{enumerate}[(i)]
\item
The $\cA$-relative $\cK$-model structure on $\cS^{\cK}$ is proper. 
\item
If $(\cK,\cA)$ is very well-structured, then the $\cA$-relative $\cK$-model structure on $\cC\cS^{\cK}$ is proper. 
 \qed
\end{enumerate}
\end{corollary}

\begin{remark}
 In general we do not know under which conditions on the operad the model structures on  $\cS^{\cK}[\mathbb D]$ are left proper. A proof based on a generalization of Lemma \ref{lem:AB-lemma} would require a more careful analysis of the functors $U_k^{\mathbb D}$.
\end{remark}

\section{Cofibrancy of structured diagram spaces}
\label{sec:structured-cofibrant}
Let again $\mathcal D$ be an operad in $\cS$ and let $\mathbb D$ be the associated monad on $\mathcal S^{\cK}$. In this section we analyze to what
extent the forgetful functor from $\mathcal S^{\cK}[\mathbb D]$
to $\mathcal S^{\cK}$ preserves $\cA$-relative cofibrancy for a (very) 
well-structured relative index category $(\cK,\cA)$. Suppose first that $\cD$ is $\Sigma$-free. 
As a motivating example consider for an object $\bld k$ in $\cA$
the cofibrant object $\mathbb D(F^{\cK}_{\mathbf k}(*))$
in the $\cA$-relative $\cK$-model structure on $\mathcal
S^{\cK}[\mathbb D]$. This has as its underlying $\mathcal
K$-space the coproduct of the $\cK$-spaces $\mathcal
D(n)\times_{\Sigma_n}F^{\cK}_{\mathbf k^{\sqcup n}}(*)$ for $n\geq 0$. In particular, for $n=0$ this is the free $\cK$-space $F_{\bld 0}(*)$. In order for the latter to be cofibrant we introduce the following assumption on our index categories; compare also to Proposition \ref{prop:well-structured-comparison}.
Here and in the following the degree functor on $\cK$ is supposed to be fixed.

\begin{coarse-assumptions}\label{assump:coarse}
Let $(\cK,\cA)$ be a well-structured relative index category
 and let $(\cK,\cB)$ be a well-structured relative index category such that 
 $\cA\subseteq \cB$ and $\cB$ contains the unit $\bld 0$ for the monoidal structure.
\end{coarse-assumptions}

Recall from Section \ref{subs:G-space-model} the notion of the \emph{coarse} model structure associated to the category of $G$-spaces for a discrete group $G$. It follows from the definition of the generating cofibrations that the $G$-action on a cofibrant object is free. In the topological setting a cofibrant object is also Hausdorff (and not just weak Hausdorff).

\begin{proposition}\label{prop:coarse-underlying-cofibrant}
  Let $\mathcal D$ be an operad in $\mathcal S$ such that $\mathcal
 D(n)$ is cofibrant in the coarse model structure on $\cS^{\Sigma_n}$ for all
$n$, and let $\cK$, $\cA$, and $\cB$ be as in Coarse 
Assumptions~\ref{assump:coarse}.
  Suppose that $A$ is a $\mathbb D$-algebra which is cofibrant in the 
  $\cA$-relative $\cK$-model structure on $\mathcal S^{\cK}[\mathbb
  D]$. Then the underlying $\cK$-space of $A$ is cofibrant in
  the $\cB$-relative $\cK$-model structure on $\mathcal S^{\mathcal  K}$.
\end{proposition}
We prove the Proposition in Section \ref{subs:coarse-underlying} below.
Notice, that in the simplicial setting this proposition applies to all
$\Sigma$-free operads. In the topological setting it applies for
instance to operads $\mathcal D$ such that $\mathcal D(n)$ is a free
$\Sigma_n$-equivariant CW complex. 

\begin{corollary}\label{cor:projective-underlying-cofibrant}
Let $\cK$ be a well-structured index category and let $\cD$ be an operad in 
$\cS$  such that $\mathcal D(n)$ is cofibrant in the coarse model structure on $\cS^{\Sigma_n}$ for all $n$. Suppose that $A$ is a $\mathbb D$-algebra which is cofibrant in the projective $\cK$-model structure on 
$\cS^{\cK}[\mathbb D]$. Then the underlying $\cK$-space of $A$ is cofibrant in the projective $\cK$-model structure on $\cS^{\cK}$. \qed
\end{corollary}
For example, this applies to the associativity operad and the corresponding category of (not necessarily commutative) $\cK$-space monoids.

Now we drop the assumption that $\cD$ be $\Sigma$-free and assume instead that $(\cK,\cA)$ is very well-structured. By Proposition 
\ref{prop:structured-lift-proposition} this ensures that the $\cA$-relative 
$\cK$-model structure on $\cS^{\cK}$ lifts to 
$\cS^{\cK}[\mathbb D]$ for any operad $\mathcal D$. However, for the forgetful functor to preserve cofibrancy we need the additional assumption stated below. Recall the canonical homomorphism 
$\Sigma_n\ltimes \cA(\bld k)^{\times n}\to\cK(\bld k^{\sqcup n})$ discussed after Definition~\ref{def:very-well-structured}.

\begin{fine-assumptions}\label{assump:fine}
Let $(\cK,\cA)$ be a very well-structured relative index category
 and let $(\cK,\cB)$ be a well-structured relative index category such that 
 $\cA\subseteq \cB$, the group  $\Sigma_n\ltimes \cA(\bld k)^{\times n}$ maps into $\cB(\bld k^{\sqcup n})$ for all objects $\bld k$ in $\cA$ and all $n\geq 1$, and $\cB$ contains the unit $\bld 0$ for the monoidal structure
\end{fine-assumptions}

Recall from Section \ref{subs:G-space-model} the notion of the \emph{fine} model structure associated to the category of $G$-spaces for a group $G$.

\begin{proposition}\label{prop:fine-underlying-cofibrant}
  Let $\mathcal D$ be an operad such that $\mathcal D(n)$ is
  cofibrant in the fine model structure on $\cS^{\Sigma_n}$ for all $n$, and let 
  $\cK$, $\cA$, and $\cB$ be as in Fine Assumptions \ref{assump:fine}. Suppose that $A$ is a 
$\mathbb D$-algebra which is cofibrant in the $\cA$-relative $\cK$-model structure on $\mathcal S^{\cK}[\mathbb D]$. Then the underlying $\cK$-space of $A$ is cofibrant in the $\cB$-relative $\cK$-model structure on $\cS^{\cK}$.
\end{proposition}
The proof will be given in Section \ref{subs:fine-underlying}.
Notice, that in the simplicial setting this proposition applies to all
operads. In the topological setting it applies for instance to operads
$\cD$ for which $\cD(n)$ is a $\Sigma_n$-equivariant CW
complex. 

\begin{corollary}\label{cor:comm-underlying-cofibrant}
Let $\cK$, $\cA$, and $\cB$ be as in Fine Assumptions 
\ref{assump:fine}. Suppose that $A$ is a commutative $\cK$-space monoid which is cofibrant in the $\cA$-relative $\cK$-model structure on
  $\mathcal C\mathcal S^{\cK}$. Then the underlying $\mathcal
  K$-space of $A$ is cofibrant in the $\cB$-relative $\cK$-model
  structure on $\mathcal S^{\cK}$.\qed
\end{corollary}

The last result about preservation of cofibrancy has no assumptions on
the operads. Instead it uses the notion of an $h$-cofibration from Section
\ref{subs:h-cofibrations}.

\begin{proposition}\label{prop:D-h-cofibration}
  Let $\mathcal D$ be an operad and let $A\to B$ be a cofibration in
  any of the lifted model structures on $\cS^{\cK}[\mathbb
  D]$ considered in Proposition \ref{prop:structured-lift-proposition}. Then the underlying map of $\cK$-spaces is an $h$-cofibration.
\end{proposition}
\begin{proof}
  This follows from Lemma \ref{structured-h-cofibration-pushout},
  Lemma \ref{structured-cocomplete}, and Proposition
  \ref{prop:h-cofibration-properties}, by expressing $A\to B$ as a
  retract of a relative cell complex (that is, a transfinite composition of maps  
  obtained by attaching generating cofibrations).
\end{proof}

\subsection{The proof of Proposition  \ref{prop:coarse-underlying-cofibrant}}
\label{subs:coarse-underlying}
In order to prove Proposition \ref{prop:coarse-underlying-cofibrant} we shall
temporarily work in the category of $\Sigma_k$-equivariant $\mathcal K$-spaces for varying $k$ and for this we need an equivariant
version of the $\cA$-relative level model structure on 
$\mathcal S^{\mathcal K}$.  In general, given a finite group $G$, we write
$(\mathcal S^{\cK})^G$ for the category of $G$-$\cK$-spaces, that is,
$\cK$-spaces with $G$-action. If we view $G$ as a category with a
single object this is the same thing as functors from 
$ G\times\cK$ to $\mathcal S$ which in turn is the same as $\cK$-diagrams in $\cS^G$.

Proceeding as in the definition of the \emph{coarse}  model structure on 
$\cS^G$ considered in Section \ref{subs:G-space-model} we now define the \emph{$G$-coarse $\cA$-relative level model structure} on $(\cS^{\cK})^G$. 
In this model structure a map of 
$G$-$\cK$-spaces is a weak equivalence (respectively a fibration) if and only if the underlying map of $\cK$-spaces is an $\cA$-relative level equivalence (respectively an $\cA$-relative level fibration) as defined in Section  
\ref{subs:level-model-structures}. Arguing as for the $\cA$-relative level model structure on $\mathcal S^{\cK}$, one checks that this defines a cofibrantly generated model structure on $(\mathcal S^{\cK})^G$. The generating
cofibrations (respectively the generating acyclic cofibrations) are the maps
of the form $G\times X\to G\times Y$ where $X\to Y$ is a generating
cofibration (respectively a generating acyclic cofibration) for the $\cA$-relative level model structure on $\mathcal S^{\cK}$. We shall use the term 
\emph{$G$-coarse $\cA$-relative cofibration} for a cofibration in this model structure. 

The properties of the cofibrations stated in the next lemma will be needed later. They are easy consequences of the fact that any such cofibration is a retract of a relative cell complex constructed from the generating cofibrations. 

\begin{lemma}\label{lem:coarse-cofibration-properties}
  Let $H$ and $G$ be finite groups.
\begin{enumerate}[(i)]
\item If $X\to Y$ is an $(H\times G)$-coarse $\cA$-relative cofibration, then the map of $G$-orbits $X/G\to Y/G$ is an $H$-coarse $\cA$-relative cofibration.
\item If $X\to Y$ is a $G$-coarse $\cA$-relative cofibration and $H$ is a subgroup of $G$, then, restricting the action, $X\to Y$ is an 
 $H$-coarse $\cA$-relative cofibration.  \qed
\end{enumerate}
\end{lemma}

The proofs of the next two lemmas are based on the following elementary observation: Let $f\colon X\to Y$ be a map of $G$-$\cK$-spaces and write $f'\colon X'\to Y'$ for the underlying map of $\cK$-spaces with trivial $G$-action. Then there is a commutative diagram of $G$-$\cK$-spaces
\begin{equation}\label{eq:G-observation}
\xymatrix{
G\times X' \ar[r]^{1_G\times f'} \ar[d]^{\cong} & G\times Y'\ar[d]^{\cong}\\
G\times X \ar[r]^{1_G\times f} & G\times Y
}
\end{equation}
where the vertical maps are the $G$-equivariant isomorphisms 
$(g,x)\mapsto (g,gx)$.

\begin{lemma}\label{lem:coarse-equivariant-lemma}
Let $G$ be a finite group and suppose that $f$ is a $G$-coarse $\cA$-relative cofibration in $(\cS^{\cK})^G$ and that $i$ is a cofibration in the fine model structure on $\cS^G$. Then the pushout product $f\Box i$ is again a 
$G$-coarse $\cA$-relative cofibration in $(\mathcal S^{\cK})^G$.
\end{lemma}
\begin{proof}
  By~\cite[Lemma 4.2.4]{Hovey-model} it suffices to check this
  when $f$ and $i$ are generating cofibrations for the respective
  model structures. Thus we may assume that $f$ has the form $1_G\times g$ for a generating $\cA$-relative cofibration $g$ in $\cS^{\cK}$. Then $f\Box i$ can be identified with $1_G\times (g\Box i)$. It follows from Proposition \ref{prop:level-S-model} that, forgetting the equivariant structure, $g\Box i$ is an $\cA$-relative cofibration. Now apply the observation above the lemma to get the result.  
\end{proof}     

As a consequence of the lemma we see that if $U$ is cofibrant in
the $G$-coarse $\cA$-relative level model structure on 
$(\mathcal S^{\cK})^G$ and $K\to L$ is a cofibration in the fine model structure on $\cS^G$, then the induced map $U\times K\to U\times L$ is a 
$G$-coarse $\cA$-relative cofibration (since $\emptyset \to U$ is a $G$-coarse $\cA$-relative cofibration).

Just as in the non-equivariant setting, the category $(\mathcal
S^{\cK})^G$ is closed symmetric monoidal with monoidal product
$X\boxtimes Y$ defined as the usual Kan extension along the monoidal
structure map $\cK\times \cK\to \cK$. Thus, the underlying $\cK$-space
of $X\boxtimes Y$ is the $\boxtimes$-product of the underlying
$\cK$-spaces of $X$ and $Y$.

\begin{lemma}\label{lem:UX-coarse-cofibrant}
  Let $G$ be a finite group. If $U$ is a $G$-$\cK$-space which
  is cofibrant in the $G$-coarse $\cA$-relative level model structure on 
  $(\mathcal S^{\cK})^G$ and $X$ is a $G$-$\cK$-space whose
  underlying $\cK$-space is $\cA$-relative cofibrant, then $U\boxtimes X$ is
  again cofibrant in the $G$-coarse $\cA$-relative level model structure on
  $(\mathcal S^{\cK})^G$.
\end{lemma} 
\begin{proof}
  We may assume without loss of generality that $U$ is a $G$-coarse cell complex in the sense that it can be identified
  with the colimit of a $\lambda$-sequence of $G$-$\cK$-spaces
  $\{U_{\alpha}\colon\alpha<\lambda\}$ (for some ordinal~$\lambda$)
  such that $U_0=\emptyset$ and $U_{\alpha}\to U_{\alpha+1}$ is
  obtained by cobase change of a generating cofibration $G\times
  X_{\alpha}\to G\times Y_{\alpha}$ (where $X_{\alpha}\to Y_{\alpha}$
  is a generating $\cA$-relative cofibration in $\mathcal S^{\mathcal
    K}$). It follows that $U\boxtimes X$ can be identified with the
  colimit of the $\lambda$-sequence $\{U_{\alpha}\boxtimes
  X\colon\alpha<\lambda\}$ and it suffices to show that each of the
  maps $U_{\alpha}\boxtimes X\to U_{\alpha+1}\boxtimes X$ is a 
  $G$-coarse $\cA$-relative cofibration. This map is obtained by cobase change from the map $G\times X_{\alpha}\boxtimes X\to G\times Y_{\alpha}\boxtimes X$ and the conclusion now follows from the observation before Lemma \ref{lem:coarse-equivariant-lemma} (letting $X_{\alpha}\boxtimes X\to Y_{\alpha}\boxtimes X$ be the map $f$ in \eqref{eq:G-observation}).
\end{proof}

We now turn to the proof of Proposition \ref{prop:coarse-underlying-cofibrant}
and assume that $\cA$ and $\cB$ are as in Coarse Assumptions \ref{assump:coarse}. Then the above lemmas apply as well to the well-structured relative index category $(\cK,\cB)$.
The key step is again to analyze pushout diagrams of the form 
\eqref{verification-structured-pushout}. 

\begin{lemma}\label{lem:structured-coarse-pushout-lemma}
  Let $\mathcal D$ be an operad in $\mathcal S$ such that $\mathcal
  D(n)$ is cofibrant in the coarse model structure on $\cS^{\Sigma_n}$ for all $n$. Suppose that the map $f$ in \eqref{verification-structured-pushout} is a generating $\cA$-relative cofibration and that $U_k^{\mathbb D}(A)$ is cofibrant in the $\Sigma_k$-coarse $\cB$-relative level model structure on 
  $(\mathcal S^{\mathcal  K})^{\Sigma_k}$ for all $k\geq 0$. Then the induced map
  \[
  U_k^{\mathbb D}(\bar f)\colon U_k^{\mathbb D}(A)\to U^{\mathbb
    D}_k(B)
  \]
  is a $\Sigma_k$-coarse $\cB$-relative cofibration for all $k\geq 0$.
\end{lemma}
\begin{proof}
  Applying the filtration from Proposition \ref{prop:pushout-filtration} it
  suffices to show that
\begin{equation}\label{eq:coarse-map}
U_{i+k}^{\mathbb D}(A)\boxtimes_{\Sigma_i}Q_{i-1}^i(f) \to
U_{i+k}^{\mathbb D}(A)\boxtimes_{\Sigma_i}Y^{\boxtimes i}
\end{equation}
is a $\Sigma_k$-coarse $\cB$-relative cofibration for all $i$ and $k$. By definition, the generating cofibration $f$ has the form 
$G_{\bld k}^{\cK}(\cK(\bld k)/H\times h)$ for an object $\bld k$ in $\cA$, a subgroup $H\subseteq \cA(\bld k)$, and a generating cofibration $h$ in $\cS$. Before passing to $\Sigma_i$-orbits, the map \eqref{eq:coarse-map} can therefore be identified with the map
\[ 
U_{i+k}^{\mathbb D}(A)\boxtimes G_{\bld k^{\sqcup i}}^{\cK}
(\cK(\bld k^{\sqcup i})/H^{\times i})\times h^{\Box i}
\]
where $h^{\Box i}$ is the iterated pushout-product in $\cS$. We can view this as a map of $(\Sigma_i\times \Sigma_k)$-$\cK$-spaces by restricting the 
$\Sigma_{i+k}$-action on $U_{i+k}^{\mathbb D}(A)$ and extending the obvious 
$\Sigma_i$-actions on the two other factors to $(\Sigma_i\times \Sigma_k)$-actions by letting $\Sigma_k$ act trivially. Using Lemma \ref{lem:coarse-cofibration-properties}(ii) (with $\cB$ instead of $\cA$) we see that the assumptions on  
$U_{i+k}^{\mathbb D}(A)$ imply that the latter is cofibrant in the 
$(\Sigma_i\times \Sigma_k)$-coarse $\cB$-relative model structure. Hence the $\boxtimes$-product of the first two factors is also cofibrant in this model structure by Lemma  \ref{lem:UX-coarse-cofibrant}. The argument given in the proof of Proposition \ref{prop:h-cofibration-properties}(vii) implies that 
$h^{\Box i}$ is a cofibration in the fine model structure on $\cS^{\Sigma_i}$. 
Combined with Lemma  \ref{lem:coarse-equivariant-lemma}
this in turn implies that the above map is a 
$(\Sigma_i\times\Sigma_k)$-coarse $\cB$-relative cofibration. Finally, 
Lemma \ref{lem:coarse-cofibration-properties}(i) then gives that the induced map of $\Sigma_i$-orbits is a $\Sigma_k$-coarse $\cB$-relative cofibration.
\end{proof}

\begin{proof}[Proof of Proposition \ref{prop:coarse-underlying-cofibrant}]
  We may assume without loss of generality that $A$ is a cell complex
  in the $\cA$-relative $\cK$-model structure on $\mathcal
  S^{\cK}[\mathbb D]$. Thus, $A$ may be identified with the colimit of
  a $\lambda$-sequence of $\mathbb D$-algebras
  $\{A_{\alpha}\colon\alpha<\lambda\}$ (for some ordinal~$\lambda$)
  such that $A_0=\bld 1_{\cK}$ and $A_{\alpha}\to A_{\alpha+1}$ is
  obtained by cobase change of a generating cofibration $\mathbb
  D(X_{\alpha})\to \mathbb D(Y_{\alpha})$ where $X_{\alpha}\to
  Y_{\alpha}$ is a generating $\cA$-relative cofibration in $\mathcal
  S^{\cK}$. Since the underlying $\cK$-space of the unit $\bld
  1_{\cK}$ can be identified with $F^{\cK}_{\mathbf 0}(*)$, it
  suffices by Lemma \ref{structured-cocomplete} to show that each of
  the morphisms $A_{\alpha}\to A_{\alpha+1}$ defines a $\cB$-relative
  cofibration in $\mathcal S^{\cK}$. In order to set up an inductive
  argument based on Lemma \ref{lem:structured-coarse-pushout-lemma} we
  in fact prove the stronger statement that (i) $U_k^{\mathbb
    D}(A_{\alpha})$ is cofibrant in the $\Sigma_k$-coarse
  $\cB$-relative level model structure on $(\mathcal
  S^{\cK})^{\Sigma_k}$ for all $k$, and (ii) $U_k^{\mathbb
    D}(A_{\alpha})\to U_k^{\mathbb D}(A_{\alpha+1})$ is a
  $\Sigma_k$-coarse $\cB$-relative cofibration for all $k$. Letting
  $k=0$ then gives the result. In order to start the induction we
  observe that by Proposition~\ref{prop:pushout-filtration}(iii)
  (identifying $\bld 1_{\cK}$ with $\mathbb D(\emptyset)$), the
  underlying $\mathcal K$-space of $U_k^{\mathbb D}(\bld 1_{\cK})$ is
  isomorphic to $F^{\cK}_{\mathbf 0}(\mathcal D(k))$ and therefore
  cofibrant in the $\Sigma_k$-coarse $\cB$-relative model structure on
  $(\mathcal S^{\cK})^{\Sigma_k}$ by the assumption on $\mathcal D$.
  Proceeding by induction we consider an ordinal $\beta$ with
  $\beta+1<\lambda$ such that (i) and (ii) hold for all
  $\alpha<\beta$. If $\beta$ is a successor ordinal, $\beta=\alpha+1$,
  it is immediately clear from Lemma
  \ref{lem:structured-coarse-pushout-lemma} that (i) and (ii) also
  hold for $\beta$. If $\beta$ is a limit ordinal, then it follows
  from the definition of a $\lambda$-sequence and the fact that the
  functor $U^{\mathbb D}_k$ preserves filtered colimits that
  $U_k^{\mathbb D}(A_{\beta})$ can be identified with
  $\colim_{\alpha<\beta} U_k^{\mathbb D}(A_{\alpha})$. By the
  induction hypothesis $\bld 1_{\cK}\to U_k^{\mathbb D}(A_{\beta})$ is
  therefore a transfinite composition of $\Sigma_k$-coarse
  $\cB$-relative cofibrations, hence itself a $\Sigma_k$-coarse
  $\cB$-relative cofibration which implies that (i) holds for
  $\beta$. By Lemma~\ref{lem:structured-coarse-pushout-lemma}, (ii)
  then also holds for $\beta$.
 
\end{proof}

\subsection{The proof of Proposition \ref{prop:fine-underlying-cofibrant} }\label{subs:fine-underlying}
In this section $(\cK,\cA)$ denotes a very well-structured relative index category. We begin by defining, for a finite group $G$, a cofibrantly generated \emph{$(G\times \cA)$-relative level model structure} on $(\mathcal S^{\cK})^G$: The weak equivalences in this model structure are the 
$\cA$-relative level equivalences of the underlying $\cK$-spaces and the fibrations are defined as in Section \ref{subs:level-model} with $H$ now being a subgroup of $G \times\cA(\bld{k})$. In order to see that this defines a cofibrantly generated model structure we notice that for each object $\bld k$ in $\cA$ the functor 
\[
G_{\bld k}^{\cK}\colon \cS^{G\times \cK(\bld k)}\to \cS^{G\times \cK},\quad
G_{\bld k}^{\cK}(L)=\cK(\bld k,-)\times_{\cK(\bld k)}L
\]
is left adjoint to the evaluation functor $\Ev_{\bld k}$. We obtain the generating (acyclic) cofibrations by applying the functors 
$G_{\bld k}^{\cK}$ to the generating (acyclic) cofibrations for the 
$(G\times \cA(\bld k))$-relative mixed model structure on $\mathcal
S^{G\times \cK(\bld{k})}$, cf.\ Section
\ref{subs:G-space-model}. We shall use the term \emph{$(G\times\cA)$-relative cofibration} for a cofibration in this model structure. 

As in the non-equivariant case there is a more explicit description of the 
$(G\times \cA)$-relative cofibrations. Given a $G$-$\cK$-space $X$ and an object $\bld k$ in $\cK$, the latching space $L_{\bld k}(X)$ of the underlying $\cK$-space is defined as in Section \ref{subs:level-model}. Recall the notion of a relative equivariant cofibration introduced in Section 
\ref{subs:G-space-model}. The arguments in the proof of Proposition 
\ref{prop:latching-characterization} easily generalizes to give the next result.

\begin{proposition}\label{prop:G-latching}
A map of $G$-$\cK$-spaces $X\to Y$ is a $(G\times\cA)$-relative cofibration if and only if the $(G\times \cK(\bld k))$-equivariant latching map
\[
 L_{\bld{k}}(Y) \cup_{L_{\bld{k}}(X)}X(\bld{k}) \to Y(\bld{k})
\]
is a $(G\times\cA(\bld k))$-relative cofibration for all objects $\bld k$ in $\cA$ and an isomorphism for all objects $\bld k$ not in $\cA$.\qed
\end{proposition}

Using this characterization it is easy to see that the 
$(G\times \cA)$-relative analogues of the statements in Lemma 
\ref{lem:coarse-cofibration-properties} hold. In the following we deduce some further properties of the $(G\times \cA)$-relative cofibrations needed for the proof of 
Proposition~\ref{prop:fine-underlying-cofibrant}. For this it is useful to observe that in general, given a discrete group $H$ and a normal subgroup $N$, a discrete $H$-space is cofibrant in the $N$-relative model structure on $\cS^H$ if and only if each of the isotropy subgroups of $H$ is contained in $N$.   

\begin{proposition}\label{G-pushout-product-axiom}
The pushout-product axiom holds for the $(G\times \cA)$-relative level model structure on $(\mathcal S^{\cK})^G$.
\end{proposition}
\begin{proof}
 By \cite[Lemma 3.5]{Schwede_S-algebras} it suffices to consider the generating (acyclic) cofibrations. Let
 $\bld k_1$ and $\bld k_2$ be objects in $\cA$, and let $h_s$ be a generating cofibration for the ($G\times \cA(\bld k_s)$)-relative model structure on $S^{G\times \cK(\bld k_s)}$ for $s=1,2$. Then we have the identification
\[
G_{\bld k_1}^{\cK}(h_1)\Box G_{\bld k_2}^{\cK}(h_2)
\cong G^{\cK}_{\bld k_1\sqcup \bld k_2}(\cK(\bld k_1\sqcup \bld k_2)
\times _{\cK(\bld k_1)\times \cK(\bld k_2)}h_1\Box h_2)
\]
and we must check that in the last expression we apply 
$G_{\bld k_1\sqcup \bld k_2}^{\cK}$ to a ($G\times\cA(\bld k_1\sqcup \bld k_2)$)-relative cofibration. Writing $h_s=(G\times \cK(\bld k_s))/H\times i_s$ for
 $H_s\subseteq G\times \cA(\bld k_s)$ and $i_s$ a generating cofibration in 
 $\cS$,  we get
\[
\cK(\bld k_1\sqcup \bld k_2)
\times _{\cK(\bld k_1)\times \cK(\bld k_2)}h_1\Box h_2\cong
(G\times G\times K(\bld k_1\sqcup \bld k_2))/(H_1\times H_2)
\times i_1\Box i_2.
\]
Now apply the ($G\times \cK(\bld k_1\sqcup \bld
k_2)$)-equivariant projection
\[
(G\times G\times K(\bld k_1\sqcup \bld k_2))/(H_1\times H_2)\to 
 K(\bld k_1\sqcup \bld k_2)/\cA(\bld k_1\sqcup \bld k_2) 
\]
to conclude that the isotropy groups of the elements in the domain are contained in 
$G\times \cA(\bld k_1\sqcup\bld k_2)$. This gives the result for the generating cofibrations.

For the second part of the pushout-product axiom (concerning the acyclic cofibrations) we first observe that a ($G\times\cA)$-relative cofibration defines an $\cA$-relative cofibration of the underlying $\cK$-spaces. This can for instance be deduced from the ($G\times \cA$)-relative version of Lemma 
\ref{lem:coarse-cofibration-properties}(ii). 
Since the weak equivalences are defined in terms of the underlying 
$\cK$-spaces the conclusion now follows from the pushout-product axiom for the $\cA$-relative level model structure, Proposition  
\ref{prop:K-pushout-product-axiom}.
\end{proof}

\begin{lemma}\label{lem:iterated-pushout-product-lemma}
Let $(\cK,\cA)$ and $(\cK,\cB)$ be as in Fine Assumptions \ref{assump:fine}, and suppose that $f\colon X\to Y$ is a generating  $\cA$-relative 
cofibration in $\cS^{\cK}$. Then the $i$-fold iterated pushout-product
$
f^{\Box i}\colon Q_{i-1}^i(f)\to Y^{\boxtimes i} 
$
is a $(\Sigma_i\times \cB)$-relative cofibration
\end{lemma}
\begin{proof}
As usual, $f=G_{\bld k}^{\cK}(\cK(\bld k)/H\times h)$ for an object $\bld k$ in 
$\cA$,  a subgroup $H\subseteq \cA(\bld k)$, and a generating cofibration $h$ for $\cS$. Then $f^{\Box i}$ can be identified with the map 
\[
G_{\bld k^{\sqcup i}}^{\cK}(\cK(\bld k^{\sqcup i})/H^{\times i}\times h^{\Box i})=
\cK(\bld k^{\sqcup i},-)/H^{\times i}\times h^{\Box i}
\]
where $h^{\Box i}$ is the iterated pushout-product in $\cS$. Thus, it suffices to show that the $\Sigma_i$-$\cK$-space $Z=\cK(\bld k^{\sqcup i},-)/H^{\times i}$ is $(\Sigma_i\times \cB)$-relative cofibrant (where $\Sigma_i$ acts via the canonical right $\Sigma_i$-action on $\bld k^{\sqcup i}$ specified by the symmetric monoidal structure). By 
Proposition~\ref{prop:G-latching} this is equivalent to the latching map $L_{\bld l}(Z)\to Z(\bld l)$ being a $(\Sigma_i\times \cB(\bld l))$-relative cofibration for all objects $\bld l$ in $\cB$ and an isomorphism for all $\bld l$ not in $\cB$. It is clear from the definition that the latching map is an isomorphism if $\bld l$ and $\bld k^{\sqcup i}$ have different degrees, $\lambda(\bld l)\neq \lambda(\bld k^{\sqcup i})$. If $\lambda(\bld l)=\lambda(\bld k^{\sqcup i})$, then $L_{\bld l}(Z)=\emptyset$ and we claim that $Z(\bld l)$ is $(\Sigma_i\times \cB(\bld l))$-relative cofibrant. This is clear if $Z(\bld l)$ is empty and otherwise there is an isomorphism $\bld l\to \bld k^{\sqcup i}$ in $\cK$ so that it suffices to consider the case $\bld l=\bld k^{\sqcup i}$ (this uses the normality assumption on $\cB$). Hence it only remains to show that the elements of $Z(\bld k^{\sqcup i})$ have isotropy groups contained in $\Sigma_i\times \mathcal B(\bld k^{\sqcup i})$. Let $z$ be an element in $Z(\bld k^{\sqcup i})$ represented by a morphism  $\beta$ in $\cK(\bld k^{\sqcup i})$. The condition for an element $(\sigma,\alpha)$ in 
$\Sigma_i\times \cK(\bld k^{\sqcup i})$ to belong to the isotropy group of $z$ is that there are elements $f_1\,\dots, f_i$ in $H$ such that 
$\alpha\circ\beta=\beta \circ\sigma_*\circ(f_1\sqcup\dots\sqcup f_i)$. Since 
$\cB(\bld k^{\sqcup i})$ is a normal subgroup of $\cK(\bld k^{\sqcup i})$ and 
$\sigma_*\circ(f_1\sqcup\dots\sqcup f_i)$ belongs to $\cB(\bld k^{\sqcup i})$, we conclude that $\alpha$ also belongs to $\cB(\bld k^{\sqcup i})$.
\end{proof}

Using the above lemmas we get an analogue of Lemma  
\ref{lem:structured-coarse-pushout-lemma} in the current setting.

\begin{lemma}\label{lem:structured-fine-pushout-lemma}
  Let $\mathcal D$ be an operad in $\mathcal S$ such that $\mathcal
  D(n)$ is cofibrant in the fine model structure on  $\cS^{\Sigma_n}$ 
  for all $n$. Suppose that the map $f$ in 
  \eqref{verification-structured-pushout} is a generating $\cA$-relative
   cofibration and that $U_k^{\mathbb D}(A)$ is cofibrant in the
  ($\Sigma_k\times \cB$)-relative model structure on 
  $(\mathcal S^{\cK})^{\Sigma_k}$ for all $k\geq 0$. Then the induced map
  \[
  U_k^{\mathbb D}(\bar f)\colon U_k^{\mathbb D}(A)\to U^{\mathbb
    D}_k(B)
  \]
  is a ($\Sigma_k\times \cB$)-relative cofibration for all $k\geq 0$.
\end{lemma}

\begin{proof}
  Applying the filtration from Proposition \ref{prop:pushout-filtration} it
  suffices to show that
\begin{equation}\label{equivariant-flat-equation}
  U_{i+k}^{\mathbb D}(A)\boxtimes_{\Sigma_i}Q_{i-1}^i(f) \to
  U_{i+k}^{\mathbb D}(A)\boxtimes_{\Sigma_i}Y^{\boxtimes i}
\end{equation}
is a $(\Sigma_k\times \cB)$-relative cofibration for all $i$ and $k$. We know from Lemma \ref{lem:iterated-pushout-product-lemma} that the iterated pushout-product 
$Q_{i-1}^i(f)\to Y^{\boxtimes i}$ is a $(\Sigma_i\times \cB)$-relative cofibration. Letting $\Sigma_k$ act trivially we may view this map as a
$(\Sigma_i\times\Sigma_k\times \cB)$-relative cofibration.  
Since $U_{i+k}^{\mathbb D}(A)$ is $(\Sigma_{i+k}\times \cB)$-relative cofibrant  by assumption, the relative version of 
Lemma \ref{lem:coarse-cofibration-properties}(ii) implies that it is also 
$(\Sigma_i\times \Sigma_k\times \cB)$-relative cofibrant. By the pushout-product axiom for the latter model structure, Proposition \ref{G-pushout-product-axiom}, it follows that (\ref{equivariant-flat-equation}) is a 
$(\Sigma_i\times\Sigma_k\times \cB)$-relative cofibration before taking 
$\Sigma_i$-orbits and by the relative version of Lemma \ref{lem:coarse-cofibration-properties}(i) the map of
$\Sigma_i$-orbits is therefore a ($\Sigma_k\times \cB$)-relative cofibration as
claimed.
\end{proof}

\begin{proof}[Proof of Proposition \ref{prop:fine-underlying-cofibrant}]
Using Lemma~\ref{lem:structured-fine-pushout-lemma} instead of 
Lemma~\ref{lem:structured-coarse-pushout-lemma}, the proof of the proposition is now completely analogous to the proof of 
Proposition~\ref{prop:coarse-underlying-cofibrant} 
\end{proof}

\section{Diagram spaces and graded spaces}
In this section $\cK$ denotes a well-structured index category as specified in Definition~\ref{def:well-structured-index}. We shall consider the corresponding projective $\cK$-model structure on $\cS^{\cK}$  from Definition~\ref{def:K-projective-model}. 
By a \emph{space graded over the classifying space $B\cK$}
we understand a space $X$ together with a map $X\to B\cK$. The purpose of this section is to relate the category of (structured) $\cK$-spaces to the category of (structured) spaces graded over $B\cK$. 

\subsection{Diagram spaces and graded spaces}
Let $\cK$ be a well-structured index category with classifying space 
$B\cK$. We write $\cS/B\cK$ for the category of spaces over $B\cK$, equipped with the standard model structure in which a map is a weak equivalence, fibration, or cofibration, if and only if the underlying map in $\cS$ is, see  \cite[Theorem 7.6.4]{Hirschhorn-model}.

\begin{theorem}\label{thm:K-spaces-over-BK}
  There is a chain of Quillen equivalences relating the projective
  $\cK$-model structure on $\cS^{\cK}$ to the standard model structure on 
  $\cS/B\cK$.
\end{theorem}

We first describe the chain of adjunctions in the theorem. Given an
object $\bld{k}$ in $\cK$, we write $(\cK\downarrow \bld{k})$ for the
comma category of objects over $\mathbf k$, and we write $E\cK$ for the
$\cK$-space $B(\cK\downarrow -)$. There is a pair of adjoint functors
\begin{equation}\label{eq:EK-adjunction}
\xymatrix{
\cS^{\cK}/E\cK \ar@<0.5ex>[r] &
\cS^{\cK} \ar@<0.5ex>[l] 
}
\end{equation}
induced by composition with and pullback along $E\cK \to *$.  It is
immediate from the definitions that this is a Quillen adjunction.  The
fact that $E\cK$ is levelwise contractible combined with the right properness of $\cS$ has the following implication.
\begin{lemma}\label{lem:EK-Q-equivalence}
  The adjunction \eqref{eq:EK-adjunction} is a Quillen equivalence
  with respect to the projective $\cK$-model structure
  on $\cS^{\cK}$.\qed
\end{lemma}
The obvious forgetful functors $(\cK\downarrow \mathbf k)\to\cK$ give rise to
a map of $\cK$-spaces $\pi\colon E\cK\to  B\cK$ (here we write $B\cK$ for the constant $\cK$-space $\const_{\cK}B\cK$). Evaluating the colimit over $\cK$, this induces an isomorphism of spaces $\colim_{\cK}E\cK\cong B\cK$.  
There is a chain of adjoint functors
\begin{equation}\label{eq:EKBKadjunction}
\xymatrix{
\colim_{\cK}\colon \cS^{\cK}/E\cK \ar@<0.5ex>[r] &
\cS^{\cK}/B\cK \ar@<0.5ex>[l] 
 \ar@<0.5ex>[r] & \cS/B\cK \ar@<0.5ex>[l] 
\thinspace \thinspace\colon\! \pi^*
}
\end{equation}
where in the first adjunction the left adjoint is defined by composing
with $\pi$ and the right adjoint takes a $\cK$-space $Y$ over $B\cK$
to the pullback $E\cK\times_{B\cK}Y$. The second adjunction is induced
by the usual $(\colim_{\cK},\const_{\cK})$ adjunction relating $\cS^{\cK}$
and $\cS$. It is clear that the left adjoint in the first adjunction preserves (acyclic) cofibrations and that the right adjoint of the second adjunction preserves (acyclic) fibrations. Thus, these are both Quillen adjunctions and the same therefore holds for their composition. 

\begin{lemma}\label{lem:EKBK-Q-equivalence}
  The adjunction $(\colim_{\cK},\pi^*)$ in \eqref{eq:EKBKadjunction}
  is a Quillen equivalence with respect to the projective
  $\cK$-model structure on $\cS^{\cK}/E\cK$ and the standard model structure on  $\cS/B\cK$.
\end{lemma}
\begin{proof}
  We use the criterion of~\cite[Corollary 1.3.16]{Hovey-model} and must check that the left adjoint reflects weak equivalences between cofibrant objects and that the derived counit of the adjunction is a weak equivalence. Here the first condition is an immediate  consequence of 
  Lemma~\ref{lem:projective-hocolim}.
That the derived counit is an equivalence means that the composition in the upper row of the diagram 
\[
\xymatrix@-1pc{
\colim_{\cK}(E\cK\times_{B\cK}Y)^{\textrm{cof}} 
\ar[r] &\colim_{\cK}(E\cK\times_{B\cK}Y) \ar[r]& Y\\
\hocolim_{\cK}(E\cK\times_{B\cK}Y)^{\textrm{cof}} 
\ar[r]^-{\simeq} \ar[u]_{\simeq}&\hocolim_{\cK}(E\cK\times_{B\cK}Y) \ar[r]\ar[u]& 
\hocolim_{\cK}Y\ar[u]
}
\]  
is a weak equivalence of spaces over $B\cK$ for any fibration $Y\to B\cK$ in $\cS$. Here $(-)^{\textrm{cof}}$ denotes cofibrant replacement in the projective $\cK$-model structure and we again write $B\cK$ and $Y$ for the corresponding constant $\cK$-spaces. Using Lemma \ref{lem:projective-hocolim} again, we see that it suffices to show that the composition in the upper row of the diagram
\[
\xymatrix@-1pc{
\hocolim_{\cK}E\cK\times _{B\cK}Y \ar[r] \ar[d]& \hocolim_{\cK}Y
\ar[r]^-{\textrm{pr}} \ar[d] & \ar[d] Y\\
\hocolim_{\cK}E\cK \ar[r] & \hocolim_{\cK}B\cK \ar[r]^-{\textrm{pr}} & B\cK
}
\]
is a weak equivalence. We know from Lemma \ref{lem:hocolim-pullback}
that homotopy
colimits preserve pullbacks; hence the left hand square is a pullback
diagram. Since the right hand square is clearly a pullback diagram it
follows that the outer square is a pullback diagram as well. By right
properness of the model structure on $\cS$ we have thereby reduced the
problem to showing that the composite map in the bottom row of the
diagram is a weak homotopy equivalence. This composition can be
identified with the canonical map
\[
\hocolim_{\cK}E\cK\to \colim_{\cK}E\cK=B\cK
\]
and the result follows from Lemma \ref{lem:projective-hocolim} and 
the fact that $E\cK$ is cofibrant by 
Proposition~\ref{prop:latching-characterization} 
(see also~\cite[Proposition 14.8.9]{Hirschhorn-model}).
\end{proof}

\noindent\textit{Proof of Theorem \ref{thm:K-spaces-over-BK}.}
  Combining Lemmas \ref{lem:EK-Q-equivalence} and  \ref{lem:EKBK-Q-equivalence} we get the required chain of Quillen
  equivalences
\[
\xymatrix{
\cS^{\cK} \ar@<-0.5ex>[r] &
\cS^{\cK}/E\cK \ar@<-0.5ex>[l]
 \ar@<0.5ex>[r]^-{\colim_{\cK}}  & \cS/B\cK \ar@<0.5ex>[l]^-{\pi^*} 
}
\eqno\qed
\]

\subsection{Structured diagram spaces and graded spaces}
\label{subs:structured-comma}
We begin with some elementary remarks on adjoint functors between categories of algebras.  Consider in general a pair of symmetric monoidal categories $(\cA,\otimes,1_{\cA})$ and $(\cB,\Box,1_{\cB})$, and 
a strong symmetric monoidal functor $V\colon \cA\to\cB$. Suppose that 
$V$ is the left adjoint of an adjoint functor pair $V\colon\cA\rightleftarrows \cB \thinspace\colon\! U$. Then the right adjoint $U$ inherits the structure of a (lax) symmetric monoidal functor with monoidal structure maps
\[
U(B)\otimes U(B')\to UV(U(B)\otimes U(B'))\xl{\cong} U(VU(B)\Box VU(B'))
\to U(B\Box B').
\]
Here the first and last arrows are induced respectively by the unit and counit of the adjunction. Similarly, the monoidal unit of $U$ is inherited from the monoidal unit of V using the counit of the adjunction, 
\[
1_{\cB}\to UV(1_{\cB})\xl{\cong}U(1_{\cA}). 
\]
With this definition it is clear that the unit and counit of the adjunction are monoidal natural transformations.  

Now let $\cD$ be an operad in $\cS$ and suppose we are given a strong symmetric monoidal functor 
$F\colon(\cS,\times,*)\to(\cA,\otimes 1_{\cA})$. As we explain in Appendix 
\ref{app:analysis-operad-algebras}, the operad $\cD$ then gives rise to a monad $\mathbb D$ on $\cA$. We write $\cA[\mathbb D]$ for the category of $\mathbb D$-algebras in $\cA$. Using the strong symmetric monoidal composition $V\circ F$ we also get a monad $\mathbb D$ on $\cB$ with a corresponding category of algebras $\cB[\mathbb D]$. The symmetric monoidal structure of $V$ gives rise to natural transformations $V\to \mathbb DV$ and $\mathbb DV\to V\mathbb D$, and consequently $V$ takes 
$\mathbb D$-algebras in $\cA$ to $\mathbb D$-algebras in $\cB$. Similarly, using the unit of the adjunction, the (lax) monoidal structure of $U$ gives rise to natural transformations $U\to \mathbb D U$ and $\mathbb D U\to U\mathbb D $. Using this, $V$ and $U$ lift to an adjoint pair of functors
\[
\xymatrix{
V\colon \cA[\mathbb D] \ar@<0.5ex>[r] &
\cB[\mathbb D] \ar@<0.5ex>[l] \thinspace\colon \! U.
}
\] 

Now we specialize to the adjunction
\[
\xymatrix{
\colim_{\cK} \colon \cS^{\cK} \ar@<0.5ex>[r] &
\cS \thinspace \colon   \! \! \const_{\cK} \ar@<0.5ex>[l] 
}
\] 
where $\cK$ denotes a well-structured index category. Let 
$\mathcal D$ be an operad in $\cS$ and let us again write $\mathbb D$ both for the associated monad on $\cS$ and the associated monad on $\cS^{\cK}$. Using that the functor $\colim_{\cK}$ is strong symmetric monoidal, the above discussion gives a pair of adjoint functors
\begin{equation}\label{eq:structured-colim-const}
\xymatrix{ \colim_{\cK} \colon \cS^{\cK}[\mathbb
  D] \ar@<0.5ex>[r] & \cS[\mathbb D] \thinspace \colon \! \!
  \const_{\cK} \ar@<0.5ex>[l] }.
\end{equation}
The next result is a structured version of Proposition \ref{prop:colim-const-Q-adjunction}.
\begin{proposition}\label{prop:structured-colim-const}
Let $\cD$ be a $\Sigma$-free operad in $\cS$ and let $\cS^{\cK}[\mathbb D]$ be equipped with the projective $\cK$-model structure and $\cS[\mathbb D]$ with the standard model structure. Then the Quillen adjunction 
\eqref{eq:structured-colim-const} is a Quillen equivalence if and only if $B\cK$ is contractible. 
\end{proposition}
\begin{proof}
First we reduce to the case where the spaces $\cD(n)$ of the operad are cofibrant in the coarse model structure on $\cS^{\Sigma_n}$ for all $n$. This is automatic in the simplicial setting and in the topological setting we may replace $\cD$ by the geometric realization of its singular complex
$\bar\cD=|\Sing \cD|$ which has this property. Indeed, the canonical map 
$\bar\cD\to \cD$ is a weak equivalence of operads so that by 
Proposition~\ref{prop:operad-change} it suffices to prove the result for 
$\bar\cD$.  With this assumption on $\cD$, 
Corollary~\ref{cor:projective-underlying-cofibrant} implies that the underlying 
$\cK$-space of a cofibrant $\mathbb D$-algebra is cofibrant in the projective model structure on $\cS^{\cK}$. From here the argument proceeds exactly as in the proof of the non-structured statement in 
Proposition~\ref{prop:colim-const-Q-adjunction}.
\end{proof}

Combining the above proposition with 
Corollary~\ref{cor:positive-projective-rectification} allows us to rectify $E_{\infty}$ spaces to strictly commutative $\cK$-space monoids provided that $B\cK$ is contractible. 

\begin{theorem}\label{theorem:K-E-infinity-rectification}
Let $\cD$ be an $E_{\infty}$ operad and let $\cK$ be a well-structured index category with contractible classifying space. Suppose that $\cK_+\to\cK$ is homotopy cofinal and that $(\cK,O\cK_+)$ is very well-structured. Then the positive projective $\cK$-model structure on $\cC\cS^{\cK}$ is related to the standard model structure on $\cS[\mathbb D]$ by a chain of Quillen equivalences.\qed
\end{theorem}

In order to investigate the situation when $B\cK$ is not contractible we specialize further and assume from now on that the
symmetric monoidal category $\cK$ is in fact permutative.
The reason for this is the following
lemma whose proof is analogous to the proof for the special case
$\cK=\cI$ considered in~\cite[Lemma~6.7]{Schlichtkrull-Thom_symm}. Let
$\mathcal E$ be the Barratt-Eccles operad, i.e., the operad whose
$k$th space is the classifying space of the translation category of
$\Sigma_k$. We write $\mathbb E$ for the associated monad.

\begin{lemma}\label{lem:Barratt-Eccles-action} 
  If $\cK$ is permutative, then the $\cK$-spaces $E\cK$ and $B\cK$ have canonical $\mathbb E$-algebra structures such that $\pi\colon E\cK\to B\cK$
  is a map of $\mathbb E$-algebras. \qed
\end{lemma}

Here we again view $B\cK$ as a constant $\cK$-space. By an operad augmented over $\cE$ we understand an operad $\cD$ equipped with a map of operads $\cD\to \cE$. In the following we assume that $\cD$ is $\Sigma$-free such that Proposition \ref{prop:structured-lift-proposition} provides a projective $\cK$-model structure on $\cS^{\cK}[\mathbb D]$.
It is easy to see that the adjunction \eqref{eq:EK-adjunction}
lifts to an adjunction of structured diagram spaces. 
\begin{lemma}\label{lem:structured-EK-Q-equivalence}
 Let $\cD$ be a $\Sigma$-free operad augmented over $\cE$. Then the adjoint functors $\cS^{\cK}[\mathbb D]/E\cK
  \rightleftarrows\cS^{\cK}[\mathbb D]$ form a Quillen equivalence
  with respect to the projective $\cK$-model structures. \qed
\end{lemma}

With $\mathbb D$ as above we also have a  structured version of the adjunctions in \eqref{eq:EKBKadjunction},
\begin{equation}\label{eq:structured-EKBKadjunction}
\xymatrix{
\colim_{\cK}\colon \cS^{\cK}[\mathbb D]/E\cK \ar@<0.5ex>[r] &
\cS^{\cK}[\mathbb D]/B\cK \ar@<0.5ex>[l] 
 \ar@<0.5ex>[r] & \cS[\mathbb D]/B\cK \ar@<0.5ex>[l]
\thinspace \thinspace\colon\! \pi^*.
}
\end{equation}
When $\cS^{\cK}[\mathbb D]$ is equipped with the projective
$\cK$-model structure and $\cS[\mathbb D]$ with the standard model structure discussed in Remark~\ref{rem:model-str-D-spaces}, the arguments for the adjunctions in \eqref{eq:EKBKadjunction} also apply here to show that the composition is a Quillen adjunction.

\begin{lemma}\label{lem:structured-EKBK-Q-equivalence}
The adjunction $(\colim_{\cK},\pi^*)$ in \eqref{eq:structured-EKBKadjunction}
is a Quillen equivalence with respect to the projective
$\cK$-model structure on $\cS^{\cK}[\mathbb D]/E\cK$ and the standard model structure on $\cS[\mathbb D]/B\cK$.
\end{lemma}

\begin{proof}
Arguing as in the proof of Proposition~\ref{prop:structured-colim-const} we first reduce to the case where the spaces $\cD(n)$ of the operad are cofibrant in the coarse model structure on $\cS^{\Sigma_n}$ for all $n$. 
Using Corollary~\ref{cor:projective-underlying-cofibrant}, the argument then proceeds exactly as in the proof of Lemma  \ref{lem:EKBK-Q-equivalence}.
\end{proof}

Combining Lemmas \ref{lem:structured-EK-Q-equivalence} and
\ref{lem:structured-EKBK-Q-equivalence} we get a
structured version of Theorem~\ref{thm:K-spaces-over-BK}.

\begin{theorem}\label{thm:structured-K-spaces-over-BK}
  Let $\cK$ be a well-structured index category that is permutative as a
  symmetric monoidal category and let $\cD$ be a $\Sigma$-free operad augmented over $\cE$. Then there is a chain of Quillen equivalences 
 \[
\xymatrix{
\cS^{\cK}[\mathbb D] \ar@<-0.5ex>[r] &
\cS^{\cK}[\mathbb D]/E\cK \ar@<-0.5ex>[l]
 \ar@<0.5ex>[r]^-{\colim_{\cK}}  & \cS[\mathbb D]/B\cK \ar@<0.5ex>[l]^-{\pi^*} 
}
\] 
relating the projective $\cK$-model structure on $\cS^{\cK}[\mathbb D]$ to the standard model structure on $\cS[\mathbb D]/B\cK$.\qed
\end{theorem}
For instance, this theorem applies to the associativity operad in which case it relates the corresponding category of (not necessary commutative) 
$\cK$-space monoids to the category of monoids (in $\cS$) over $B\cK$.  Applying the theorem to the operad $\cE$ itself and combining the result with Corollary~\ref{cor:positive-projective-rectification}, we finally get the next theorem where $\mathbb E$ again denotes the monad associated to $\cE$. 

\begin{theorem}\label{thm:comm-graded-theorem}
Let $\cK$ be a well-structured index category whose underlying symmetric monoidal category is permutative. Suppose that $\cK_+\to \cK$ is homotopy cofinal and that $(\cK,O\cK_+)$ is very well-structured. Then there is a chain of Quillen equivalences relating the positive projective $\cK$-model structure on $\cC\cS^{\cK}$ to the standard model structure on $\cS[\mathbb E]/B\cK$.  \qed
\end{theorem}

\section{Diagram spaces and symmetric spectra}
\label{sec:diagram-spaces-symmetric-spectra}
Let $\cK$ be a small symmetric monoidal category and assume that we are given a strong symmetric monoidal functor $H \colon \cK^{\op} \to \Spsym$. The examples to keep in mind are the functors $F_{-}(S^{-})$ in 
Lemmas~\ref{lem:I-and-free-spsym} and \ref{lem:J-and-free-spsym}. In general, such a functor $H$ gives rise to an adjunction
\begin{equation}\label{eq:k-space-spsym-adj}
\mS^{\cK}[-] \colon \cS^{\cK} \rightleftarrows \Spsym \colon
\Omega^{\cK}
\end{equation}
relating the categories of $\cK$-spaces and symmetric spectra. The left adjoint takes a $\cK$-space $X$ to the coend 
\[\mS^{\cK}[X]=\int^{\bld k\in \cK}H(\bld k)\wedge X(\bld k)_+
\] 
of the diagram $H \sm X_+ \colon \cK^{\op} \times\cK \to \Spsym$, 
where $X(\bld k)_+$ denotes the union of $X(\bld k)$ with a disjoint base point. The right adjoint takes a symmetric spectrum $E$ to the $\cK$-space 
$\Omega^{\cK}(E)$ defined by 
\[
\OmegaK(E)(\bld{k}) =\Map_{\Spsym}(H(\bld{k}),E).
\]
We recall that $\Spsym$ is enriched over based spaces with  
$\Map_{\Spsym}(D,E)$ defined as the appropriate subspace of the product of the mapping spaces $\Map(D(n),E(n))$. 
The next lemma is a consequence of $H$ being strong symmetric monoidal.
\begin{lemma}\label{lem:S_J_Omega_J-monoidal} 
The functor $\mS^{\cK}[-]$ is strong symmetric monoidal and $\OmegaK$ is lax symmetric monoidal.\qed
\end{lemma}

Now let $\cD$ be an operad in $\cS$ and let us write $\mathbb D$ both for the associated monad on $\cS^{\cK}$ and the associated monad on $\Spsym$.
By the discussion at the beginning of Section \ref{subs:structured-comma}, the fact that $\mathbb S^{\cK}[-]$ is strong symmetric monoidal implies that the adjunction \eqref{eq:k-space-spsym-adj} lifts to the corresponding categories of $\mathbb D$-algebras.

\begin{proposition}\label{prop:structured-adjunction-SK-Spsym}
Let $\cD$ be an operad in $\cS$. Then  the adjunction $(\mS^{\cK}[-], \Omega^{\cK})$ lifts to an adjunction relating the associated categories of 
$\mathbb D$-algebras,
\[
\mS^{\cK}[-] \colon \cS^{\cK}[\mathbb D] \rightleftarrows \Spsym[\mathbb D] \colon\Omega^{\cK}.
\eqno\qed
\]
\end{proposition}
This applies in particular to give an adjunction between the categories of (commutative) $\cK$-space monoids and (commutative) symmetric ring spectra. 

It remains to analyze the homotopical properties of these adjunctions. For this we need the next two lemmas.

\begin{lemma}\label{lem:mS_of_FK} 
There are natural isomorphisms 
\[
\mS^{\cK}[F^{\cK}_{\bld{k}}(K)] \iso  H(\bld{k})\sm K_+
\qquad \text{and} \qquad
\mS^{\cK}[G^{\cK}_{\bld{k}}(L)] \iso H(\bld{k})\sm_{\cK(\bld{k})} L_+ 
\]
for any space $K$ and any $\cK(\bld k)$-space $L$. 
\end{lemma}
\begin{proof}
By the uniqueness of adjoint functors, the second isomorphism follows from the chain of natural isomorphisms
\[ 
\cS^{\cK(\bld{k})}(L,\Ev^{\cK}_{\bld{k}} \Omega^{\cK}(E))
\iso \cS^{\cK(\bld{k})}(L, \Map_{\Spsym}(H(\bld{k}), E)) \iso
\Spsym(H(\bld{k})\sm_{\cK(\bld{k})} L_+, E).\] The first isomorphism is a consequence of the second.
\end{proof}

Given a space $K$, a morphism $\alpha\colon\bld{k}\to\bld{l}$ in $\cK$ induces a map of $\cK$-spaces $ \alpha^* \colon F^{\cK}_{\bld{l}}(K) \to
F^{\cK}_{\bld{k}}(K)$. 
\begin{lemma}\label{lem:ms_of_maps_of_FK}
 Via the first isomorphism of Lemma \ref{lem:mS_of_FK}, the induced map
$\mS^{\cK}[\alpha^*]$ corresponds to 
\[
H(\alpha)\sm K_+\colon H(\bld{l})\sm K_+ \to H(\bld{k})\sm K_+.
\eqno\qed
\]
\end{lemma}

As for the category of $\cK$-spaces, there are several useful model structures on the category of symmetric spectra. Most importantly, we have the (positive) projective model structures from \cite{HSS} and \cite{MMSS} and the (positive) flat model structures from \cite{HSS} and 
\cite{Shipley-convenient} (where they are called $S$-model structures); see also \cite{Schwede-SymSp}. In the next proposition we assume that 
$\Spsym$ is equipped with an $\cS$-model structure whose weak equivalences are the stable equivalences and we provide necessary and sufficient conditions on the functor $H$ for the $(\mS^{\cK}[-], \Omega^{\cK})$ adjunction to be a Quillen adjunction.

\begin{proposition}\label{prop:K-spaces-Spsym-adjunction}
Let $(\cK,\cA)$ be a well-structured relative index category and let $\cS^{\cK}$ be equipped with the $\cA$-relative $\cK$-model structure. Let $\Spsym$ be equipped with an $\cS$-model structure whose weak equivalences are the stable equivalences. Then the $(\mS^{\cK}[-], \Omega^{\cK})$ adjunction \eqref{eq:k-space-spsym-adj} is a Quillen adjunction if and only if 
\begin{enumerate}[(i)]
\item 
$H(\bld k)/K$ is cofibrant in the given model structure on $\Spsym$ for all objects $\bld k$ in $\cA$ and all subgroups $K\subseteq \cA(\bld k)$,
\item
the canonical map $H(\bld k)_{hK}\to H(\bld k)/K$ is a stable equivalence for all objects $\bld k$ in $\cA$ and all subgroups $K\subseteq \cA(\bld k)$, and
\item
the induced map $H(\alpha)\colon H(\bld l)\to H(\bld k)$ is a stable equivalence for all morphisms $\alpha\colon \bld k\to \bld l$ in $\cK_{\cA}$. 
\end{enumerate}
\end{proposition}

\begin{proof}
We show that the assumptions in the proposition are sufficient for the adjunction to be a Quillen adjunction. Similar arguments show that the assumptions are also necessary. Using the criterion in \cite[Lemma 2.1.20]{Hovey-model}, it suffices to prove that $\mathbb S^{\cK}[-]$ takes the generating cofibrations for $\cS^{\cK}$ to cofibrations in $\Spsym$ and the generating acyclic cofibrations for 
$\cS^{\cK}$ to stable equivalences. In the following $\bld k$ denotes an object in $\cA$, $K$ is a subgroup of $\cA(\bld k)$, and $i$ is a generating cofibration for $\cS$. By Lemma \ref{lem:mS_of_FK}, $\mathbb S^{\cK}[-]$ takes a generating cofibration of the form 
$G_{\bld k}^{\cK}(\cK(\bld k)/K\times i)$ to the map $H(\bld k)/K\wedge i_+$, up to isomorphism. The latter is a cofibration in $\Spsym$ by assumption (i) and the requirement that the model structure on $\Spsym$ be an $\cS$-model structure. If we replace $i$ with a generating acyclic cofibration for $\cS$, a similar argument gives that $\mathbb S^{\cK}[-]$ takes the resulting generating acyclic cofibration to a stable equivalence. Consider then a generating acyclic cofibration $G_{\bld k}^{\cK}(j_K\Box i)$ where $j_K$ is as in 
\eqref{eq:pi_H-factorization}. The functors $\mathbb S^{\cK}[-]$ and $G_{\bld k}^{\cK}$ both preserve tensors and colimits so $\mathbb S^{\cK}[G_{\bld k}^{\cK}(j_K\Box i)]$ can be identified with the pushout-product 
$\mathbb S^{\cK}[G_{\bld k}^{\cK}(j_K)]\Box i$ in $\Spsym$. As in Section \ref{subs:G-space-model}, let $\pi_K$ denote the projection 
$\cK(\bld k)\times_K EK\to \cK(\bld k)/K$, and observe that by 
Lemma~\ref{lem:mS_of_FK} we can identify $\mathbb S^{\cK}[G_{\bld k}^{\cK}(\pi_K)]$ with the canonical map $H(\bld k)_{hK}\to H(\bld k)/K$ which is a stable equivalence by assumption (ii). Since $\mathbb S^{\cK}[-]$ and 
$G_{\bld k}^{\cK}$ preserve mapping cylinders, this implies that 
$\mathbb S^{\cK}[G_{\bld k}^{\cK}(j_K)]$ is an acyclic cofibration which by the $\cS$-model structure on $\Spsym$ shows that 
$\mathbb S^{\cK}[G_{\bld k}^{\cK}(j_K)]\Box i$ is a stable equivalence.
Finally, consider a generating acyclic cofibration 
$G_{\bld k}^{\cK}(j_{\alpha}\Box i)$ where $\alpha$ is a morphism in 
$\cK_{\cA}$ and $j_{\alpha}$ is as in \eqref{eq:alpha-star-fact}. Arguing as above, it suffices to show that $\mathbb S^{\cK}[G_{\bld k}^{\cK}(j_{\alpha})]$ is a stable equivalence and using Lemma \ref{lem:ms_of_maps_of_FK} this can easily deduced from assumption (iii).
\end{proof}

\begin{remark}
Let $\cD$ be an operad in $\cS$. It is a formal consequence of the above proposition that whenever the model structures under consideration can be lifted to the categories of $\mathbb D$-algebras in $\cS^{\cK}$ and $\Spsym$, then the stated assumptions on $H$ imply that the induced adjunction of 
$\mathbb D$-algebras in Proposition \ref{prop:structured-adjunction-SK-Spsym} is a Quillen adjunction.
\end{remark}

\appendix

\section{Analysis of operad algebras}\label{app:analysis-operad-algebras}
We here set up the machinery needed to establish the structured pushout filtrations in Proposition \ref{prop:pushout-filtration}.
Most of the constructions are of a general category theoretical nature
and we shall formulate them in a level of generality which makes this
clear. Our primary sources for results about operads and their algebras
are ~\cite{EKMM,Elmendorf-M_infinite-loop,May_geometry}. 

\subsection{Operad actions in a closed symmetric monoidal category} 
Consider in general a cocomplete closed symmetric monoidal category
$(\mathcal A,\otimes,1_{\mathcal A})$ and suppose that we are given a
strong symmetric monoidal functor $F\colon(\cS,\times,*)\to (\mathcal
A,\otimes,1_{\mathcal A})$ where, as usual, the category of spaces
$\cS$ is as specified in Section \ref{sec:preliminaries}. Given a
space $X$ and an object $A$ in $\mathcal A$, we write $X\times A$ for
the monoidal product $F(X)\otimes A$. By the definition of a
strong symmetric monoidal functor, there are canonical isomorphisms
\[
(X\times Y)\times A\cong X\times(Y\times A)
\]
and 
\[
X\times(A\otimes B)\cong (X\times A)\otimes B\cong A\otimes(X\times B)
\]
for all spaces $X$ and $Y$ and all objects $A$ and $B$ in $\mathcal
A$. When $\mathcal A$ is the category $\cS^{\cK}$, the functor $F$
will be given by the free functor $F^{\cK}_{\mathbf 0}$ associated to the
unit $\mathbf 0$ for the monoidal structure in $\cK$ and the product
$X\times A$ will be the tensor of Lemma \ref{lem:K-spaces-S-category}.  
Let now
$\mathcal D$ be an operad in $\cS$ as defined in Section
\ref{sec:structured-section} and let $\mathbb D$ be the corresponding
monad on $\mathcal A$ defined by
\[
\mathbb D(A)=\coprod_{n=0}^{\infty}\mathcal D(n)\times_{\Sigma_n}A^{\otimes n}.
\]
Here $A^{\otimes0}$ denotes the unit $1_{\mathcal A}$ for the monoidal
structure.  We write $\mathcal A[\mathbb D]$ for the category of
$\mathbb D$-algebras in $\mathcal A$.  Many properties of monads and
their algebras can be conveniently expressed in terms of coequalizer
diagrams. For the convenience of the reader we briefly review the
relevant definitions and results.

\subsection{Review of coequalizer diagrams}\label{coequalizer-section}
Following~\cite[Section VI.6]{MacLane_working} we say that, in a
general category $\mathcal A$, a \emph{fork} is a diagram of the form
\begin{equation}
\label{fork}
\xymatrix{
A\ar@<0.5ex>[r]^{\partial _0}\ar@<-.5ex>[r]_{\partial _1}& B \ar[r]^{e} &C
}
\end{equation}
such that $e\partial_0=e\partial_1$. It is a \emph{split fork} if
there are additional morphisms $s\colon C\to B$ and $t\colon B\to A$,
such that $es=1_{C}$, $\partial_0 t=1_B$, and $\partial_1t=se$.  Such
a diagram is also called a \emph{split coequalizer diagram}. This is
motivated by the following easy lemma.
\begin{lemma}
  In a split fork as above, $e$ is a coequalizer of $\partial _0$ and
  $\partial_1$.\qed
\end{lemma}
As an example, consider a monad $\mathbb M$ on the category $\mathcal
A$ with multiplication $\mu_A\colon \mathbb M\mathbb M(A)\to \mathbb
M(A)$ and unit $\eta_A\colon A\to \mathbb M(A)$. For each $\mathbb
M$-algebra $A$ with structure map $\xi\colon \mathbb M(A)\to A$ we
have the split coequalizer
\begin{equation}\label{canonicalsplit}
\xymatrix{
\mathbb M\mathbb M(A) \ar@<.5ex>[r]^{\partial_0=\mu_A} 
\ar@<-.5ex>[r]_{\partial_1=\mathbb M(\xi)} & \mathbb M(A) \ar[r]^{\xi} &A,
}
\end{equation}
where the splittings are given by the morphisms $s=\eta_A$ and
$t=\eta_{\mathbb M(A)}$.  A fork as in (\ref{fork}) is a
\emph{reflexive coequalizer} if $e$ is the coequalizer of $\partial_0$
and $\partial_1$ and in addition there exists a morphism $h\colon B\to
A$ such that $\partial_0h=1_B$ and $\partial_1h=1_B$. As an example,
the diagram (\ref{canonicalsplit}) is also a reflexive coequalizer
with simultaneous splitting $\mathbb M(\eta_A)\colon\mathbb M(A)\to
\mathbb M\mathbb M(A)$. Reflexive coequalizers are important because
of the following easy lemma which we quote from~\cite[Lemma II
6.6]{EKMM}.

\begin{lemma}\label{reflexivelemma}
  Let $\mathbb M$ be a monad on $\mathcal A$ that preserves reflexive
  coequalizers and suppose that the fork (\ref{fork}) is a reflexive
  coequalizer such that $A$ and $B$ are $\mathbb M$-algebras and
  $\partial_0$ and $\partial_1$ are maps of $\mathbb M$-algebras. Then
  there is a unique $\mathbb M$-algebra structure on $C$ such that $e$
  is a map of $\mathbb M$-algebras, and with this structure the
  diagram is a coequalizer diagram in the category of $\mathbb
  M$-algebras. \qed
\end{lemma}

Let us now return to the case of the monad $\mathbb D$ associated to
an operad $\mathcal D$ in $\cS$. It is proved in~\cite[Proposition II
7.2]{EKMM} (the argument is attributed to M. Hopkins) that each of the
functors $A\mapsto A^{\otimes n}$ preserves reflexive coequalizers
(since we assume $\mathcal A$ to be closed). This easily implies the
following result.
\begin{lemma}\label{D-reflexive}
The monad $\mathbb D$ preserves reflexive coequalizers in $\mathcal A$.\qed
\end{lemma}

\subsection{Iterated monoidal products of pushouts}\label{iterated-pushout-section} 
The purpose of this section is to set up a kind of filtration on the
iterated monoidal product of a pushout in the closed symmetric
monoidal category $\mathcal A$. Thus, let $\mathcal P$ be the category
$1\leftarrow 0\to 2$, and consider a $\mathcal P$-diagram
\begin{equation}\label{Pdiagram}
X_1\xl{f_1}X_0\xr{f_2}X_2
\end{equation}
in $\mathcal A$.  We write $P(f_1,f_2)$ for the associated pushout
(the colimit of the diagram). The assumption that $\mathcal A$ be
closed implies that there is a canonical isomorphism
\[
\colim_{(s_1,\dots, s_n)\in \mathcal P^n}X_{s_1}\otimes\dots\otimes X_{s_n}\xr{\sim} 
P(f_1,f_2)^{\otimes n}
\]
for each $n\geq 1$, where $\mathcal P^n$ denotes the $n$-fold product
category. For each $i=0,\dots,n$, we now let $\mathcal P^n_i$ be the
full subcategory of $\mathcal P^n$ whose objects $(s_1,\dots,s_n)$
have at most $i$ components equal to $1$, and we define
$P_i^n(f_1,f_2)$ to be the object in $\mathcal A$ given by the
associated colimit,
\[
P_i^n(f_1,f_2)=\colim_{(s_1,\dots, s_n)\in \mathcal
  P_i^n}X_{s_1}\otimes\dots\otimes X_{s_n}.
\]
We think of the chain of morphisms 
\[
X_2^{\otimes n}=P^n_0(f_1,f_2)\to P_1^n(f_1,f_2)\to\dots\to P_n^n(f_1,f_2)=
P(f_1,f_2)^{\otimes n}
\]
as a filtration of $P(f_1,f_2)^{\otimes n}$ by the number of
components in $X_1$. Notice, that $P_i^n(f_1,f_2)$ has a canonical
$\Sigma_n$-action such that this is a chain of $\Sigma_n$-equivariant
maps. We also consider the category $\mathcal Q$ given by $0\to 1$ and
write $\mathcal Q^n_i$ for the full subcategory of the $n$-fold
product category whose objects $(s_1,\dots,s_n)$ have at most $i$
components equal to 1. Given a morphism $f\colon X_0\to X_1$, we write
$Q_i^n(f)$ for the colimit of the associated $\mathcal Q_i^n$-diagram
and we similarly get the filtration
\begin{equation}\label{EMfiltration}
X_0^{\otimes n}=Q_0^n(f)\to Q^n_1(f)\to\dots\to Q^n_n(f)=X_1^{\otimes n}
\end{equation}
considered by Elmendorf-Mandell~\cite[Section
12]{Elmendorf-M_infinite-loop}. The last morphism $Q_{n-1}^n(f)\to
X_1^{\otimes n}$ in the filtration can be described in terms of the
iterated pushout-product as we now explain. Recall that in general,
given morphisms $f\colon X_0\to X_1$ and $g\colon Y_0\to Y_1$, the
pushout-product is the induced map
\[
f\Box g\colon X_1\otimes Y_0\cup_{X_0\otimes Y_0} X_0\otimes Y_1\to X_1\otimes Y_1.
\]
This construction can be iterated and we write $f^{\Box n}$ for the
$n$-fold iterated pushout-product of a morphism $f$.
\begin{lemma}\label{iterated-pushout-product}
  Given a morphism $f\colon X_0\to X_1$, the domain of the $n$-fold
  iterated pushout-product can be identified with $Q_{n-1}^n(f)$ such
  that $f^{\Box n}\colon Q_{n-1}^n(f)\to X_1^{\otimes n}$ agrees with
  the last morphism in the filtration \ref{EMfiltration}.\qed
\end{lemma}

In the next lemma we again consider a diagram as in (\ref{Pdiagram}).
We remark that if $G$ is a (discrete) group, then a square diagram of
$G$-equivariant maps which is a pushout in $\mathcal A$ is the same
thing as a pushout diagram in the category of $G$-equivariant objects
in $\mathcal A$.
\begin{lemma}\label{iteratedpushoutlemma}
For $i=1,\dots,n$ there are $\Sigma_n$-equivariant pushout diagrams
\[
\xymatrix@-1pc{ \Sigma_n\times_{\Sigma_{n-i}\times
    \Sigma_{i}}X_2^{\otimes(n-i)}\otimes Q_{i-1}^i(f_1) \ar[rr] \ar[d] & &
  \Sigma_n\times_{\Sigma_{n-i}\times\Sigma_i}X_2^{\otimes (n-i)}\otimes X_1^{\otimes i} 
  \ar[d] \\
   P^n_{i-1}(f_1,f_2) \ar[rr] && P^n_i(f_1,f_2).}
\]
\end{lemma}
\begin{proof}
  Given a subset $U$ of $\mathbf n$ and a subset $V\subseteq U$ we
  introduce the notation
\[
s_j^U=
\begin{cases}
1, &\text{if $j\in U$}\\
2, &\text{if $j\notin U$}
\end{cases},\quad \text{and}\quad
t_j^V=
\begin{cases}
1,&\text{if $j\in V$}\\
0,&\text{if $j\in U-V$}\\
2, &\text{if $j\notin U$}.
\end{cases}
\]
It follows from the universal property of colimits that there are
pushout diagrams
\[
\xymatrix@-1pc{
  \displaystyle\coprod_{|U|=i}\colim_{V\varsubsetneq
    U}X_{t_1^V}\otimes\dots\otimes X_{t_n^V} \ar[d] \ar[rr]
  & & \displaystyle\coprod_{|U|=i}X_{s_1^U}\otimes\dots\otimes X_{s_n^U} \ar[d]\\
  P_{i-1}^n(f_1,f_2) \ar[rr] & & P_i^n(f_1,f_2) }
\]
where in the upper row $U$ runs through all subsets in $\mathbf n$ of
cardinality $i$.  We then obtain the pushout diagrams stated in the
lemma from the $\Sigma_n$-equivariant isomorphisms
\[
\Sigma_n\times_{\Sigma_{n-i}\times\Sigma_i}X_2^{\otimes(n-i)}\otimes
X_1^{\otimes i}\xr{\sim} \coprod_{|U|=i}X_{s_1^U}\otimes\dots\otimes
X_{s_n^U}
\]
and
\[
\Sigma_n\times_{\Sigma_{n-i}\times\Sigma_i}X_2^{\otimes(n-i)}\otimes
Q_{i-1}^i(f_1)\xr{\sim} \coprod_{|U|=i}\colim_{V\varsubsetneq
  U}X_{t_1^V}\otimes\dots\otimes X_{t_n^V}.
\]
The first isomorphism is induced by the
$\Sigma_{n-i}\times\Sigma_i$-equivariant morphism that identifies
$X_2^{\otimes(n-i)}\otimes X_1^{\otimes i}$ as the component indexed
by the subset of $\mathbf n$ given by the $i$ last elements. For the
second isomorphism we use that $\mathcal A$ is closed symmetric
monoidal to define the $\Sigma_{n-i}\times \Sigma_i$-equivariant
isomorphism
\[
X_2^{\otimes(n-i)}\otimes Q_{i-1}^i(f_1)\xr{\sim}\colim_{(r_1,\dots,r_i)\in 
\mathcal Q^i_{i-1}}X_2^{\otimes(n-i)}
\otimes X_{r_1}\otimes\dots\otimes X_{r_i}
\]
and we again include this as the component indexed by the subset of
$\mathbf n$ given by the last $i$ elements.
\end{proof}

\subsection{The functor \texorpdfstring{$U_k^{\mathbb D}$}{UkD} on \texorpdfstring{$\mathbb D$}{D}-algebras}
Let again $\mathcal D$ be an operad in $\cS$ and consider for each
$k\geq 0$ the functor $\mathbb D(-,k)$ that to an object $X$ in
$\mathcal A$ associates the $\Sigma_k$-equivariant object
\[
\mathbb D(X,k)=\coprod_{n=0}^{\infty}\mathcal D(n+k)\times_{\Sigma_n}X^{\otimes n}.
\]
Here $\Sigma_n$ acts on $\mathcal D(n+k)$ via the inclusion in
$\Sigma_{n+k}$ as the subgroup of permutations that keep the last $k$
letters fixed. For $k=0$ this is the usual monad $\mathbb D$
associated to $\mathcal D$. We shall need the following construction
from~\cite[Section 12]{Elmendorf-M_infinite-loop}.

\begin{definition}\label{def:U_k_B}
Let $A$ be  a $\mathbb D$-algebra. For each $k\geq 0$ we define  
$U_k^{\mathbb D}(A)$ to be the $\Sigma_k$-equivariant object in 
$\mathcal A$ defined by the coequalizer diagram 
\begin{equation}\label{U_k}
\xymatrix{
\mathbb D(\mathbb D(A),k)
\ar@<0.5ex>[r]^{\partial _0}\ar@<-.5ex>[r]_{\partial 1}& \mathbb D(A,k)\ar[r] &
U_k^{\mathbb D}(A).
} 
\end{equation}
Here we define $\partial_0$ by  identifying the domain with a quotient of 
\[
\coprod_{n=0}^{\infty}\big(\coprod_{j_1=0}^{\infty}\dots\coprod_{j_n=0}^{\infty}
\mathcal D(n+k)\times\mathcal D(j_1)\times\dots\times\mathcal
D(j_n)\times A^{\otimes j_1}\otimes\dots\otimes A^{\otimes j_n}\big)
\]
and mapping the component indexed by $j_1,\dots,j_n$ to the component
indexed by the sum $j_1+\dots+ j_n$ via the composite map
\[
\mathcal D(n+k)\times \prod_{s=1}^n\mathcal D(j_s)
\to\mathcal D(n+k)\times \prod_{s=1}^n\mathcal D(j_s)
\times \mathcal D(1)^k\to\mathcal D(j_1+\dots+j_n+k)
\]  
where the first arrow is the inclusion determined by the unit in
$\mathcal D(1)$ and the second arrow is the structure map of the
operad. The map $\partial_1$ is induced by the algebra structure
$\mathbb D(A)\to A$.
\end{definition}

Notice, that using the split coequalizer diagram (\ref{canonicalsplit}) we get a canonical identification of $U_0^{\mathbb D}(A)$ with $A$. For all $k$
this construction defines a functor $U_k^{\mathbb D}$ from $\mathcal
A[\mathbb D]$ to the category of $\Sigma_k$-equivariant objects in
$\mathcal A$.

\begin{example}\label{ex:commutativity-example} 
  If $\mathbb C$ denotes the monad associated to the commutativity
  operad $\mathcal C$ with $\mathcal C(n)=*$ for all $n$, then
  $U^{\mathbb C}_k(A)=A$ with trivial $\Sigma_k$-action. This
  follows from the split coequalizer diagram \eqref{canonicalsplit} by
  letting $\mathbb M=\mathbb C$.
\end{example}

\begin{lemma}\label{D(-,k)-lemma}
  If $A\rightrightarrows B\to C$ is a fork in $\mathcal A[\mathbb D]$
  which is a reflexive coequalizer in $\mathcal A$, then the induced
  diagram
  \[
  \xymatrix{ U_k^{\mathbb D}(A)\ar@<0.5ex>[r]\ar@<-.5ex>[r]&
    U_k^{\mathbb D}(B) \ar[r] & U_k^{\mathbb D}(C) }
  \]
  is a coequalizer diagram in $\mathcal A$.
\end{lemma}
\begin{proof}
  Consider the commutative diagram
  \[
\xymatrix@-1pc{
\mathbb D(\mathbb D(A),k)
\ar@<0.5ex>[rr] \ar@<-.5ex>[rr] 
\ar@<0.5ex>[d]_{\partial_0\ } \ar@<-.5ex>[d]^{\ \partial_1} & &
\mathbb D(\mathbb D(B),k)
\ar[rr]\ar@<0.5ex>[d]_{\partial_0\ } \ar@<-.5ex>[d]^{\ \partial_1}&&
\mathbb D(\mathbb D(C),k)
 \ar@<0.5ex>[d]_{\partial_0\ } \ar@<-.5ex>[d]^{\ \partial_1}\\
\mathbb D(A,k)
\ar@<0.5ex>[rr]\ar@<-.5ex>[rr]\ar[d]& &
\mathbb D(B,k)
\ar[rr] \ar[d]&&
\mathbb D(C,k)\ar[d]\\
U_k^{\mathbb D}(A)
\ar@<0.5ex>[rr]\ar@<-.5ex>[rr]& &
U_k^{\mathbb D}(B)
\ar[rr] && U_k^{\mathbb D}(C)
}
\]
where the columns are coequalizer diagrams by definition. It is proved
in~\cite[Proposition II 7.2]{EKMM} that the endofunctors $X\mapsto
X^{\otimes n}$ on $\mathcal A$ preserve reflexive coequalizers (since
we assume that $\mathcal A$ be closed) and it easily follows that the
two upper rows are coequalizer diagrams as well. The bottom row is
therefore also a coequalizer diagram.
\end{proof}

\begin{lemma}\label{freeshifted}
  For a free $\mathbb D$-algebra of the form $\mathbb D(X)$ we have
  \[U_k^{\mathbb D}(\mathbb D(X))\cong\mathbb D(X,k).\]
\end{lemma}
\begin{proof}
  This follows from the split coequalizer diagram
\[
\xymatrix{ \mathbb D(\mathbb D\mathbb D(X),k)
  \ar@<0.5ex>[r]^{\partial_0}\ar@<-.5ex>[r]_{\partial_1} & \mathbb
  D(\mathbb D(X),k)\ar[r]^{e} & \mathbb D(X,k) }
\]
where $e$ is defined as $\partial_0$ in (\ref{U_k}). The unit
$\eta_X\colon X\to\mathbb D(X)$ induces a section $s$ of $e$ and the
morphism $\mathbb D(\eta_X)\colon \mathbb D(X)\to\mathbb D \mathbb
D(X)$ induces a section $t$ of $\partial_1$ such that
$\partial_0t=se$. (Thus, the roles of $\partial_0$ and $\partial_1$
are interchanged compared to the definition of a split coequalizer in
Section \ref{coequalizer-section}.)
\end{proof}

For the next lemma we use that our operads are reduced in the sense
that $\mathcal D(0)$ is a one-point space.

\begin{lemma}\label{terminal-projection}
  Suppose that the cocomplete symmetric monoidal category $\mathcal A$
  has a terminal object $*$. Then there is a canonical $\Sigma_k$-equivariant
  map $U_k^{\mathbb D}(A)\to \mathcal D(k)\times *$ for all $k\geq 0$.
\end{lemma}

\begin{proof}
  Consider the $\Sigma_k$-equivariant map
  \[
  \mathbb D(A,k)=\coprod_{n\geq 0}\mathcal D(n+k)\times
  _{\Sigma_n}A^{\otimes n}\to \mathcal D(k)\times *
  \]
  given by the maps $A^{\otimes n}\to *$ onto the terminal object in
  $\mathcal A$ and the $\Sigma_k$-equivariant maps
\[
\mathcal D(n+k)\to \mathcal D(n+k)\times \mathcal D(0)^n\times
\mathcal D(1)^k\xr{\gamma}\mathcal D(k),\quad
c\mapsto\gamma(c,\underbrace{*,\dots,*}_n,\underbrace{1,\dots,1}_k)
\]
defined by applying the operad structure map $\gamma$ as stated (where
in the latter formula $*$ denotes the point in $\mathcal D(0)$.) It is
easy to check that the above map gives rise to a fork in $\mathcal A$,
\[
\xymatrix@-1pc{
\mathbb D(\mathbb D(A),k) \ar@<0.5ex>[r]^{\partial_0}\ar@<-.5ex>[r]_{\partial_1}
& \mathbb D(A,k)\ar[r]^{e} & \mathcal D(k)\times *
}
\]
(not a coequalizer diagram), such that there is an induced
$\Sigma_k$-equivariant map $U_k^{\mathbb D}(A)\to\mathcal D(k)\times
*$.
\end{proof}

This lemma applies in particular when $\mathcal A$ is the category of
$\cK$-spaces $\cS^{\cK}$ in which case $\mathcal D(k)\times *$ is the
constant $\cK$-space $\mathcal D(k)$.

\subsection{Analysis of pushout diagrams}
Consider a pushout diagram in $\mathcal A[\mathbb D]$ of the form
\[
\xymatrix@-1pc{
\mathbb D(X)\ar[rr]^{\mathbb D(f)}\ar[d] && \mathbb D(Y)\ar[d]\\
A\ar[rr]^{\bar f} &&B
}
\]
where $f\colon X\to Y$ is a map in $\mathcal A$ and $\mathbb D(X)\to
A$ is the map of $\mathbb D$-algebras associated to a map $X\to A$ in
$\mathcal A$.  Our objective is to express the objects $U_k^{\mathbb
  D}(B)$ in terms of the objects $U_k^{\mathbb D}(A)$ and the map $f$.
Let $Q_{i-1}^i(f)$ be defined as in Section
\ref{iterated-pushout-section}.
\begin{proposition}\label{prop:cat-pushout-filtration}
  There is a natural sequence of $\Sigma_k$-equivariant objects and morphisms in $\mathcal A$ of the form
\[
U_k^{\mathbb D}(A)=F_0 U_k^{\mathbb D}(B)\to 
F_1 U_k^{\mathbb D}(B)\to\dots\to F_iU_k^{\mathbb D}(B)\to\dots
\]
such that $\colim_i F_iU_k^{\mathbb D}(B)=U_k^{\mathbb D}(B)$, the transfinite composition of the sequence equals $U_k^{\mathbb D}(\bar f)$, and 
there are $\Sigma_k$-equivariant pushout diagrams
\[
\xymatrix@-1pc{
  U_{i+k}^{\mathbb D}(A)\otimes_{\Sigma_i}Q_{i-1}^i(f) \ar[rr]\ar[d]& &U_{i+k}^{\mathbb D}(A)\otimes_{\Sigma_i}Y^{\otimes i} \ar[d]\\
  F_{i-1} U_k^{\mathbb D}(B) \ar[rr]& & F_iU_k^{\mathbb D}(B)
}
\]
in $\cA$ for all $i\geq 1$. 
\end{proposition}
Notice, that in the case $k=0$ this gives a filtration of $\bar f$.
In order to define the terms in the filtration we first give a
convenient presentation of the objects $U_k^{\mathbb D}(B)$. Let us
write $A\cup_XY$ and $\mathbb D(A)\cup_XY$ for the pushouts (in
$\mathcal A$) of the diagrams
\[
A\xl{p} X\xr{f} Y,\qquad \text{and}\qquad \mathbb D(A)\xl{\eta}A\xl{p}X\xr{f}Y.
\] 
\begin{lemma}\label{D(-,k)presentation}
There is a coequalizer diagram in $\mathcal A$ of the form
\[
\xymatrix{
\mathbb D(\mathbb D(A)\cup_XY,k)
\ar@<0.5ex>[r]^{\partial_0}\ar@<-.5ex>[r]_{\partial_1}&
\mathbb D(A\cup_XY,k)\ar[r] &U_k^{\mathbb D}(B).
}
\]
\end{lemma}
\begin{proof}
  For $k=0$ the claim is that there is a coequalizer diagram in
  $\mathcal A$ of the form
\[
\xymatrix{
\mathbb D(\mathbb D(A)\cup_XY)\ar@<0.5ex>[r]^{\partial_0}
\ar@<-.5ex>[r]_{\partial_1} & \mathbb D(A\cup_XY) \ar[r]^-{e} & B.
}
\]
Here $\partial_0$ is the extension of the composition
\[
\mathbb D(A)\cup_XY\to \mathbb D(A)\cup_{\mathbb D(X)}\mathbb D(Y)\to
\mathbb D(A\cup_XY)
\]
to a $\mathbb D$-algebra morphism, $\partial_1$ is induced by the
algebra structure $\mathbb D(A)\to A$, and $e$ is the extension of the
canonical morphism $A\cup_XY\to B$ to a $\mathbb D$-algebra morphism.
Now it easily follows from the universal property of a pushout of
$\mathbb D$-algebras that this is a coequalizer diagram in $\mathcal
A[\mathbb D]$. Since this is a reflexive coequalizer and the monad
$\mathbb D$ preserves reflexive coequalizers by Lemma
\ref{D-reflexive} it follows formally from Lemma \ref{reflexivelemma}
that it is also a reflexive coequalizer in $\mathcal A$. In order to
obtain the result for general $k$ we apply the functor $U_k^{\mathbb
  D}$ to this coequalizer diagram. By Lemma \ref{D(-,k)-lemma} this
gives a new coequalizer diagram in $\mathcal A$ and identifying the
first two terms as in Lemma \ref{freeshifted} the result follows.
\end{proof}

It follows from the construction that the morphism $\partial_1$ in
Lemma \ref{D(-,k)presentation} is induced by the algebra structure
$\mathbb D(A)\to A$. We also wish to give an explicit description of
the morphism $\partial_0$ and for this purpose we introduce some
convenient notation. Let again $\mathcal P$ be the category
$1\leftarrow 0\to 2$ and consider the $\mathcal P$-diagram obtained by
setting $X_0=X$, $X_1=Y$, and $X_2=A$. We write $\mathcal D_2$ for the
operad $\mathcal D$ and let $\mathcal D_0=\mathcal D_1$ be the
``operad'' that is the one-point space $\{1\}$ in degree 1 and the
empty set in all other degrees (this is strictly speaking not an
operad in our sense since the 0th space is empty). Thus, writing
$\mathbb D_0$, $\mathbb D_1$, and $\mathbb D_2$ for the associated
monads on $\mathcal A$ we have that $\mathbb D_2=\mathbb D$ and that
$\mathbb D_0=\mathbb D_1$ are the identity functors. With notation as
in Section \ref{iterated-pushout-section} we identify the domain of
$\partial_0$ with a quotient of the coproduct
\[
\coprod_{n=0}^{\infty}\big(\coprod_{j_1=0}^{\infty}\dots\coprod_{j_n=0}^{\infty}
\operatornamewithlimits{colim}_{(s_1,\dots,s_n)\in\mathcal P^n} 
\mathcal D(n+k)\times\mathcal D_{s_1}(j_1)\times\dots\times\mathcal D_{s_n}(j_n)
\times X_{s_1}^{\otimes j_1}\otimes\dots\otimes X_{s_n}^{\otimes j_n}\big)
\]
and the target with a quotient of 
\[
\coprod_{n=0}^{\infty}\operatornamewithlimits{colim}_{(t_1,\dots t_n)\in \mathcal P^n}
\mathcal D(n+k)\times X_{t_1}\otimes\dots\otimes X_{t_n}. 
\]
With this notation the morphism $\partial _0$ maps the component
indexed by $(j_1,\dots,j_n)$ to the component indexed by
$j_1+\dots+j_n$ via the map
\[
\mathcal D(n+k)\times \prod_{i=1}^n\mathcal D_{s_i}(j_i)\to
\mathcal D(n+k)\times\prod_{i=1}^n\mathcal D(j_i)\times \mathcal D(1)^k\to
\mathcal D(j_1+\dots+j_n+k)
\]
where the first arrow is induced by the canonical maps $\mathcal
D_s\to\mathcal D$ together with the diagonal inclusion of the unit in
$\mathcal D(1)^k$, and the second arrow is the operad structure map.
The term in the colimit indexed by $(s_1,\dots,s_n)$ is mapped to the
term indexed by
\[
(\underbrace{s_1,\dots,s_1}_{j_1},\underbrace{s_2,\dots,s_2}_{j_2},\dots,\underbrace{s_n,\dots,s_n}_{j_n}).
\]
We are now ready to define the objects $F_iU_k^{\mathbb D}(B)$ in the
filtration. Recall from Section \ref{iterated-pushout-section} the
filtrations of $(A\cup_XY)^{\otimes n}$ and $(\mathbb
D(A)\cup_XY)^{\otimes n}$ defined by the objects $P_i^n(p,f)$ and
$P_i^n(\eta p,f)$ for $0\leq i\leq n$. We extend the definition of
these objects to all $i\geq 0$ by letting
\[
P_i^n(p,f)=(A\cup_XY)^{\otimes n}\quad\text{and}\quad
P_i^n(\eta p,f)=(\mathbb D(A)\cup_XY)^{\otimes n}
 \quad\text{for $i\geq n$}.
\] 
It easily follows from the explicit description of the morphisms
$\partial_0$ and $\partial_1$ given above that these morphisms
``restrict'' to morphisms
\[
\partial_0,\partial_1\colon
\coprod_{n=0}^{\infty}\mathcal D(n+k)\times_{\Sigma_n}P_i^n(\eta p,f)\to 
\coprod_{n=0}^{\infty}\mathcal D(n+k)\times_{\Sigma_n}P_i^n(p,f)
\] 
for all $i\geq 0$. 

\begin{definition}
The objects $F_iU_k^{\mathbb D}(B)$  are defined by the coequalizer diagrams
\[
\xymatrix{
\displaystyle\coprod_{n=0}^{\infty}\mathcal D(n+k)\times_{\Sigma_n}P_i^n(\eta p,f)
 \ar@<.5ex>[r]^{\partial_0} \ar@<-.5ex>[r]_{\partial_1} &  
\displaystyle \coprod_{n=0}^{\infty}\mathcal D(n+k)\times_{\Sigma_n}P_i^n(p,f)
  \ar[r] &F_iU_k^{\mathbb D}(B).
}
\]
\end{definition} 
 
\begin{proof}[Proof of Proposition \ref{prop:cat-pushout-filtration}]
  It is clear from the definition that $U_k^{\mathbb D}(B)$ can be
  identified with the colimit of the objects $F_iU_k^{\mathbb D}(B)$.
  In order to establish the pushout diagrams in the lemma we first
  apply Lemma \ref{iteratedpushoutlemma} to get the pushout diagrams
\[
\xymatrix@-1pc{
  \mathbb D(A,i+k)\otimes_{\Sigma_i}Q_{i-1}^i(f) \ar[d]\ar[rr] &&
  \mathbb D(A,i+k)\otimes_{\Sigma_i}Y^{\otimes i} \ar[d] \\
  \displaystyle\coprod_{n=0}^{\infty}\mathcal
  D(n+k)\times_{\Sigma_n}P_{i-1}^n(p,f) \ar[rr] &&
  \displaystyle\coprod_{n=0}^{\infty}\mathcal
  D(n+k)\times_{\Sigma_n}P^n_i(p,f).
}
\]
There are similar pushout diagrams with $A$ replaced by $\mathbb D(A)$
and these diagrams fit together to form the larger diagrams
\[
\def\objectstyle{\scriptstyle}
\def\labelstyle{\scriptstyle} 
\xymatrix@C=2pt{
&\mathbb D(A,i+k)\otimes_{\Sigma_i} Q_{i-1}^i(f)\ar[rr]\ar'[d][dd]& &
\mathbb D(A,i+k)\otimes_{\Sigma_i} Y^{\otimes i}\ar[dd]\\
\mathbb D(\mathbb D(A),i+k)\otimes_{\Sigma_i} Q_{i-1}^i(f)\ar[rr]
\ar@<-.5ex>[ur]^{\partial_0\otimes\mathrm{id}\ \ \ \ }\ar@<.5ex>[ur]_{\ \ \ \ \partial_1\otimes\mathrm{id}}\ar[dd]& &
\mathbb D(\mathbb D(A),i+k)\otimes_{\Sigma_i} Y^{\otimes i}
\ar@<-.5ex>[ur]^{\partial_0\otimes\mathrm{id}\ \ \ \ }
\ar@<.5ex>[ur]_{\ \ \ \ \partial_1\otimes\mathrm{id}} \ar[dd]&\\
&\coprod\mathcal D(n+k)\times_{\Sigma_n}P_{i-1}^n \ar '[r] [rr] & &
\coprod\mathcal D(n+k)\times_{\Sigma_n}P_{i}^n\\
\coprod\mathcal D(n+k)\times_{\Sigma_n}\bar P_{i-1}^n\ar[rr]
\ar@<-.5ex>[ur]^{\partial_0\ \ \ \ }\ar@<.5ex>[ur]_{\ \ \ \partial_1}& &
\coprod\mathcal D(n+k)\times_{\Sigma_n}\bar P_{i}^n
\ar@<-.5ex>[ur]^{\partial_0\ \ \ \ }\ar@<.5ex>[ur]_{\ \ \ \partial_1}&
}
\]
where we write $P_i^n$ for $P_i^n(p,f)$ and $\bar P_i^n$ for
$P_i^n(\eta p,f)$. Evaluating the coequalizers of these diagrams we
get the pushout diagrams in the proposition.
\end{proof}

\end{document}